\documentclass[12pt]{amsart}

\usepackage[letterpaper,margin=.6in]{geometry}
\usepackage[matrix,arrow,ps,color,line,curve,frame,all]{xy}

\usepackage{url}
\usepackage{color}

\usepackage[all]{xypic}

\usepackage{braket}
\usepackage{enumerate}

\usepackage{latexsym}

\usepackage{amssymb}
\usepackage{amsfonts}
\usepackage{amscd}
\usepackage{amsmath,amsthm}
\usepackage{verbatim}

\usepackage{imakeidx}

\usepackage{tikz}
\usetikzlibrary{matrix,arrows,decorations.pathmorphing,fit,cd}
\tikzset{myboxgroup/.style={draw, densely dotted}} 

\usepackage[colorlinks,linkcolor=blue,citecolor=black!75!red]{hyperref}
\usepackage{cleveref}

\newtheorem{lemma}{Lemma}[section]
\newtheorem{proposition}[lemma]{Proposition}
\newtheorem{theorem}[lemma]{Theorem}
\newtheorem{corollary}[lemma]{Corollary}

{

}

\theoremstyle{definition}

\newtheorem{definition}[lemma]{Definition}

\theoremstyle{remark}

\newtheorem{remark}[lemma]{Remark}

\makeatletter
\let\xx@thm\@thm
\AtBeginDocument{\let\@thm\xx@thm}
\makeatother

\crefname{section}{section}{sections}
\Crefname{section}{Section}{Sections}
\crefformat{section}{#2section~#1#3}
\Crefformat{section}{#2Section~#1#3}

\crefformat{subsection}{#2\S#1#3}
\Crefformat{subsection}{#2\S#1#3}
\crefrangeformat{subsection}{\S\S#3#1#4--#5#2#6}
\Crefrangeformat{subsection}{\S\S#3#1#4--#5#2#6}
\crefmultiformat{subsection}{\S\S#2#1#3}{ and~#2#1#3}{, #2#1#3}{ and~#2#1#3}
\Crefmultiformat{subsection}{\S\S#2#1#3}{ and~#2#1#3}{, #2#1#3}{ and~#2#1#3}
\crefrangemultiformat{subsection}{\S\S#3#1#4--#5#2#6}{ and~#3#1#4--#5#2#6}{, #3#1#4--#5#2#6}{ and~#3#1#4--#5#2#6}
\Crefrangemultiformat{subsection}{\S\S#3#1#4--#5#2#6}{ and~#3#1#4--#5#2#6}{, #3#1#4--#5#2#6}{ and~#3#1#4--#5#2#6}

\crefname{appendix}{Appendix}{Appendices}
\crefformat{appendix}{#2Appendix~#1#3}
\Crefformat{appendix}{#2Appendix~#1#3}

\crefname{definition}{Definition}{Definitions}
\crefformat{definition}{#2Definition~#1#3}
\Crefformat{definition}{#2Definition~#1#3}

\crefname{definitionnodiamond}{Definition}{Definitions}
\crefformat{definitionnodiamond}{#2Definition~#1#3}
\Crefformat{definitionnodiamond}{#2Definition~#1#3}

\crefname{example}{Example}{Examples}
\crefformat{example}{#2Example~#1#3}
\Crefformat{example}{#2Example~#1#3}

\crefname{examplenodiamond}{Example}{Examples}
\crefformat{examplenodiamond}{#2Example~#1#3}
\Crefformat{examplenodiamond}{#2Example~#1#3}

\crefname{remark}{Remark}{Remarks}
\crefformat{remark}{#2Remark~#1#3}
\Crefformat{remark}{#2Remark~#1#3}

\crefname{remarknodiamond}{Remark}{Remarks}
\crefformat{remarknodiamond}{#2Remark~#1#3}
\Crefformat{remarknodiamond}{#2Remark~#1#3}

\crefname{convention}{Convention}{Conventions}
\crefformat{convention}{#2Convention~#1#3}
\Crefformat{convention}{#2Convention~#1#3}

\crefname{notation}{Notation}{Notations}
\crefformat{notation}{#2Notation~#1#3}
\Crefformat{notation}{#2Notation~#1#3}

\crefname{notationnodiamond}{Notation}{Notations}
\crefformat{notationnodiamond}{#2Notation~#1#3}
\Crefformat{notationnodiamond}{#2Notation~#1#3}

\crefname{lemma}{Lemma}{Lemmas}
\crefformat{lemma}{#2Lemma~#1#3}
\Crefformat{lemma}{#2Lemma~#1#3}

\crefname{proposition}{Proposition}{Propositions}
\crefformat{proposition}{#2Proposition~#1#3}
\Crefformat{proposition}{#2Proposition~#1#3}

\crefname{corollary}{Corollary}{Corollaries}
\crefformat{corollary}{#2Corollary~#1#3}
\Crefformat{corollary}{#2Corollary~#1#3}

\crefname{theorem}{Theorem}{Theorems}
\crefformat{theorem}{#2Theorem~#1#3}
\Crefformat{theorem}{#2Theorem~#1#3}

\crefname{assumption}{Assumption}{Assumptions}
\crefformat{assumption}{#2Assumption~#1#3}
\Crefformat{assumption}{#2Assumption~#1#3}

\crefname{enumi}{}{}
\crefformat{enumi}{(#2#1#3)}
\Crefformat{enumi}{(#2#1#3)}

\crefname{equation}{}{}
\crefformat{equation}{(#2#1#3)}
\Crefformat{equation}{(#2#1#3)}

\crefname{align}{}{}
\crefformat{align}{(#2#1#3)}
\Crefformat{align}{(#2#1#3)}

\crefname{proofstep}{Step}{Steps}
\crefformat{proofstep}{#2Step~#1#3}
\Crefformat{proofstep}{#2Step~#1#3}

\crefname{table}{Table}{Tables}
\crefformat{table}{#2Table~#1#3}
\Crefformat{table}{#2Table~#1#3}

\makeatletter
\renewcommand{\p@enumii}{}
\renewcommand{\p@enumiii}{}
\makeatother

\numberwithin{equation}{section}
\renewcommand{\theequation}{\thesection-\arabic{equation}}

\newcommand\bC{{\mathbb C}}

\newcommand\bP{{\mathbb P}}

\newcommand\bR{{\mathbb R}}

\newcommand\bZ{{\mathbb Z}}

\def\cB{{\mathcal B}}

\def\cF{{\mathcal F}}

\def\cL{{\mathcal L}}

\newcommand\fa{{\mathfrak a}}

\def\a{\alpha}
\def\b{\beta}

\def\l{\lambda}

\def\s{\sigma}

\def\ve{\varepsilon}

\def\sfx{{\sf x}}
\def\sfy{{\sf y}}

\newcommand\End{\operatorname {End}}
\newcommand\Ext{\operatorname {Ext}}
\newcommand\Hom{\operatorname {Hom}}

\newcommand\id{\operatorname{id}}

\newcommand\im{\operatorname{im}}
\newcommand\rank{\operatorname{rank}}
\newcommand\nullity{\operatorname{nullity}}
\newcommand\sgn{\operatorname{sgn}}
\newcommand\mult{\operatorname{mult}}
\newcommand\sym{\operatorname{sym}}
\newcommand\rel{\operatorname{rel}}
\newcommand\Sym{\operatorname{Sym}}
\newcommand\Lat{\operatorname{\mathsf{Lat}}}
\newcommand\GL{\operatorname{GL}}
\newcommand\Alt{\operatorname{Alt}}
\newcommand\Grass{\operatorname{Grass}}

\renewcommand{\Im}{\operatorname{Im}}

\def\dirlim{\mathop{\vtop{\baselineskip -100pt\lineskip -1pt\lineskiplimit 0pt
\setbox0\hbox{lim}\copy0\hbox to \wd0{\rightarrowfill}}}\limits}
\def\invlim{\mathop{\vtop{\baselineskip -100pt\lineskip -1pt\lineskiplimit 0pt
\setbox0\hbox{lim}\copy0\hbox to \wd0{\leftarrowfill}}}\limits}

\def\I11{{1 \kern -0.8pt \! \mbox{l}}}
\def\mumu{{\mu\kern-4.2pt\mu}}
\def\bfmu{{\mu\kern-4.2pt\mu}}
\def\2slash{\backslash \! \backslash}

\def\hdot{{\:\raisebox{2.5pt}{\text{\circle*{2.7}}}}}

\newcommand\numberthis{\addtocounter{equation}{1}\tag{\theequation}}

\makeindex

\begin{document}

\title[R-matrices and elliptic algebras]{Elliptic R-matrices and\\
Feigin and Odesskii's elliptic algebras}

\author[Alex Chirvasitu]{Alex Chirvasitu}
\author[Ryo Kanda]{Ryo Kanda}
\author[S. Paul Smith]{S. Paul Smith}

\address[Alex Chirvasitu]{Department of Mathematics, University at
  Buffalo, Buffalo, NY 14260-2900, USA.}
\email{achirvas@buffalo.edu}

\address[Ryo Kanda]{Department of Mathematics, Graduate School of Science, Osaka City University, 3-3-138, Sugimoto, Sumiyoshi, Osaka, 558-8585, Japan.}
\email{ryo.kanda.math@gmail.com}

\address[S. Paul Smith]{Department of Mathematics, Box 354350,
  University of Washington, Seattle, WA 98195, USA.}
\email{smith@math.washington.edu}



\begin{abstract}
The algebras $Q_{n,k}(E,\tau)$ introduced by Feigin and Odesskii as generalizations of the 4-dimensional Sklyanin algebras form a family of quadratic algebras parametrized by coprime integers $n>k\ge 1$, a complex elliptic curve $E$, and a point $\tau\in E$. The main result in this paper is that $Q_{n,k}(E,\tau)$ has the same Hilbert series as the polynomial ring on $n$ variables when $\tau$ is not a torsion point. We also show that $Q_{n,k}(E,\tau)$ is a Koszul algebra, hence of global dimension $n$ when $\tau$ is not a torsion point, and, for all but countably many $\tau$, it is Artin-Schelter regular. The proofs use the fact that the space of quadratic relations defining $Q_{n,k}(E,\tau)$ is the image of an operator $R_{\tau}(\tau)$ that belongs to a family of operators $R_{\tau}(z):\mathbb{C}^n\otimes\mathbb{C}^n\to\mathbb{C}^n\otimes\mathbb{C}^n$, $z\in\mathbb{C}$, that (we will show) satisfy the quantum Yang-Baxter equation with spectral parameter.
\end{abstract}

\maketitle
\tableofcontents

\section{Introduction}
\label{se.intro}

\subsection{Feigin and Odesskii's elliptic algebras are deformations of polynomial rings}
\label{sect.intro.1}

In this paper we continue our study of the basic properties of the elliptic algebras $Q_{n,k}(E,\tau)$ defined by Feigin and Odesskii in 1989 in their papers \cite{FO-Kiev} and \cite{FO89}.  These are graded $\bC$-algebras generated by $n$ degree-one elements. 
The \textsf{Hilbert series} of  $Q_{n,k}(E,\tau)$ is the formal power series $\sum_{i=0}^\infty \dim(Q_{n,k}(E,\tau)_i) t^i$.

The main results in this paper are as follows. (The notation is explained after their statement.)

\begin{theorem}[\cref{thm.H.series}]
\label{thm.intro.H.series}
For all $\tau\in(\bC-\bigcup_{m\geq 1}\frac{1}{m}\Lambda)\cup\frac{1}{n}\Lambda$, the Hilbert series of $Q_{n,k}(E,\tau)$ is the same as that of the polynomial ring on $n$ variables placed in degree one, namely $(1-t)^{-n}$.
\end{theorem}

The quadratic dual of $Q_{n,k}(E,\tau)$ is denoted by $Q_{n,k}(E,\tau)^!$\index{Q_n,k(E,tau)^"!@$Q_{n,k}(E,\tau)^"!$}.

\begin{theorem}[\cref{thm.H.series.dual}]
For all $\tau\in (\bC-\bigcup_{m=1}^{n+1}\frac{1}{mn}\Lambda) \cup \frac{1}{n}\Lambda$, the Hilbert series of $Q_{n,k}(E,\tau)^!$ is the same as that of the exterior algebra on $n$ variables placed in degree one, namely 
$(1+t)^n$.
\end{theorem}

\begin{theorem}[\cref{thm.Koszul,th.fgl}]
For all $\tau\in(\bC-\bigcup_{m\geq 1}\frac{1}{m}\Lambda)\cup\frac{1}{n}\Lambda$, $Q_{n,k}(E,\tau)$ is a Koszul algebra whose global dimension is $n$.
\end{theorem}

\begin{theorem}[\cref{th.reg}]
For fixed $(n,k,E)$, $Q_{n,k}(E,\tau)$ is Artin-Schelter regular for all but countably many $\tau$.
\end{theorem}

\begin{proposition}
[\cref{lem.S.frob}]
For fixed $(n,k,E)$, $Q_{n,k}(E,\tau)^!$ is a Frobenius algebra for all but finitely many cosets $\tau+\Lambda$.
\end{proposition}

Although $Q_{n,k}(E,0)$ is a polynomial ring, most $Q_{n,k}(E,\tau)$'s are not  commutative.

The algebras $Q_{n,k}(E,\tau)$ depend on a pair of relatively prime integers $n>k \ge 1$\index{n@$n$}\index{k@$k$}, a point $\tau \in \bC$\index{tau@$\tau$}, and 
a complex elliptic curve $E:=\bC/\Lambda$\index{E@$E$}
where $\Lambda:=\bZ+\bZ\eta$\index{Lambda@$\Lambda$} is the lattice spanned by 1 and a point $\eta$\index{eta@$\eta$} lying in the upper half plane.
Fix a vector space $V \cong \bC^n$\index{V@$V$} with basis $x_0,\ldots,x_{n-1}$\index{x_i@$x_{i}$} indexed by the cyclic group $\bZ_n$\index{Z_n@$\bZ_{n}$}.
We fix this notation for the rest of the paper.

The algebra $Q_{n,k}(E,\tau)$\index{Q_n,k(E,tau)@$Q_{n,k}(E,\tau)$} is defined to be the quotient of the tensor algebra $TV$\index{TV@$TV$} modulo the ideal generated by 
the subspace $\rel_{n,k}(E,\tau) \subseteq V^{\otimes 2}$\index{rel_n,k(E,tau)@$\rel_{n,k}(E,\tau)$} spanned by the $n^2$ elements\index{r_ij@$r_{ij}$}
\begin{equation}
\label{the-relns-1}
	r_{ij}\; :=\; \sum_{r \in \bZ_n} \frac{\theta_{j-i+(k-1)r}(0)}{\theta_{j-i-r}(-\tau)\theta_{kr}(\tau)} \,\, x_{j-r} \otimes x_{i+r}	
\end{equation}
where the indices $i$ and $j$ belong to $\bZ_n$ and $\theta_0(z),\ldots,\theta_{n-1}(z)$ are certain theta functions of order $n$ (defined in \cite[Prop.~2.6]{CKS1} and \Cref{eq:official.theta_alpha} below), also indexed by $\bZ_n$, that are quasi-periodic with respect to $\Lambda$.

If $\tau \in \frac{1}{n}\Lambda$, then $\theta_{kr}(\tau)=0$ for some $r$ so the relations do not make sense.
Nevertheless, we can extend the definition of $Q_{n,k}(E,\tau)$  to all $\tau \in \bC$ (see \cref{sssec.def.rel.Qnk} and \cite[\S3.3]{CKS1}).

Up to isomorphism, $Q_{n,k}(E,\tau)$ depends only on the image of $\tau$ in $E$. 

\subsection{Elliptic solutions to the quantum Yang-Baxter equation}
The relations \cref{the-relns-1} defining $Q_{n,k}(E,\tau)$ come from an elliptic solution to the quantum Yang-Baxter equation; see \cref{eq:relns.QYBE} below.

For each $\tau \in \bC-\frac{1}{n}\Lambda$, and each $z \in \bC$, we define the linear operator
\begin{equation*}
  R(z)\; =\; R_{\tau}(z)\; =\; R_{n,k,\tau}(z):V^{\otimes 2} \to V^{\otimes 2}
\end{equation*}
by the formula\index{R(z), R_tau, R_n,k,tau@$R(z)$, $R_{\tau}(z)$, $R_{n,k,\tau}(z)$}
\begin{equation}\label{eq:odr}
  R(z)(x_i\otimes x_j)  \; := \; 
  \frac{\theta_0(-z) \cdots \theta_{n-1}(-z)}{\theta_1(0) \cdots \theta_{n-1}(0)}
 \sum_{r\in \bZ_n}
  \frac{\theta_{j-i+r(k-1)}(-z+\tau)}
  {\theta_{j-i-r}(-z)\theta_{kr}(\tau)} \,
  x_{j-r}\otimes x_{i+r}
\end{equation}
for all $(i,j) \in \bZ_n^2$. Since $\tau \notin \frac{1}{n}\Lambda$, the denominators $\theta_{kr}(\tau)$ are never zero;  
hence $z \mapsto R(z)$ is a holomorphic function $\bC \to \End_\bC(V^{\otimes 2})$.

The next result is the main result in \cref{se.main}. 

\begin{theorem}[\cref{th.qybe2}]
\label{thm.qybe}
For all $u,v \in \bC$, 
\begin{equation*}
R(u)_{12} R(u+v)_{23}R(v)_{12} \;=\;  R(v)_{23} R(u+v)_{12}R(u)_{23} 
\end{equation*}
as an operator on $V^{\otimes 3}$, where $R(z)_{12}:=R(z) \otimes I$, $R(z)_{23}:=I \otimes R(z)$,  and $I$ is the identity operator on $V$. 
\end{theorem}

In other words, $R(z)$ satisfies the quantum Yang-Baxter equation with spectral parameter.

Comparing the formula for $R(z)$ with the defining relations for $Q_{n,k}(E,\tau)$ in \cref{the-relns-1}, one sees that
\begin{equation}
\label{eq:relns.QYBE}
\rel_{n,k}(E,\tau) \;=\;   \operatorname{span}\{r_{ij} \; | \; i,j \in \bZ_n\} \;=\; \text{the image of $R(\tau)$}.
\end{equation}
This equality and the fact that $R(z)$ satisfies the quantum Yang-Baxter equation is used in all our main results.

It is easy to see that $R(0)=I \otimes I$ where $I$ denotes the identity operator on $V$\index{I@$I$} (see \cref{rmk.R(0)}).  
As a consequence of this and \cref{thm.qybe}, for each $z \in \bC$, $R(z)$ is either an isomorphism or satisfies $R(z)R(-z)=0$ 
(\cref{lem.comp.pm}). \cref{cor.R.isom} shows that $R_{\tau}(z)$ is a non-isomorphism if and only if $z \in \pm \, \tau +\frac{1}{n}\Lambda$. That result also shows that $\rank R_\tau(\tau+\zeta)=\tbinom{n}{2}$
for all $\zeta \in \frac{1}{n}\Lambda$.\footnote{This implies that the dimension of $\rel_{n,k}(E,\tau)$ is $\tbinom{n}{2}$ which is the first step toward proving \cref{thm.intro.H.series}.}  It follows  that
$$
\text{the image of $R_\tau(z)$} \;=\; \text{the kernel of $R_\tau(-z)$}
$$
for all $z \in\pm\tau+\frac{1}{n}\Lambda$. 
The next result is a consequence of these remarks and the calculation before \cref{th:is-kszdl}.

\begin{theorem}[\cref{th:is-kszdl}]
The quadratic dual $Q_{n,k}(E,\tau)^!$ is isomorphic to the tensor algebra $TV$ modulo the ideal generated by
the kernel of the operator $R_{n,n-k,\tau}(\tau):V^{\otimes 2} \to V^{\otimes 2} $. 
\end{theorem}

\subsubsection{}
The fact that $Q_{n,k}(E,0)$ is a polynomial ring on $n$ variables (see \cite[Prop.~5.1]{CKS1} for a proof) is related to the fact that $\lim_{\tau \to 0} R_{\tau}(\tau)$ is the anti-symmetrization operator $v \otimes v' \mapsto v \otimes v'-v' \otimes v$. The fact that $Q_{n,k}(E,0)^!$ is an exterior algebra on $n$ variables is related to the fact that $\lim_{\tau \to 0} R_{\tau}(-\tau)$ is the symmetrization operator $v \otimes v' \mapsto v \otimes v' +v' \otimes v$.  These are special cases of \cref{le.ulim0} which shows that 
$\lim_{\tau \to 0} R_\tau(m\tau)$ is the 
skew-symmetrization operator $v \otimes v' \to v \otimes v' -mv' \otimes v$ for all $m \in \bZ$. This observation is used in an essential way in the proof of \cref{thm.intro.H.series}: it is used to show that the space of degree-$d$ relations for 
$Q_{n,k}(E,\tau)$ is the kernel of a certain operator $F_d(-\tau):V^{\otimes d} \to V^{\otimes d}$.  
Like $R_\tau(\tau)$, $F_d(-\tau)$ belongs to a family of operators $F_{d}(z)$, $z\in\bC$, and \cref{pr.flim} shows that the limits of 
$F_d(-\tau)$ and $F_d(\tau)$ as $\tau \to 0$ are the symmetrization and anti-symmetrization operators on $V^{\otimes d}$, respectively. 
This gives a heuristic explanation as to why we might expect that $Q_{n,k}(E,\tau)$ and 
$Q_{n,k}(E,\tau)^!$ are deformations of the polynomial and exterior algebras, respectively.

\subsection{Methods} 
\label{sect.methods}
In the hope that the methods we use might be useful in other situations we say a little about them.
For the purposes of the discussion we write $A(\tau)=Q_{n,k}(E,\tau)$. 
Thus, $A(0)$ is the polynomial ring $SV=\bC[x_{0},\ldots,x_{n-1}]$\index{SV@$SV$}.

The main results in this paper are of the  form: $A(\tau)$ has property ${\bf P}(\tau)$, 
where ${\bf P}(0)$ is a property of $A(0)$. In all cases of interest the property ${\bf P} (\tau)$
can be formulated as a statement that a certain subspace $S(\tau) \subseteq V^{\otimes d}$ has the 
same dimension as $S(0)$.

We realize $S(\tau)$  as the image or kernel of a linear operator 
$P(\tau'): V^{\otimes d} \to V^{\otimes d}$, where $\tau'$ is usually an integer multiple of $\tau$,
and reduce the question of interest to a question about the rank of $P(\tau')$.
In all cases of interest, $P(\tau')$ belongs to a family of linear operators $P(z): V^{\otimes d} \to V^{\otimes d}$, $z \in \bC$, 
whose matrix entries (with respect to some, hence every, basis) are theta functions with respect to $\Lambda$ having the same
quasi-periodicity properties. We call such a $P(z)$ a theta operator (see \cref{ssec.theta.operator}). 
The determinant $\det P(z)$ is then a theta function, of order $r$ say, and \cref{lem.theta.fns}  tells us that 
$\det P(z)$ has $r$ zeros (counted with multiplicity) in a fundamental parallelogram (and also tells us the sum of those zeros).    
In other words, if $ \mult_p(\det P(z))$ denotes the multiplicity of $p$ as a zero of $\det P(z)$, 
$\sum_p \mult_p(\det P(z))=r$ where the sum is taken over all points in a fundamental parallelogram.
By \cref{pr.hol-op}, $\mult_p(\det P(z))  \ge  \dim(\ker P(p))$.

Often, we are able to narrow down the possibilities for the zeros of $\det P(z)$ to a finite
number of $\frac{1}{n} \Lambda$-cosets of the form $m\tau+\frac{1}{n} \Lambda$; 
see \cref{pr.det-zeros,le.hdet} for example. We then obtain for ``enough'' of those $m$'s a result of the form
$\dim (\ker P(m\tau)) \ge$ some number, $d_m$ say. It then follows that
\begin{equation*}
r \;=\; \sum_p \mult_p(\det P(z)) \; \ge \;  \sum_p  \dim(\ker P(p))  \; \ge \;  \sum_m d_m.
\end{equation*}
If the right-most sum equals $r$, then these inequalities are equalities and we conclude that we have found all the zeros of
$\det P(z)$ and their individual multiplicities.  In particular, we now know  $\dim(\ker P(\tau'))$. 

Among the operators playing the role of $P(z)$  are: 
\begin{itemize}
  \item[$\hdot$]
  $R(z)$ in  \cref{se.det} where we show that $R(z)$ is not an isomorphism if and only if $z \in \pm\tau +\frac{1}{n}\Lambda$ and 
  that $\rank R(\tau)
  =\binom{n}{2}$; the dimension of the space of quadratic relations for  $Q_{n,k}(E,\tau)$ is therefore the same as for $SV$; 
  \item[$\hdot$] 
  $F_d(z)$ and $G_{\tau}(z)$ in \cref{sec.hilb.series} where we show that the dimension of the space of degree-$d$ relations for  
  $Q_{n,k}(E,\tau)$, which is the kernel of $F_d(-\tau)$,  is the same as for $SV$; 
  \item[$\hdot$] 
  $H_\tau(z)$ in \cref{se.ksz} where we prove that a certain lattice of subspaces of $V^{\otimes d}$ is distributive by showing that 
  certain elements of it have the same dimension as their counterparts for $SV$.
\end{itemize}
The operators $G_{\tau}(z)$ and $H_\tau(z)$ are not defined on all of $V^{\otimes d}$.

\subsection{Contents of this paper}
In \cref{se.main} we show that $R(z)=R_{n,k,\tau}(z)$ satisfies the QYBE (\cref{thm.qybe}). 
The proof relies on the known fact that the operator $S(z)$ defined in \cref{eq:wii} satisfies 
the QYBE. The operator $S(z)$ is defined in terms of theta functions with characteristics; 
the relation between those and  the $\theta_\a$'s appearing in the definition of $R(z)$ 
is stated in \cref{le.factor}. \cref{prop.reln.btw.R(z).and.Tk(z)} uses 
some calculations in \cite{ric-tra} to relate $R(z)$ to $S(z)$ (and its relatives). 

\Cref{se.hol} establishes some general results about a holomorphic linear operator 
$A(z)$ on a finite-dimensional vector space and relates the location and multiplicities of the zeros of $\det A(z)$ to the 
dimension of the kernel of $A(z)$. These results are used in \cref{se.det,sec.hilb.series,se.ksz}. We also introduce the notion of a
theta operator in this section.

\Cref{se.det} takes the first step toward showing that $Q_{n,k}(E,\tau)$ has the same Hilbert series as the polynomial ring $\bC[x_0,\ldots,x_{n-1}]$ by showing that the dimension of $\rel_{n,k}(E,\tau)$ is $\tbinom{n}{2}$. This is not straightforward.   
We must understand the kernel and image of $\lim_{\tau \to 0} R(\pm\tau+\zeta)$ when $\zeta \in \frac{1}{n}\Lambda$. To do this we show that $\det R(z)$ is a theta function with respect to  $\frac{1}{n}\Lambda$; we also need to know the location and multiplicities of the 
zeros of $\det R(z)$. Odesskii already knew this but he did not prove the formula for $\det R(z)$ in his survey \cite{Od-survey}  
so we do that in \cref{pr.det-zeros} (and in doing so make a small correction to his formula); that, and a proof that the dimension of $\rel_{n,k}(E,\tau)$ is $\tbinom{n}{2}$, 
are the main results in \cref{se.det}.

In \cref{sec.hilb.series} we show that the Hilbert series of $Q_{n,k}(E,\tau)$ is $(1-t)^{-n}$ for all $\tau\in(\bC-\bigcup_{m\geq 1}\frac{1}{m}\Lambda)\cup\frac{1}{n}\Lambda$. The method is a jazzed up 
version of the method in \cref{se.det}. The space of degree-$d$ relations for $Q_{n,k}(E,\tau)$ is realized as the 
kernel of the linear operator $F_d(-\tau):V^{\otimes d} \to V^{\otimes d}$, and the proof of the Hilbert series
result requires a careful analysis  of the theta operator 
$F_d(z)$ that is similar in spirit to some of the arguments in \cref{se.det}.
The argument we use to prove the Hilbert series result bears no resemblance to earlier arguments showing that  the Hilbert series of $Q_{n,1}(E,\tau)$ is $(1-t)^{-n}$. We make some further remarks about this in \cref{sect.H.series.remark}.

In \cref{subsec.dual} we show that the Hilbert series of $Q_{n,k}(E,\tau)^!$ is $(1+t)^n$ for all $\tau\in (\bC-\bigcup_{m=1}^{n+1}\frac{1}{mn}\Lambda) \cup \frac{1}{n}\Lambda$.
The methods there resemble those in  \cref{sec.hilb.series} but now the space of degree-$d$ relations for (a quotient of $TV$ that is 
isomorphic to) $Q_{n,k}(E,\tau)^!$  is realized as the kernel of $F_d(\tau)$.

Since the space of degree-$d$ relations for $Q_{n,k}(E,\tau)$ is the kernel of $F_d(-\tau)$, there is a canonical graded vector space isomorphism
$Q_{n,k}(E,\tau) \cong \bigoplus_{d=0}^\infty \im F_d(-\tau)$. The multiplication on  $Q_{n,k}(E,\tau)$ can therefore be transferred to 
this subspace of $TV$ in a canonical way. 
\Cref{sect.mult.in.Q} gives an explicit description of this multiplication via the operators $M_{a,b}$ defined there. 
This multiplication is analogous to the shuffle product  on the subspace of the tensor algebra consisting of the symmetric tensors.

In \cref{se.ksz}, we show $Q_{n,k}(E,\tau)$ is a Koszul algebra for all $\tau\in(\bC-\bigcup_{m\geq 1}\frac{1}{m}\Lambda)\cup\frac{1}{n}\Lambda$, by verifying the ``distributive lattice'' criterion. The proof does not rely on knowing the Hilbert series of $Q_{n,k}(E,\tau)^!$ so the result provides an independent proof of the fact that the Hilbert series of $Q_{n,k}(E,\tau)^!$ is $(1+t)^n$.  The operators $M_{a,b}$ defined in \cref{sect.mult.in.Q}, and others derived from them, play a crucial role. 
Once more, we use the methods described in \cref{sect.methods}.

In \cref{se.reg} we show that, for fixed $n$, $k$, and $E$, $Q_{n,k}(E,\tau)$ is an Artin-Schelter regular algebra for all but countably many 
$\tau$.

An index of notation appears just before the bibliography.

\subsection{Acknowledgements}

A.C. acknowledges support through NSF grant DMS-1801011.

R.K. was a JSPS Overseas Research Fellow, and supported by JSPS KAKENHI Grant Numbers JP16H06337, JP17K14164, and JP20K14288, Leading Initiative for Excellent Young Researchers, MEXT, Japan, and Osaka City University Advanced Mathematical Institute (MEXT Joint Usage/Research Center on Mathematics and Theoretical Physics JPMXP0619217849). R.K. would like to express his deep gratitude to Paul Smith for his hospitality as a host researcher during R.K.'s visit to the University of Washington.

\section{Preliminaries}
\label{se.prelim}

Whenever possible, the notation in this paper is the same as that in our earlier papers \cite{CKS1,CKS2,CKS3}.  (We will advise the reader to consult those papers when necessary.)  For example, we write\index{e(z)@$e(z)$}
\begin{equation*}
e(z)\;=\; e^{2\pi i z}.
\end{equation*}

We introduced the notation $(n,k,E)$, $\eta$, $\Lambda$, $\tau \in \bC$, and $\rel_{n,k}(E,\tau)$ in \cref{sect.intro.1}.  This notation will be fixed throughout the paper.
The space of theta functions $\Theta_n(\Lambda)$\index{Theta_n(Lambda)@$\Theta_{n}(\Lambda)$} and its distinguished basis $\theta_\a$, $\a \in \bZ_n$, are defined in \cite[\S2]{CKS1}. The basic properties of the $\theta_\a$'s appear in \cite[Prop.~2.6]{CKS1}, so we pause here briefly only to recall the definition:\index{theta_alpha(z)@$\theta_{\alpha}(z)$}
\begin{equation}
  \label{eq:official.theta_alpha}
  \theta_\a(z) \;:=\; e\left(\a z +\tfrac{\a}{2n}+\tfrac{\a(\a-n)}{2n}\eta\right) \prod_{m=0}^{n-1} \theta\!\left(z +\tfrac{m}{n}+\tfrac{\a}{n}\eta\right),
\end{equation}
where $\theta$ is the order-1 theta function\index{theta(z)@$\theta(z)$}
\begin{equation*}
  \theta(z) \;:=\; \sum_{n \in \bZ} (-1)^n e\big(nz + \tfrac{1}{2}n(n-1)\eta\big). 
\end{equation*}
An explicit formula for $\theta_\a(z)$ as an infinite exponential sum can be obtained by combining \cref{le.factor} with the definition of the 
function $\vartheta \hbox{$\left[ a \atop b \right]$}(z  \, | \, \eta)$ at the start of \cref{ssect.theta.fn.char}.

\subsection{Notation for linear operators}
\label{sect.notn}

Always, $V$ denotes a complex vector space of dimension $n$ with basis $x_{i}$, $i\in\bZ_{n}$.

We will write $I$ for the identity operator on $V$.

If $A:V \to V$ is a linear operator and $1 \le i \le d$, we write $A_i$\index{A_i@$A_i$} for the operator $I^{\otimes(i-1)}\otimes  A\otimes I^{\otimes(d-i)}$ on $V^{\otimes d}$.

If $A:V^{\otimes 2} \to V^{\otimes 2}$ is a linear operator and $1\leq i\leq d-1$, we write $A_{i,i+1}$\index{A_i,i+1@$A_{i,i+1}$} for the operator 
$I^{\otimes (i-1)} \otimes A \otimes I^{\otimes (d-i-1)}$ on $V^{\otimes d}$. 

Given integers  $0 \le p \le d$ and a linear operator $A:V^{\otimes p} \to V^{\otimes p}$,
we write $A^L$ (resp., $A^R$)\index{A^L, A^R@$A^L$, $A^R$} for the operator $A\otimes I^{\otimes(d-p)}$ (resp., $I^{\otimes(d-p)}\otimes A$) on $V^{\otimes d}$ given by $A$ acting on the left-most (resp., right-most) $p$ tensorands of $V^{\otimes d}$. For a family of linear operators $A(z_{1},\ldots,z_{p})$, we write $A^{L}(z_{1},\ldots,z_{p})$ for $A(z_{1},\ldots,z_{p})^{L}$. We also write $A^{R}(z_{1},\ldots,z_{p}):=A(z_{1},\ldots,z_{p})^{R}$.

Various linear operators of the form $A(z_1,\ldots,z_p)$ will be evaluated when several of its arguments are the same. If $i \le j$ and
$z_i=\cdots=z_j=\nu$ we write $A(z_1,\ldots,z_{i-1}, \nu^{j-i+1}, z_{j+1},\ldots,z_p)$ for $A(z_1,\ldots,z_p)$.

\subsection{The quantum Yang-Baxter equation with spectral parameter}
\label{sect.qybe}
The material in this subsection is standard.  

Let $A \in \End_\bC(V \otimes V)$. We define linear operators $A_{12},A_{23},A_{13} \in \End(V^{\otimes 3})$ by $A_{12}:=A \otimes I$, $A_{23}:=I \otimes A$, where $I$ is the identity operator on $V$, and $A_{13}$ acts as the identity on the middle $V$ and as $A$ does on the first and third factors of $V \otimes V \otimes V$.

A family of linear operators $R(z) \in \End(V \otimes V)$, parametrized by $z \in \bC$, satisfies the {\sf first quantum Yang-Baxter equation} if
\begin{equation}
  \label{QYBE1}
  \tag{QYBE1}
  R(u)_{12} R(u+v)_{13}R(v)_{23} \;=\;  R(v)_{23} R(u+v)_{13}R(u)_{12} 
\end{equation}
for all $u,v\in \bC$. 
We say that $R(z)$ satisfies the {\sf second quantum Yang-Baxter equation}  if 
\begin{equation}
  \label{QYBE2}
  \tag{QYBE2}
  R(u)_{12} R(u+v)_{23}R(v)_{12} \;=\;  R(v)_{23} R(u+v)_{12}R(u)_{23} 
\end{equation}
for all $u,v\in \bC$. The family of operators $R(z)$ satisfying (\ref{QYBE1}) or (\ref{QYBE2}) is called an \textsf{R-matrix}.

Let $P \in \End(V \otimes V)$ be the linear map $P(x \otimes y)=y \otimes x$. 
If $A \in \End(V \otimes V)$ we define 
\begin{equation*}
A'\; :=\; PA \qquad \text{and} \qquad A''\;:=\;AP.
\end{equation*}  
Multiplication by $P$ provides a bijection between solutions to (\ref{QYBE1}) and (\ref{QYBE2}).

\begin{proposition}
  \label{prop.R.R'.QYBE}
  A family of operators $R(z)$ satisfies (\ref{QYBE1}) if and only if $R(z)'$ satisfies (\ref{QYBE2}) if and only if $R(z)''$ satisfies  (\ref{QYBE2}).
\end{proposition}

This is an immediate consequence of the following routine lemma.

\begin{lemma}
  Let $A,B,C \in \End(V \otimes V)$.
  Then $A_{12}B_{13}C_{23}=C_{23}B_{13}A_{12}$ if and only if 
  $A'_{12}B'_{23}C'_{12}=C'_{23}B'_{12}A'_{23}$ if and only if $A''_{12}B''_{23}C''_{12}=C''_{23}B''_{12}A''_{23}$
\end{lemma}

The next result plays a crucial role in \cref{ssec.det.r.z}.

\begin{lemma}\label{lem.comp.pm}
  Let $R(z)$, $z \in \bC$, be a family of operators satisfying (\ref{QYBE2}). 
  If $R(0)=I\otimes I$, then there are scalars $c(z)\in\bC$ such that
  \begin{equation*}
    R(z)R(-z) \; =\; R(-z)R(z)\;=\; c(z)I\otimes I.
  \end{equation*}
  In particular,  if $R(z)$ is not an isomorphism, then $R(z)R(-z)=0=R(-z)R(z)$.
\end{lemma}
\begin{proof}
  Since $R(0)=I \otimes I$, substituting $u=-v=z$ in (\ref{QYBE2}) yields
  \begin{equation}\label{eq.lem.comp.pm.part}
    R(z)_{12}R(-z)_{12} \;=\; R(-z)_{23}R(z)_{23}.
  \end{equation}
  We use the fixed basis $\{x_{i}\}_{i}$ for $V$. Applying both sides of \cref{eq.lem.comp.pm.part} to $x_{i}\otimes x_{j}\otimes x_{k}$ yields
  \begin{equation*}
    R(z)R(-z)(x_{i}\otimes x_{j})\otimes x_{k} \; = \; x_{i}\otimes R(-z)R(z)(x_{j}\otimes x_{k}).
  \end{equation*}
  Hence there is a linear map $F(u): V\to V$ such that
  \begin{equation*}
    R(z)R(-z)\otimes I \;= \; I \otimes F(z)\otimes I \;=\; I\otimes R(-z)R(z),
  \end{equation*}
  which implies $R(z)R(-z)=I\otimes F(z)$ and $R(-z)R(z)=F(z)\otimes I$. 
  The same argument for $-z$ implies $R(-z)R(z)=I\otimes F(-z)$ and $R(z)R(-z)=F(-z)\otimes I$. 
  Hence there is $c(z)\in\bC$ such that
  \begin{equation*}
    I\otimes F(z)=F(-z)\otimes I=c(z)I\otimes I.
  \end{equation*}
  Therefore $F(z)=c(z)I$.
\end{proof}

\begin{proposition}
\label{rmk.R(0)}
If $R_\tau(z)$ is the operator defined in \cref{eq:odr}, then
\begin{enumerate}
\item\label{item.prop.r.zero}
$R_\tau(0)\,=\, I \otimes I$ and 
\item\label{item.prop.r.pm}
$R_\tau(\tau)R_\tau(-\tau) \, = \, 0 \, = \, R_\tau(-\tau)R_\tau(\tau)$. 
\end{enumerate}
\end{proposition}
\begin{proof}
\cref{item.prop.r.zero}
Since the zeros of $\theta_\a(z)$ are the points in $-\frac{\a}{n}\eta + \bZ \eta + \frac{1}{n}\bZ$ (\cite[Prop.~2.6(6)]{CKS1}),
the $\theta_0(0)$ term appearing before the $\Sigma$ sign in the expression for $R_\tau(0)$ annihilates all the 
terms after the $\Sigma$ sign except the $r=j-i$ summand whose denominator, $\theta_{j-i-r}(0)$,  cancels the $\theta_0(0)$
term. Hence 
\begin{equation*}
  R_\tau(0)(x_i\otimes x_j)  \; = \; 
  x_{i}\otimes x_{j}
\end{equation*}
for all $i,j \in \bZ_n$.

\cref{item.prop.r.pm}
\Cref{th.qybe2}  below shows that the $R_\tau(z)$ defined in \cref{eq:odr} 
satisfies (\ref{QYBE2}) so, since $R_\tau(0)= I \otimes I$, the conclusion of \cref{lem.comp.pm} applies to $R_\tau(z)$.

Since $Q_{n,k}(E,\tau)$ has an infinite-dimensional cyclic module, namely a point module (see \cite[\S1.4]{CKS2}), 
$\rel_{n,k}(E,\tau) \ne V^{\otimes 2}$. But $\rel_{n,k}(E,\tau)$ is the image of $R_\tau(\tau)$ so the result
follows from \cref{lem.comp.pm}.
\end{proof}

\subsubsection{R-matrices in arbitrary algebras}\label{sssec:genr}

It will be convenient to generalize the setup for $R$-matrices and the quantum Yang-Baxter equation. Instead of operators
$R(z)$ in $\End(V \otimes V)$ we can take elements $R(z)\in S\otimes_Z S$ where $S$ is a $\bC$-algebra and 
$Z\subseteq S$ is a central subalgebra. There are obvious definitions of $R(z)_{ij}\in S\otimes_Z S\otimes_Z S$ for
$ (ij)\in \{(12),(13),(23)\}$. The equations \Cref{QYBE1,QYBE2} then acquire the obvious meanings. If $V$ is a left $S$-module, then 
the various  $R(z)_{ij}$ act on $V^{\otimes 3}$. If $S$ is a finite-dimensional $\bC$-algebra we can speak of holomorphic or meromorphic $R(z)$.

\Cref{se.main} uses this idea  with $S=\bC \Gamma$ for a finite group $\Gamma$ and $Z=\bC \Delta$ for a central subgroup $\Delta<\Gamma$.

\subsection{Theta functions in one variable}
We make frequent use of the following result. A proof of it appears in the appendix to \cite{CKS1}.  
  
\begin{lemma} 
\label{lem.theta.fns}
Assume $\Lambda =\bZ\eta_1 + \bZ\eta_2$ is a lattice in $\bC$ such that $\Im(\eta_2/\eta_1)>0$, 
and suppose $f$ is a non-constant holomorphic function on $\bC$.
If there are constants $a,b,c,d \in \bC$ such that 
\begin{align*}
f(z+\eta_1) & \;=\; e^{-2\pi i (az+b)}f(z) \qquad \hbox{and} \qquad
\\
f(z+\eta_2) & \;=\; e^{-2\pi i (cz+d)}f(z), 
\end{align*}
then
\begin{enumerate}
  \item 
 $c \eta_1 - a\eta_2 \in\bZ_{\geq 0}$, and
  \item 
  $f$ has  $c \eta_1 - a\eta_2$ zeros (counted with multiplicity) in every fundamental parallelogram for $\Lambda$, and
  \item 
the sum of those zeros is ${{1}\over{2}}(c \eta_1^2 - a\eta_2^2 ) + (c-a)\eta_1\eta_2 + b\eta_2 - d\eta_1 $ modulo $\Lambda$. 
\end{enumerate}
\end{lemma}

 \subsection{Transformation properties of $R_\tau(z)$} 
Let $S,T,N \in \GL(V)$\index{S@$S$}\index{T@$T$}\index{N@$N$}, and $P \in \GL(V^{\otimes 2})$\index{P@$P$},  be the automorphisms 
\begin{equation*}
S \cdot x_\a=e\big(\tfrac{\a}{n}\big) x_\a, \quad T \cdot x_\a=x_{\a+1}, \quad N\cdot x_\a=x_{-\a}, \quad P(u \otimes v):=v \otimes u.
\end{equation*}
The group generated by $S$ and $T$ is the Heisenberg group $H_n$ of order $n^3$, and $V$ is an irreducible representation of $H_n$ (see \cref{eq:14}).  At \cite[\S1, Rmk.~2]{FO89}, Odesskii and Feigin observed that $S$ and $T$ extend to automorphisms of $Q_{n,k}(E,\tau)$ (see \cite[Prop.~3.23]{CKS1} for the details).  We will often use the projective representation of $\frac{1}{n}\Lambda/\Lambda$ on $V$ given by
\begin{equation*}
\tfrac{a}{n} \, + \, \tfrac{b}{n}\eta \; \mapsto \; T^bS^{ka}.
\end{equation*}

Let\index{b(z)@$b(z)$} 
\begin{equation}
\label{eq:defn.b(z)}
b(z) \; :=\;   e\big(-nz+\tau + \tfrac{1}{2}  -\tfrac{n+1}{2}\eta\big).
\end{equation}

\begin{proposition}
\label{lem.transf.props.R.tau.z}
Let $k' \in \bZ$\index{k'@$k'$} be the unique integer such that $n>k'\ge 1$ and $kk'=1$ in $\bZ_n$. 
Then
\begin{align}
R_\tau\big( z+\tfrac{1}{n}\big) & \;=\; (-1)^{n-1} (I \otimes S^{-k}) \, R_\tau(z) \, (S^k \otimes I),
\numberthis  \label{eq.z+1/n}
\\
R_\tau\big( z+\tfrac{1}{n}\eta\big) & \;=\; b(z) (I \otimes T^{-1}) \, R_\tau(z) \, (T \otimes I),
\numberthis  \label{eq.z+1/n eta}
\\
R_\tau(-z) & \;=\; e(n^2z) P \, R_{-\tau}(z) \, P,
\numberthis  \label{eq.-z}
\\
 R_{\tau}(-z) &  \;=\;  e(n^2z)(N \otimes N) \, R_{-\tau}(z) \, (N \otimes N),
\numberthis  \label{eq.-tau.-z}
\\
R_{\tau+\frac{1}{n}} ( z) & \;=\;  (S \otimes I) \, R_\tau(z) \, (S^{-1} \otimes I),
\numberthis  \label{eq.tau+1/n}
\\
R_{\tau+\frac{1}{n}\eta} ( z) & \;=\;  e(z) (I \otimes T^{-k'}) \, R_\tau(z) \, (I \otimes T^{k'}).
\numberthis  \label{eq.tau+1/n.eta}
\end{align}
Furthermore, 
\begin{enumerate}
\item\label{item.prop.rank.r.tw}
$\rank R_\tau(z+\zeta)\,=\, \rank R_\tau(z)$  for all $\zeta \in \frac{1}{n}\Lambda$;
\item\label{item.prop.r.pm.tors}
$R_{\tau}(z)R_{\tau}(-z)=0=R_{\tau}(-z)R_{\tau}(z)$ for all $z \in \, \pm \, \tau +  \frac{1}{n}\Lambda$;\footnote{It follows from \cref{cor.R.isom} that 
$R(z)R(-z)=0=R(-z)R(z)$ if and only if  $z \in \, \pm \, \tau +  \frac{1}{n}\Lambda$, and that $R(z)$ is an isomorphism if 
$z \notin \, \pm \, \tau +  \frac{1}{n}\Lambda$.}
\item\label{item.prop.s.t.comm}
$S\otimes S$ and $T \otimes T$ commute with $R_\tau(z)$.\footnote{This implies that $S$ and $T$ extend to 
 automorphisms of $Q_{n,k}(E,\tau)$ (cf.,\cite[Prop.~3.23]{CKS1}).}  
\end{enumerate}
\end{proposition}
\begin{proof} 
We will use the notation $D:=\theta_1(0) \cdots \theta_{n-1}(0)$.

\noindent
{\bf Proof of \cref{eq.z+1/n}.}
Since $\theta_\a\left(z+\tfrac{1}{n}\right) =  e\left(\tfrac{\a}{n}\right)\theta_\a(z)$, by  \cite[Prop.~2.6(3)]{CKS1}, and 
\begin{equation*}
e(-\tfrac{kr}{n})  x_{j-r}\otimes x_{i+r} \;=\; e(\tfrac{ki}{n}) (I \otimes S^{-k})( x_{j-r}\otimes x_{i+r} ),
\end{equation*}
we see that $ R_\tau(z+\tfrac{1}{n})(x_i\otimes x_j)$ equals
\begin{align*}
 & 
\frac{1}{D}
\left(  \prod_{\a \in \bZ_n} \theta_\a(-z-\tfrac{1}{n})  \right)
 \sum_{r\in \bZ_n}
  \frac{\theta_{j-i+r(k-1)}(-z-\tfrac{1}{n}+\tau)}
  {\theta_{j-i-r}(-z-\tfrac{1}{n})\theta_{kr}(\tau)} \,
  x_{j-r}\otimes x_{i+r}
  \\
  & \;=\; 
e\big(-\tfrac{n-1}{2}\big) \frac{1}{D}
 \left(  \prod_{\a \in \bZ_n} \theta_\a(-z)   \right)
 \sum_{r\in \bZ_n}
  e\big(-\tfrac{kr}{n}\big)
  \frac{\theta_{j-i+r(k-1)}(-z+\tau)}
  {\theta_{j-i-r}(-z)\theta_{kr}(\tau)} \,
  x_{j-r}\otimes x_{i+r}
    \\
  & \;=\; 
(-1)^{n-1}  e(\tfrac{ki}{n}) \,
\frac{1}{D}
 \left(  \prod_{\a \in \bZ_n} \theta_\a(-z)   \right)
 \sum_{r\in \bZ_n} 
  \frac{\theta_{j-i+r(k-1)}(-z+\tau)}
  {\theta_{j-i-r}(-z)\theta_{kr}(\tau)} \,
  (I \otimes S^{-k})(  x_{j-r}\otimes x_{i+r} )
    \\
  & \;=\; 
(-1)^{n-1}  e(\tfrac{ki}{n}) \, (I \otimes S^{-k}) R_\tau(z)(x_i\otimes x_j) 
   \\
  & \;=\; 
(-1)^{n-1}  \, (I \otimes S^{-k}) R_\tau(z)(S^k \otimes I)(x_i\otimes x_j). 
  \end{align*}

\noindent
{\bf Proof of \cref{eq.z+1/n eta}.}
By \cite[Prop.~2.6(4)]{CKS1}, $\theta_\a\!\left(z-\tfrac{1}{n}\eta\right) = e\!\left(z+\tfrac{1}{2n}   - \tfrac{n+1}{2n}\eta \right) \theta_{\a-1}(z)$.

Therefore $R_\tau(z+\tfrac{1}{n}\eta)(x_i \otimes x_j)$ equals  
\begin{align*}
& \frac{1}{D} \left( \prod_{\a \in \bZ_n}  \theta_\a(-z-\tfrac{1}{n}\eta) \right)
\sum_{r\in \bZ_n}
  \frac{\theta_{j-i+r(k-1)}(-z-\tfrac{1}{n}\eta+\tau)}
  {\theta_{j-i-r}(-z-\tfrac{1}{n}\eta)\theta_{kr}(\tau)} \,
  x_{j-r}\otimes x_{i+r}
  \\
  & \;=\;
e\!\left(-nz+\tfrac{1}{2}   - \tfrac{n+1}{2}\eta \right)  \frac{1}{D} \left( \prod_{\a \in \bZ_n}  \theta_{\a-1}(-z) \right)
 \sum_{r\in \bZ_n}
 e(\tau) 
  \frac{\theta_{j-i+r(k-1)-1}(-z+\tau)}
  {\theta_{j-i-r-1}(-z)\theta_{kr}(\tau)} \,
  x_{j-r}\otimes x_{i+r}
    \\
  & \;=\;
b(z) \,  \frac{1}{D} \left( \prod_{\a \in \bZ_n}  \theta_{\a}(-z) \right)
 \sum_{r\in \bZ_n}
  \frac{\theta_{j-i+r(k-1)-1}(-z+\tau)}
  {\theta_{j-i-r-1}(-z)\theta_{kr}(\tau)} \,
(I\otimes T^{-1}) ( x_{j-r}\otimes x_{i+1+r})
   \\
  & \;=\;
b(z) \, (I\otimes T^{-1}) R_\tau(z)(T \otimes I)(x_{i} \otimes x_j).
\end{align*}

\noindent
{\bf Proof of \cref{eq.-z}.}
Since $\theta_\a(-z)=-e\big(-nz+\frac{\a}{n}\big)\theta_{-\a}(z)$ by \cite[Prop.~2.6(5)]{CKS1},  
\begin{align*}
 R_\tau(-z)(x_i\otimes x_j)  & \; = \; 
  \frac{\theta_0(z) \cdots \theta_{n-1}(z)}{\theta_1(0) \cdots \theta_{n-1}(0)} 
 \sum_{r\in \bZ_n}
  \frac{\theta_{j-i+r(k-1)}(z+\tau)}
  {\theta_{j-i-r}(z)\theta_{kr}(\tau)} \,
  x_{j-r}\otimes x_{i+r}
   \\
  & \; = \; 
  (-1)^{n}e(n^2z+\tfrac{n-1}{2})   \frac{\theta_0(-z) \cdots \theta_{n-1}(-z)}{\theta_1(0) \cdots \theta_{n-1}(0)} 
 \sum_{r\in \bZ_n}
 (-1) \frac{\theta_{i-j-r(k-1)}(-z-\tau)}
  {\theta_{i-j+r}(-z)\theta_{-kr}(-\tau)} \,
  x_{j-r}\otimes x_{i+r}
   \\
  & \; = \; 
   e(n^2z)   \frac{\theta_0(-z) \cdots \theta_{n-1}(-z)}{\theta_1(0) \cdots \theta_{n-1}(0)} 
 \sum_{r\in \bZ_n}
 \frac{\theta_{i-j+r(k-1)}(-z-\tau)}
  {\theta_{i-j-r}(-z)\theta_{kr}(-\tau)} \,
 P(x_{i-r} \otimes x_{j+r} )
   \\
  & \; = \; 
    e(n^2z) P \, R_{-\tau}(z) \, P (x_i\otimes x_j).  
\end{align*}

\noindent
{\bf Proof of \cref{eq.-tau.-z}.}
This follows from \cref{eq.-z} and
\begin{align*}
PR_{\tau}(z)P(x_{i}\otimes x_{j})
&\;=\;\frac{\theta_0(-z)\cdots\theta_{n-1}(-z)}{\theta_1(0)\cdots\theta_{n-1}(0)}
\sum_{r\in \bZ_n}
\frac{\theta_{i-j+r(k-1)}(-z+\tau)}{\theta_{i-j-r}(-z)\theta_{kr}(\tau)}\,
x_{j+r}\otimes x_{i-r}\\
&\;=\;(N\otimes N)R_{\tau}(z)(N\otimes N)(x_{i}\otimes x_{j}).
\end{align*}

\noindent
{\bf Proof of \cref{eq.tau+1/n}.}
Since $\theta_\a\left(z+\tfrac{1}{n}\right) =  e\left(\tfrac{\a}{n}\right)\theta_\a(z)$,  $R_{\tau+\frac{1}{n}}(z)(x_i \otimes x_j)$ equals  
\begin{align*}
& \frac{1}{D} \left( \prod_{\a \in \bZ_n}  \theta_\a(-z) \right)
\sum_{r\in \bZ_n}
  \frac{\theta_{j-i+r(k-1)}(-z+\tau+\tfrac{1}{n})}
  {\theta_{j-i-r}(-z)\theta_{kr}(\tau+\tfrac{1}{n})} \,
  x_{j-r}\otimes x_{i+r}
  \\
  & \;=\; 
  \frac{1}{D} \left( \prod_{\a \in \bZ_n}  \theta_\a(-z) \right)
\sum_{r\in \bZ_n}
e\big(\tfrac{j-i-r}{n}\big)  \frac{\theta_{j-i+r(k-1)}(-z+\tau)}
  {\theta_{j-i-r}(-z)\theta_{kr}(\tau)} \,
  x_{j-r}\otimes x_{i+r}
  \\
  & \;=\; 
 e\big(\tfrac{-i}{n}\big)  \frac{1}{D} \left( \prod_{\a \in \bZ_n}  \theta_\a(-z) \right)
\sum_{r\in \bZ_n}
 \frac{\theta_{j-i+r(k-1)}(-z+\tau)}
  {\theta_{j-i-r}(-z)\theta_{kr}(\tau)} \,
 (S \otimes I)( x_{j-r}\otimes x_{i+r})
 \\
  & \;=\; 
 (S \otimes I)R_\tau(z)(S^{-1}\otimes I)( x_{i}\otimes x_{j}).
\end{align*}

\noindent
{\bf Proof of \cref{eq.tau+1/n.eta}.}
Since $\theta_\a\!\left(z+\tfrac{1}{n}\eta\right) = e\!\left(-z-\tfrac{1}{2n} + \tfrac{n-1}{2n}\eta \right) \theta_{\a+1}(z)$, 
$R_{\tau+\frac{1}{n}\eta}(z)(x_i \otimes x_j)$ equals  
\begin{align*}
& \frac{1}{D} \left( \prod_{\a \in \bZ_n}  \theta_\a(-z) \right)
\sum_{r\in \bZ_n}
  \frac{\theta_{j-i+r(k-1)}(-z+\tau+\tfrac{1}{n}\eta)}
  {\theta_{j-i-r}(-z)\theta_{kr}(\tau+\tfrac{1}{n}\eta)} \,
  x_{j-r}\otimes x_{i+r}
  \\
  & \;=\; 
 e(z)  \frac{1}{D} \left( \prod_{\a \in \bZ_n}  \theta_\a(-z) \right)
\sum_{r\in \bZ_n}
 \frac{\theta_{j-i+r(k-1)+1}(-z+\tau)}
  {\theta_{j-i-r}(-z)  \theta_{kr+1}(\tau) } \,
  x_{j-r}\otimes x_{i+r}
    \\
  & \;=\; 
 e(z)  \frac{1}{D} \left( \prod_{\a \in \bZ_n}  \theta_\a(-z) \right)
\sum_{s\in \bZ_n}
 \frac{\theta_{j-i+k' +s(k-1)}(-z+\tau)}
  {\theta_{j-i+k'-s}(-z)  \theta_{ks}(\tau) } \,
  x_{j+k'-s}\otimes x_{i-k'+s}
  \quad \text{(where $s=r+k'$)}
    \\
  & \;=\; 
 e(z)  \frac{1}{D} \left( \prod_{\a \in \bZ_n}  \theta_\a(-z) \right)
\sum_{s\in \bZ_n}
 \frac{\theta_{j-i+k' +s(k-1)}(-z+\tau)}
  {\theta_{j-i+k'-s}(-z)  \theta_{ks}(\tau) } \,
 (I \otimes T^{-k'})(  x_{j+k'-s}\otimes x_{i+s})
   \\
  & \;=\;  e(z) (I \otimes T^{-k'}) \, R_\tau(z) \,(x_i \otimes x_{j+k'})
  \\
  & \;=\;  e(z) (I \otimes T^{-k'}) \, R_\tau(z) \, (I \otimes T^{k'})(x_i \otimes x_j).
\end{align*}

\cref{item.prop.rank.r.tw}
This is an immediate consequence of \cref{eq.z+1/n} and \cref{eq.z+1/n eta}.

\cref{item.prop.r.pm.tors}
We observed in \cref{rmk.R(0)} that $R_\tau(\tau)$ is  not an isomorphism. Hence $R_\tau(\tau+\zeta)$ is not an isomorphism
either. The result now follows from  \cref{lem.comp.pm}  and \cref{rmk.R(0)}\cref{item.prop.r.zero}.

\cref{item.prop.s.t.comm}
To show that $[R_\tau(z), S \otimes S]=0$ we make $V^{\otimes 2}$ a $\bZ_n$-graded vector space by setting
 $\deg(x_\a\otimes x_\b):=\a+\b$. Since the action of $R_\tau(z)$ preserves degree and the homogeneous components 
 are $S^{\otimes 2}$-eigenspaces, the actions of $R_\tau(z)$ and $S \otimes S$ on $V^{\otimes 2}$ commute with each other. 

Write $c_{i,j,r}$ for the coefficient of $x_{j-r} \otimes x_{i+r}$ in $R_\tau(z)(x_i \otimes x_j)$. Since $c_{i+1,j+1,r}=c_{i,j,r}$, 
\begin{align*}
R_\tau(z)(T \otimes T) (x_i \otimes x_j) & \;=\; R_\tau(z)(x_{i+1} \otimes x_{j+1}) 
\\
& \;=\; \sum_r c_{i+1,j+1,r} x_{j+1-r} \otimes x_{i+1+r}
\\
& \;=\; (T \otimes T) \Big(\sum_r c_{i,j,r} \, x_{j-r} \otimes x_{i+r}\Big).
\end{align*}
Hence $R_\tau(z)(T \otimes T) = (T \otimes T) R_\tau(z)$, as claimed.
\end{proof}

An induction argument using \cref{eq.z+1/n} and \cref{eq.z+1/n eta} proves the following.

\begin{corollary}
\label{cor.R.transl.props}
If $a,b \in \bZ$ and $\zeta =\tfrac{a}{n}+\tfrac{b}{n}\eta$, then 
\begin{equation*}
R_\tau(z+\zeta)  \;=\; f(z,\zeta,\tau) (I \otimes T^bS^{ka})^{-1} R_\tau(z) (T^bS^{ka} \otimes I)
\end{equation*} 
where $f(z,\zeta,\tau) = e(-bnz)e\big(b\tau +\tfrac{b+a(n-1)}{2} - \tfrac{b(n+b)}{2}\eta\big)$\index{f(z,zeta,tau)@$f(z,\zeta,\tau)$}.
\end{corollary}

\subsection{Theta functions with characteristics} 
\label{ssect.theta.fn.char}
The {\sf Jacobi theta function} with respect to $\Lambda$ is the holomorphic function\index{theta(z,eta)@$\vartheta(z\,\vert\,\eta)$}
$$
\vartheta(z \, | \, \eta)  \;  :=\;  \sum_{m \in \bZ}  e\big( mz  +  \hbox{$\frac{1}{2}$} m^2\eta\big).
$$
Clearly, $\vartheta (z+1  \, | \, \eta)= \vartheta (z  \, | \, \eta)$ and $\vartheta (z+\eta  \, | \, \eta) 
=e(-z-\tfrac{1}{2}\eta)\, \vartheta (z  \, | \, \eta)$.

For real numbers $a$ and $b$  the {\sf theta function with characteristics} $a$ and $b$ is\index{theta[ab](z,eta)@$\vartheta \hbox{$\left[ a \atop b \right]$}(z  \, \vert \, \eta)$} 
\begin{align*}
\vartheta \hbox{$\left[ a \atop b \right]$}(z  \, | \, \eta) & \;:=\;   e\big(a(z+b) + \hbox{$\frac{1}{2}$}a^2\eta \big) \vartheta(z+a\eta+b\,|\,\eta)
\\
& \;  =\;  \sum_{m \in \bZ}  e\big( (a+m)(z+b)  +  \hbox{$\frac{1}{2}$} (a+m)^2\eta\big).
\end{align*}
This is the same as the definition at \cite[(2.5)]{ric-tra}. In \cite[p.~10]{Mum07} and \cite[(3.1)]{tra}, $\vartheta \hbox{$\left[ a \atop b \right]$}(z  \, | \, \eta)$ is denoted by $\vartheta_{a,b}(z,\eta)$. The 
papers \cite{ric-tra} and \cite{tra} play a role in \cref{se.main}.

It is easy to see that
\begin{equation}
\label{eq:a+1.b+1}
\vartheta \hbox{$\left[ a+1 \atop b \right]$} (z  \, | \, \eta) \;=\; \vartheta \hbox{$\left[ a \atop b \right]$} (z  \, | \, \eta)
\qquad \text{and}  \qquad
\vartheta \hbox{$\left[ a \atop b+1 \right]$} (z  \, | \, \eta) \;=\; e(a) \vartheta \hbox{$\left[ a \atop b \right]$} (z  \, | \, \eta).
\end{equation}
Since $\vartheta(z \, | \, \eta)=0$ if and only if $z \in \frac{1}{2}(1+\eta) + \Lambda$ by \cref{lem.theta.fns}, 
$\vartheta \hbox{$\left[ a \atop b \right]$}(z  \, | \, \eta) =0$ if and only if 
\begin{equation*}
  z \, \in \, \tfrac{1}{2}(1+\eta) \,- \, (a\eta +b) \,+\, \Lambda .
\end{equation*}

\begin{proposition}
\label{prop.qp.theta.char}
If $s,t \in \bZ$, then 
$
\vartheta \hbox{$\left[ a \atop b \right]$} (z + s\eta+t  \, | \, \eta) \,=\, e\big(at-s(z+b) -\tfrac{1}{2}s^2\eta\big) \,
\vartheta \hbox{$\left[ a \atop b \right]$} (z  \, | \, \eta).
$
\end{proposition}
\begin{proof}
We observed above that $\vartheta (z+1  \, | \, \eta)= \vartheta (z  \, | \, \eta)$ and $\vartheta (z+\eta  \, | \, \eta) 
=e(-z-\tfrac{1}{2}\eta)\, \vartheta (z  \, | \, \eta)$.
An induction argument shows that $\vartheta(z+s\eta \, | \, \eta)=e(-sz-\tfrac{1}{2}s^2\eta)\vartheta(z \, | \, \eta)$ for all integers $s$
and it follows from this that 
$\vartheta(z+s\eta+t \, | \, \eta)=e(-sz-\tfrac{1}{2}s^2\eta)\vartheta(z \, | \, \eta)$ for all integers $s$ and $t$.
Hence
\begin{align*}
\vartheta \hbox{$\left[ a \atop b \right]$} (z + s\eta+t  \, | \, \eta)
& \;=\;  e\big(a(z+s\eta+t+b) + \hbox{$\frac{1}{2}$}a^2\eta \big) \vartheta(z+s\eta+t+a\eta+b\,|\,\eta)
\\
& \;=\;  e\big(a(z+s\eta+t+b) + \hbox{$\frac{1}{2}$}a^2\eta \big) 
e(-s(z+a\eta+b)-\tfrac{1}{2}s^2\eta)
\vartheta(z+a\eta+b\,|\,\eta)
\\
& \;=\;  e(a(s\eta+t)) 
e(-s(z+a\eta+b)-\tfrac{1}{2}s^2\eta)
\vartheta \hbox{$\left[ a \atop b \right]$} (z  \, | \, \eta)
\\
& \;=\; e\big(at-s(z+b) -\tfrac{1}{2}s^2\eta\big) \,
\vartheta \hbox{$\left[ a \atop b \right]$} (z  \, | \, \eta)
\end{align*}
as claimed.
\end{proof}

The functions $\vartheta \! \left[ a \atop b \right]$ are related to the $\theta_\a$'s 
defined in \cite[Prop.~2.6]{CKS1} in the following way.

\begin{lemma}
\label{le.factor} 
There is a non-zero constant $c \in \bC$,  independent of $\a$ and $z$, such that
\begin{equation*}
 \vartheta  \! \left[ \frac{\a}{n}+\frac{1}{2} \atop \frac{1}{2} \right] \! (z \, | \, n\eta) \; = \; c^{-1}\, e(-\tfrac{1}{2}z) \, \theta_\a\big(\tfrac{z}{n}\big)
 \end{equation*}
 for all $\a \in \bZ$ and all $z \in \bC$.
\end{lemma}
\begin{proof}
  Since the functions $\theta_{\a}(z)$, $\a\in\bZ$, are characterized up to a common non-zero scalar multiple by their quasi-periodicity properties
  \begin{equation*}
    \theta_{\a}\big(z+\tfrac{1}{n}\big) \;=\; e\big(\tfrac{\a}{n}\big)\theta_{\a}(z)
    \qquad\text{and}\qquad
    \theta_{\a}\big(z+\tfrac{1}{n}\eta\big)\;=\; e\big(-z-\tfrac{1}{2n}+\tfrac{n-1}{2n}\eta\big)\theta_{\a+1}(z)
  \end{equation*}
  it suffices to show that the functions $e(\frac{1}{2}nz) \vartheta \! \left[ \frac{\a}{n}+\frac{1}{2} \atop \frac{1}{2} \right] \! (nz \, | \, n\eta)$ have the same quasi-periodicity properties.  The first equality in \cref{eq:a+1.b+1} implies that $e(\frac{1}{2}nz) \vartheta \! \left[ \frac{\a}{n}+\frac{1}{2} \atop \frac{1}{2} \right] \! (nz \, | \, n\eta)$ depends only on the image of $\alpha$ in $\bZ_{n}$.  By \cite[p.~10]{Mum07},
  \begin{equation*}
    \vartheta \hbox{$\left[ a \atop b \right] $}  (z+1\, | \, \eta)=e(a)\vartheta \hbox{$\left[ a \atop b \right] $}  (z\, | \,\eta)
    \quad\text{and}\quad
    \vartheta\hbox{$\left[ a \atop b \right] $} (z+\tfrac{1}{n}\eta\, | \,\eta)=
    e(-\tfrac{1}{n}z-\tfrac{b}{n}-\tfrac{1}{2n^{2}}\eta)\vartheta \hbox{$\left[ a +\frac{1}{n} \atop b \right] $} (z\, | \,\eta).
  \end{equation*}
  Therefore
  \begin{align*}
    e\big(\tfrac{1}{2}n(z+\tfrac{1}{n})\big)  \vartheta \! \left[ \frac{\a}{n}+\frac{1}{2} \atop \frac{1}{2} \right] \! \big(n(z+\tfrac{1}{n}) \, | \, n\eta\big)
    & \;=\;  e\big(\tfrac{1}{2}\big)  e\big(\tfrac{1}{2}nz\big)  e\big(\tfrac{\a}{n} + \tfrac{1}{2}\big)
      \vartheta \hbox{$\left[\frac{\a}{n}+\frac{1}{2}  \atop \frac{1}{2} \right] $}  (nz\, | \, n\eta) 
    \\
    & \;=\; e\big(\tfrac{\a}{n}\big)e\big(\tfrac{1}{2}nz\big)\vartheta \hbox{$\left[\frac{\a}{n}+\frac{1}{2}  \atop \frac{1}{2} \right] $}  (nz\, | \, n\eta)
  \end{align*}
  and
  \begin{align*}
 e\big(\tfrac{1}{2}n(z+\tfrac{1}{n}\eta)\big) \vartheta  \hbox{$\left[\frac{\a}{n}+\frac{1}{2}  \atop \frac{1}{2} \right] $}  (n(z+\tfrac{1}{n}\eta)\, | \,n\eta)
    &\; =\; e\big(\tfrac{1}{2}\eta\big)  e\big(\tfrac{1}{2}nz\big)  e\big(-z-\tfrac{1}{2n}-\tfrac{1}{2n}\eta\big)
      \vartheta \hbox{$\left[\frac{\a+1}{n}+\frac{1}{2}  \atop \frac{1}{2} \right] $} (nz\, | \,n\eta)
      \\
    &\; =\;  e\big(-z-\tfrac{1}{2n}+\tfrac{n-1}{2n}\eta\big )e\big(\tfrac{1}{2}nz\big) \vartheta \hbox{$\left[\frac{\a+1}{n}+\frac{1}{2}  \atop \frac{1}{2} \right] $} (nz\, | \,n\eta).
  \end{align*} 
    Thus, $e(\frac{1}{2}nz)  \vartheta \! \left[ \frac{\a}{n}+\frac{1}{2} \atop \frac{1}{2} \right] \! (nz  \, | \, n\eta)$ has the same 
    quasi-periodicity properties as $\theta_\a(z)$.
\end{proof}

\section{Elliptic solutions to the quantum Yang-Baxter equation}
\label{se.main}

In this section we assume $\tau \notin \frac{1}{n} \Lambda$ and set\index{xi@$\xi$}
\begin{equation*}
\xi \; :=\; \tau \, +\, \tfrac{1}{2}(1+\eta).
\end{equation*}
Notice that $\xi \notin \frac{1}{2}(1+\eta) + \frac{1}{n}\Lambda$. 

We will use the proof of \cref{th.qybe1} in \cite{ric-tra}  to show that $R(z)$ satisfies (\ref{QYBE2}), i.e., 
to prove \cref{th.qybe2} below.\footnote{The $\sum$ symbol in \cite[(3.11)]{ric-tra} should be  $\prod$, and the symbol $\gamma_0$ 
in that equation denotes a non-zero scalar.}

Following \cite[(3.2)]{ric-tra}, for each $(a,b) \in \bZ^2$, we define\index{w_(a,b)(z)@$w_{(a,b)}(z)$}
\begin{equation}
\label{defn.w_p}
  w_{(a,b)}(z) \;:=\; \frac{\vartheta \! \left[ a/n \atop b/n \right] \! (z+\xi \, | \, \eta)}{\vartheta \! \left[ a/n \atop b/n \right]\! (\xi \, | \, \eta)} \,\, .
\end{equation}
Since $\xi \notin \frac{1}{2}(1+\eta) + \frac{1}{n}\Lambda$, the denominator of $w_{(a,b)}(z)$ is non-zero whence $w_{(a,b)}(z)$ is a holomorphic function of $z$.  It follows from \cref{eq:a+1.b+1} that $w_{(a,b)}(z)$ depends only on the images of $a$ and $b$ in $\bZ_n$. Thus, if $p=(a,b) \in \bZ_n^2$, there is a well-defined holomorphic function $w_p(z)$.

\begin{theorem}
\cite[Thm.~4.4]{tra}
\label{th.qybe1}
For $p=(a,b) \in \bZ_n^2$, let $I_p:V \to V$\index{I_p, I_(a,b)@$I_p$, $I_{(a,b)}$} be the linear map $I_p(x_i)=\omega^{ib}x_{i-a}$, where $\omega:=e(\frac{1}{n})$.
The operator\index{S(z)@$S(z)$}
\begin{equation}\label{eq:wii}
 S(z) \; :=\;  \sum_{p\in \bZ_n^2}w_p(z)I_p\otimes I_p^{-1} 
\end{equation}
satisfies (\ref{QYBE1}).
\end{theorem}

For $n=2$, this result was proved by Baxter \cite{bax1,bax2,bax3}. It was formulated and conjectured to be true for all $n$ by Belavin \cite{bel}, and was subsequently proved by Cherednik \cite{cher}, Chudnovsky and Chudnovsky \cite{ch2}, and by Tracy \cite{tra}.

We need a slightly more elaborate version of \Cref{th.qybe1}. In \cite{tra,ric-tra}, and in the other papers showing that $S(z)$ satisfies (\ref{QYBE1}), the operators $I_{p}$ are defined after first realizing $V$ as an irreducible representation of the Heisenberg group\index{H_n@$H_n$}
\begin{equation}\label{eq:14}
  H_n \; := \; 
  \left\langle
    \gamma, \chi, \epsilon\ |\ \gamma^n=\chi^n=\epsilon^n=1,\ [\gamma,\epsilon]=[\chi,\epsilon]=1,\ [\gamma,\chi]=\epsilon
  \right\rangle.
\end{equation}
of order $n^3$.
The representation on $V$ is via operators $\gamma\mapsto g\in \End(V)$ and $\chi\mapsto h\in \End(V)$ where 
$g\cdot x_i:=\omega^i x_i$ and $h\cdot x_i:=x_{i-1}$ and $\omega=e\big(\frac{1}{n}\big)$; the central element $\epsilon\in H_n$ now acts as multiplication by $\omega^{-1}$, and we have $I_{(a,b)}=h^ag^b$.

We can now apply the discussion in \Cref{sssec:genr} to the group algebra $S:=\bC H_n$ with $Z:=\bC\langle\epsilon\rangle$, the group algebra of the center $\langle\epsilon\rangle<H_n$.  

\begin{theorem}\label{th:qybe-gp-alg}
  For $p=(a,b) \in \bZ_n^2$, let $J_p\in \bC H_n$ be the element $\chi^a\gamma^b$. The family of operators
\begin{equation}\label{eq:13}
  S(z) \; :=\;  \sum_{p\in \bZ_n^2}w_p(z)J_p\otimes J_p^{-1}
  \quad
  \in
  \quad
  \bC H_n\otimes_{ \bC \langle\epsilon\rangle} \bC H_n
\end{equation}
satisfies (\ref{QYBE1}).  
\end{theorem}
\begin{proof}
  This is essentially what the proof of \cite[Thm.~4.4]{tra} shows; at no point does that proof use the specific realization of $h$ and $g$ as operators on $V$, beyond the fact that their commutator is a root of unity of order (dividing) $n$. One can therefore replace those operators with their abstract versions in \Cref{eq:14}, and $\omega$ with the generator $\epsilon$ of the center of $H_n$.
\end{proof}

As a consequence, we have the following generalization of \Cref{th.qybe1}.

\begin{corollary}\label{cor:tracy-plus}
  Let $\phi:H_n\to \End(V)$ be a representation such that $\epsilon$ acts on $V$ as a scalar multiplication and define $I_{(a,b)}^\phi :=\phi(J_{(a,b)})=\phi(\chi^a \gamma^b)$. The operator
\begin{equation}\label{eq:wii-more}
  S^\phi(z) \; :=\;  \sum_{p\in \bZ_n^2}w_p(z)I^{\phi}_p\otimes (I^{\phi}_p)^{-1}
\end{equation}
satisfies (\ref{QYBE1}).
\qedhere
\end{corollary}

Recall that $P\in\End(V\otimes V)$ is defined by $P(x\otimes y)=y\otimes x$.

\begin{proposition}
\label{prop.reln.btw.R(z).and.Tk(z)}
Let $k'$ be the unique integer such that $kk'=1$ in $\bZ_n$ and $n> k'\ge 1$, and define 
\begin{equation*} 
S_k(z)  \; := \; \sum_{(a,b) \in \bZ_n^2} w_{(a,b)}(z) I_{(-k'a,b)}\otimes I^{-1}_{(-k'a,b)}
\end{equation*}
where $I_{(-k'a,b)}:V \to V$ is the operator $x_i \mapsto \omega^{ib}x_{i+k'a}$.\footnote{The operator $S_k(z)$ is defined in the same way as $S(z)$ after replacing the generator $h$ by the new generator  $h^{-k'}$.}
Then\index{S_k(z)@$S_k(z)$} 
\begin{equation}
\label{eq.reln.Sk.R} 
S_k(-nz) \; =\; n e(\tfrac{1}{2}n(n+1)z) \, PR_{n,k,\tau}(z).
\end{equation}
\end{proposition}
\begin{proof} 
We first prove the result for $k=-1$. When $k=-1$, $S_k(z)$ is the operator $S(z)$ in \cref{eq:wii}. 
The coefficient of $x_{i+r}\otimes x_{j-r}$ in $S(z)(x_i\otimes x_j)$ is
\begin{equation*}
S(z)^{i,j}_{i+r,j-r} \; :=\;   \sum_{b\in \bZ_n} w_{(-r,b)}(z)\, \omega^{-b(j-i-r)}.   
 \end{equation*}
This is the function $S^{r,j-i-r}(z,w,\ldots)$ in \cite[(3.4)]{ric-tra} (after replacing their $a$, $b$, $\alpha$, and $\tau$ in \cite[(3.4)]{ric-tra} by our
$r$, $j-i-r$, $b$, and $\eta$, respectively).  If, in \cite[(3.3)]{ric-tra}, we replace their $\tau$ and $w$ by our $\eta$ and $n\tau$, 
respectively, then their $\eta$ becomes our $\xi$. The second variable $w$ in $S^{r,j-i-r}(z,w,\cdots)$ becomes $n\tau$.
We now have  
\begin{align*} 
S(z)^{i,j}_{i+r,j-r} 
  & \;=\; 
      f(z) \,\,
   \frac  { \vartheta  \! \left[ \frac{j-i-2r}{n}+\frac{1}{2} \atop \frac{1}{2} \right] \! (z+n\tau \, | \, n\eta) }
{ \vartheta  \! \left[ -\frac{r}{n}+\frac{1}{2} \atop \frac{1}{2} \right] \! (n\tau \, | \, n\eta)   \cdot   \vartheta  \! \left[ \frac{j-i-r}{n}+\frac{1}{2} \atop \frac{1}{2} \right] \! (z\, | \, n\eta)   } 
\qquad\qquad \text{by \cite[(3.10)]{ric-tra}}
  \\
 & \;=\; 
      f(z) \,\,
\frac  
{c^{-1}e\big(-\frac{1}{2}(z+n\tau) \big)  \theta_{j-i-2r}(\tfrac{z}{n}+\tau) }
{c^{-1}e\big(-\frac{1}{2}n\tau \big)     \theta_{-r}(\tau)   \cdot  c^{-1}e\big(-\frac{1}{2}z \big)    \theta_{j-i-r}(\tfrac{z}{n})   } 
\qquad \text{by \cref{le.factor}}
 \\
   & \;=\; 
  c    f(z) \,\,
\frac  
{  \theta_{j-i-2r}(\tfrac{z}{n}+\tau) }
{  \theta_{-r}(\tau)  \theta_{j-i-r}(\tfrac{z}{n})   } 
\end{align*}
where $c$ is the constant in \Cref{le.factor} and 
\begin{align*}
f(z)  & \; = \; n e\big(\!-\tfrac{1}{2}z\big)  \,\,
 \vartheta \! \left[\frac{1}{2} \atop \frac{1}{2} \right] \! (z \, | \, n\eta)  \, 
\prod_{\a=1}^{n-1}  
\left(
 \frac{  \vartheta \! \left[ \frac{\a}{n}+\frac{1}{2} \atop \frac{1}{2} \right] \! (z \, | \, n\eta)   }
{  \vartheta \! \left[ \frac{\a}{n}+\frac{1}{2} \atop \frac{1}{2} \right] \! (0 \, | \, n\eta)    }
\right)
\qquad \text{by \cite[(3.10)]{ric-tra}}
\\
& \; = \; n e\big(\!-\tfrac{1}{2}z\big)    c^{-1}e(-\tfrac{1}{2}z)  \theta_0(\tfrac{z}{n} )  \, \,
\prod_{\a=1}^{n-1}  
\left(
 \frac{  e(-\tfrac{1}{2}z)  \theta_\a(\tfrac{z}{n})   }
{  \theta_\a(0)  } \right)
\qquad \text{by \Cref{le.factor}.}
\end{align*}
Therefore
\begin{align*}
S(-nz)^{i,j}_{i+r,j-r} 
  & \;=\; 
c \,  f(-nz) \,
  \frac  
{  \theta_{j-i-2r}(-z+\tau) }
{  \theta_{-r}(\tau)  \theta_{j-i-r}(-z)   }  
\\
 & \;=\; 
n \,  e(\tfrac{1}{2}n(n+1)z)  \,\,
\frac{\theta_0(-z) \ldots \theta_{n-1}(-z)} {\theta_1(0) \ldots \theta_{n-1}(0)}   
\, \,
\frac  
{  \theta_{j-i-2r}(-z+\tau) }
{  \theta_{-r}(\tau)  \theta_{j-i-r}(-z)   }  \, . 
\end{align*}
The last expression is $n  e(\tfrac{1}{2}n(n+1)z)$ times the coefficient of $x_{j-r}\otimes x_{i+r}$ in
$R(z)(x_i \otimes x_j)$  when $k=-1$ (see \cref{eq:odr}). Thus, the proposition is true for $k=-1$.

We now address the general case.

The coefficient of $x_{i+r} \otimes x_{j-r}$ in  $S_k(-nz)(x_i \otimes x_j)$ is
\begin{equation*}
S_k(-nz)^{i,j}_{i+r,j-r}   \; := \;  \sum_{b\in \bZ_n} w_{(kr,b)}(-nz)\, \omega^{-b(j-i-r)} \, .
\end{equation*}
A suitable adjustment to the arguments in \cite[\S3]{ric-tra} shows that
\begin{align*}
S_k(-nz)^{i,j}_{i+r,j-r} 
&   \;=\; 
      f(-nz) \,\,
         \frac  { \vartheta  \! \left[ \frac{j-i+r(k-1)}{n}+\frac{1}{2} \atop \frac{1}{2} \right] \! (-nz+n\tau \, | \, n\eta) }
{ \vartheta  \! \left[ \frac{kr}{n}+\frac{1}{2} \atop \frac{1}{2} \right] \! (n\tau \, | \, n\eta)   \cdot   \vartheta  \! \left[ \frac{j-i-r}{n}+\frac{1}{2} \atop \frac{1}{2} \right] \! (-nz\, | \, n\eta)   } 
\\
& \;=\; c f(-nz)    \,\,
\frac{
  {{\theta}}_{j-i+r(k-1)} (-z+\tau)
  }
  { 
  {{\theta}}_{kr}  (\tau)
  {{\theta}}_{j-i-r} (-z)
  }
  \\
  & \;=\;
 n \, e\big(\tfrac{1}{2}n(n+1)z\big)  \,  \theta_0(-z) \,
 \left(
\prod_{\a=1}^{n-1}  
 \frac{  
 \theta_\a(-z)
 }
{
\theta_\a(0)
}
\right)  
\frac{
  {{\theta}}_{j-i+r(k-1)} (-z+\tau)
  }
  { 
  {{\theta}}_{kr}  (\tau)
  {{\theta}}_{j-i-r} (-z)
  }
\end{align*}
where $c$ is the constant in \cref{le.factor}.  
Comparing this with the definition of $R(z)$ in \cref{eq:odr} completes the proof.
\end{proof}

\begin{proposition}
Let $k'$ be the unique integer such that $n>k'\ge 1$ and $kk'=1$ in $\bZ_n$.
Define 
\begin{equation*}
T(z) \; :=\;   \sum_{(a,b) \in \bZ_n^2} w_{(a,b)}(z) I_{(-k'a,b)}\otimes I^{-1}_{(-k'a,b)}.
\end{equation*}
Then\index{T(z)@$T(z)$}
\begin{equation}
\label{eq:reln.T(z).R(z)} 
T(-nz) \; =\; n \, e\big(\tfrac{1}{2}n(n+1)z\big) \, PR_{n,k,\tau}(z).
\end{equation}
\end{proposition}
\begin{proof}$\!\!$\footnote{When $k=n-1$, $T(z)$ equals the operator $S(z)$ in equation (3.1) of Richey and Tracy's paper \cite{ric-tra}. 
The arguments in this proof are similar to theirs. 
In particular, some of our calculations are similar to those that produce equations (3.3)-(3.12) in their paper.}
When the operators on the left- and right-hand sides of 
\cref{eq:reln.T(z).R(z)} are evaluated at $x_i \otimes x_j$ the result is a linear combination of $x_{i+r} \otimes x_{j-r}$, $r \in \bZ_n$.
Thus, to prove the proposition it suffices to show that the coefficients of all $x_{i+r} \otimes x_{j-r}$ in these evaluations are the same
(for all $i,j,r$). 
This is what we will prove.

For the remainder of the proof we fix $i,j,r$ and set $s:= j-i-r$.

The coefficient of $x_{i+r} \otimes x_{j-r}$ in  $T(-nz)(x_i \otimes x_j)$ is
\begin{equation*}
F(z)   \; := \;  \sum_{b\in \bZ_n} w_{(kr,b)}(-nz)\, \omega^{-bs} \, .
\end{equation*}
The coefficient of $x_{i+r} \otimes x_{j-r}$ in  $n e\big(\tfrac{1}{2}n(n+1)z\big) PR_{n,k,\tau}(z) (x_i \otimes x_j)$ is  
\begin{equation*}
G(z) \; :=\;  n \, e\big(\tfrac{1}{2}n(n+1)z\big) \,  \frac{\theta_0(-z) \ldots \theta_{n-1}(-z)}{\theta_1(0) \ldots \theta_{n-1}(0)}
  \frac{\theta_{s+kr}(-z+\tau)}
  {\theta_{s}(-z)\theta_{kr}(\tau)}\, .
\end{equation*}
We must show that $F(z)=G(z)$. To do this we will show that $F(z)$ and $G(z)$ have the same quasi-periodicity properties (with respect to the lattice $\frac{1}{n}\bZ+\bZ\eta$),  the same zeros, and that  $F\big( \tfrac{s}{n} \eta \big)=G\big( \tfrac{s}{n} \eta \big)$. 
It follows from the first two of these facts that $F(z)$ and $G(z)$ are scalar multiples of each other, 
and it then follows from the equality that this scalar is 1.

{\bf Quasi-periodicity properties of $G(z)$:}
Since $\theta_\a(z+\frac{1}{n})=e\big(\frac{\a}{n}\big) \theta_\a(z)$,
\begin{equation*}
G\big(z+\tfrac{1}{n}\big) \;=\;   e\big(-\tfrac{kr}{n}\big)  \,  G(z).
\end{equation*}
Since $\theta_\a(-z-\eta)=e\big(n(-z-\eta)  - \tfrac{1}{2}\big) \theta_\a(-z)$,
\begin{equation*}
\prod_{\substack{\a=0 \\ \a \ne s}}^{n-1} \theta_\a(-z-\eta)
 \;=\; e\big((n-1)n(-z-\eta)  - \tfrac{n-1}{2}\big)  \prod_{\substack{\a=0 \\ \a \ne s}}^{n-1} \theta_\a(-z)
\end{equation*}
and
\begin{equation*}
 \theta_{s+kr}(-z-\eta+\tau) \;=\;  e\big(n(-z-\eta+\tau)  - \tfrac{1}{2}\big) \theta_{s+kr}(-z+\tau).
\end{equation*}
Therefore
\begin{align*}
G(z+\eta)  & \;=\;  e\big(\tfrac{1}{2}n(n+1)\eta) \,  e\big(n^2(-z-\eta)  +n\tau - \tfrac{n}{2}  \big) \, G(z)
\\
& \;=\;  e\big(-n^2z -  \tfrac{1}{2}n^2\eta    +  \tfrac{n}{2}(1+\eta)   +n\tau  \big) \,  G(z).
\end{align*}

{\bf Computation of $G\big( \tfrac{s}{n} \eta \big)$:}
By \cite[Prop.~2.6(7)]{CKS1},  
\begin{equation*}
\theta_\a(z-\tfrac{s}{n}\eta) \; =\;  e\big(sz+\tfrac{s}{2n} - \tfrac{sn+s^2}{2n}\eta\big)\theta_{\a-s}(z),
\end{equation*}
whence
\begin{equation*}
\frac{\theta_{s+kr}(-\tfrac{s}{n}\eta +\tau)}{\theta_{kr}(\tau)} \;=\; e\big(s\tau+\tfrac{s}{2n} - \tfrac{sn+s^2}{2n}\eta\big)
\end{equation*}
and
\begin{equation*}
\theta_\a(-\tfrac{s}{n}\eta) \;=\;   e\big(\tfrac{s}{2n} - \tfrac{sn+s^2}{2n}\eta\big) \theta_{\a-s}(0).
\end{equation*}
The singularity of the function 
\begin{equation*}
 \frac{\theta_0(-z) \ldots \theta_{n-1}(-z)}  {\theta_s(-z)}
 \end{equation*}
 at $z =  \tfrac{s}{n}\eta$ is removable and the value of the associated holomorphic function at $z=\frac{s}{n}\eta$ is
\begin{align*}
\prod_{\substack{\a=0 \\ \a \ne s}}^{n-1} \theta_\a\big(-\tfrac{s}{n} \eta \big)& \;=\;  \prod_{\substack{\a=0 \\ \a \ne s}}^{n-1} e\big(\tfrac{s}{2n} - \tfrac{sn+s^2}{2n}\eta\big) \theta_{\a-s}(0).
\\
 & \;=\;       e\big((n-1)\big(\tfrac{s}{2n} - \tfrac{sn+s^2}{2n}\eta\big)\big)  \,    \prod_{\a=1}^{n-1}  \theta_{\a}(0)  .
\end{align*}
Therefore 
\begin{align*}
G\big(\tfrac{s}{n}\eta \big)  & \;  =\;  n \, e\big(\tfrac{1}{2}(n+1) s\eta\big)  \,   e\big((n-1)\big(\tfrac{s}{2n} - \tfrac{sn+s^2}{2n}\eta\big)\big) \,
e(s\tau+\tfrac{s}{2n} - \tfrac{sn+s^2}{2n}\eta)
\\
& \;  =\;  n \, e\big(\tfrac{s}{2} (\eta +1)   - \tfrac{s^2}{2}\,\eta+ s\tau \big).
\end{align*}

{\bf The zeros of $G(z)$:}
Since $\theta_\a(z)$ has zeros at points in $-\tfrac{\a}{n}\eta+\tfrac{1}{n}\bZ+\bZ \eta$,
$G(z)$ has zeros at the points in the set
\begin{equation*}
\big\{ \tau + \tfrac{s+kr}{n}\eta\big\} \, \cup \, 
\big\{0,\tfrac{1}{n}\eta,\ldots, \tfrac{n-1}{n}\eta\big\}  \, -\,  \big\{\tfrac{s}{n}\eta\big\}.
\end{equation*} 

{\bf Quasi-periodicity properties of $F(z)$:}
Since $  \vartheta \hbox{$\left[ a \atop b \right] $}  (z+1\, | \, \eta)=e(a)\vartheta \hbox{$\left[ a \atop b \right] $}  (z\, | \,\eta)$,
\begin{align*}
F(z+\tfrac{1}{n})  & \;=\; \sum_{b\in \bZ_n} w_{(kr,b)}(-nz-1)\, \omega^{-bs}
\\
& \;=\; \sum_{b\in \bZ_n} \frac{\vartheta \! \left[ kr/n \atop b/n \right] \! (-nz-1+\xi \, | \, \eta)}{\vartheta \! \left[ kr/n \atop b/n \right]\! (\xi \, | \, \eta)} \, \, \omega^{-bs}
\\
& \;=\; \sum_{b\in \bZ_n} e\big(-\tfrac{kr}{n}\big)  
 \frac{\vartheta \! \left[ kr/n \atop b/n \right] \! (-nz+\xi \, | \, \eta)}{\vartheta \! \left[ kr/n \atop b/n \right]\! (\xi \, | \, \eta)} \, \, \omega^{-bs}
\qquad 
\\
 & \;=\;    \omega^{-kr} F(z).
\end{align*}

If $p,q \in \bZ$, then 
$
\vartheta \hbox{$\left[ a \atop b \right]$} (z + p\eta+q \, | \, \eta) = e\big(aq-p(z+b) -\tfrac{p^2}{2}\eta\big) \,
\vartheta \hbox{$\left[ a \atop b \right]$} (z  \, | \, \eta)
$
by  \cref{prop.qp.theta.char}, so
\begin{align*}
F(z+\tfrac{m}{n}\eta)  & \;=\; \sum_{b\in \bZ_n} w_{(kr,b)}(-nz-m\eta)\, \omega^{-bs}
\\
& \;=\; \sum_{b\in \bZ_n} 
\frac
{\vartheta \! \left[ kr/n \atop b/n \right] \! (-nz-m\eta+\xi \, | \, \eta)}
{\vartheta \! \left[ kr/n \atop b/n \right]\! (\xi \, | \, \eta)} \, \, \omega^{-bs}
\\
& \;=\; \sum_{b\in \bZ_n}
e\big(m(-nz+\xi+ b/n) -\tfrac{m^2}{2}\eta\big) \,
 \frac{\vartheta \! \left[ kr/n \atop b/n \right] \! (-nz+\xi \, | \, \eta)}{\vartheta \! \left[ kr/n \atop b/n \right]\! (\xi \, | \, \eta)} \, \, \omega^{-bs}
 \\
& \;=\;  e\big(-nmz + m\xi  -\tfrac{m^2}{2}\eta\big)  \sum_{b\in \bZ_n}
e\big(\tfrac{bm}{n}) \,
 \frac{\vartheta \! \left[ kr/n \atop b/n \right] \! (-nz+\xi \, | \, \eta)}{\vartheta \! \left[ kr/n \atop b/n \right]\! (\xi \, | \, \eta)} \, \, \omega^{-bs}
  \\
& \;=\;  e\big(-nmz + m\xi  -\tfrac{m^2}{2}\eta\big)  \sum_{b\in \bZ_n}
 \frac{\vartheta \! \left[ kr/n \atop b/n \right] \! (-nz+\xi \, | \, \eta)}{\vartheta \! \left[ kr/n \atop b/n \right]\! (\xi \, | \, \eta)} \, \, \omega^{-b(s-m)}.
\end{align*}

Setting $m=n$, we see that
\begin{equation*}
F(z+\eta)   \;=\; 
e\big(-n^2z + n\tau+ \tfrac{n}{2}(1+\eta)   -\tfrac{n^{2}}{2}\eta\big) F(z).
\end{equation*}
Thus,  $F(z)$ and $G(z)$ have the same quasi-periodicity properties with respect to $\frac{1}{n}\bZ + \bZ \eta$.

{\bf The zeros of $F(z)$:}
It follows from the formulas for $F(z+\frac{1}{n})$ and $F(z+\eta) $ that 
$F(z)$ has $n$ zeros\footnote{Since $F(z+\frac{1}{n})=e(-\frac{kr}{n})F(z)$ we may apply \cref{lem.theta.fns}  to $F(z)$ 
with $\eta_1=\frac{1}{n}$, $\eta_2=\eta$, $a=0$, $b=\tfrac{kr}{n}$, $c=n^2$, and $d= n(-\tau-\tfrac{1}{2}(1+\eta)+\tfrac{n}{2}\eta)$. 
Thus, $c \eta_1 - a\eta_2 =n $ and 
\begin{align*}
\tfrac{1}{2}(c \eta_1^2 - a\eta_2^2 ) + (c-a)\eta_1\eta_2 + b\eta_2 - d\eta_1 & \;=\; \tfrac{1}{2} +n\eta + \tfrac{kr}{n}\eta
+\tau +\tfrac{1}{2}(1+\eta)-\tfrac{n}{2}\eta
\\
& \;=\; \tau+\tfrac{kr}{n}\eta   +  \tfrac{1}{2}(n+1)\eta \qquad \text{modulo $\tfrac{1}{n}\bZ+\bZ\eta$.}
\end{align*}
}
 in each fundamental parallelogram for $\frac{1}{n}\bZ+\bZ \eta$, and the sum of these zeros is 
$\tau+\tfrac{kr}{n}\eta  +  \tfrac{1}{2}( n+1)\eta$ modulo $\frac{1}{n}\bZ+\bZ \eta$.

Setting $z=0$ above, we see that 
\begin{equation*}
F(\tfrac{m}{n}\eta) \;=\; e\big(m\xi  -\tfrac{m^2}{2}\eta\big)  \sum_{b\in \bZ_n}
 \frac{\vartheta \! \left[ kr/n \atop b/n \right] \! (\xi \, | \, \eta)}{\vartheta \! \left[ kr/n \atop b/n \right]\! (\xi \, | \, \eta)} \, \, \omega^{-b(s-m)}
\end{equation*}
which is zero when $m \ne s$ in $\bZ_{n}$. 
Thus, $F(z)$ vanishes at the $n-1$ points in the set
\begin{equation}
\label{eq:zeros.F(z)}
\big\{0,\tfrac{1}{n}\eta,\ldots, \tfrac{n-1}{n}\eta\big\}  \, -\,  \big\{\tfrac{s}{n}\eta\big\}.
\end{equation}
These points belong to a single fundamental parallelogram for $\frac{1}{n}\bZ+\bZ\eta$ and their sum is
$
\tfrac{1}{2}(n-1)\eta \, -\,  \tfrac{s}{n}\eta.
$
Hence there is another zero at
$$
\tau+\tfrac{kr}{n}\eta  +  \tfrac{1}{2}( n+1)\eta  \, - \, \tfrac{1}{2}(n-1)\eta \, +\,  \tfrac{s}{n}\eta \; = \; \tau+\tfrac{s+kr}{n}\eta \qquad \text{modulo $\tfrac{1}{n}\bZ+\bZ\eta$} .
$$

{\bf Comparison of $F(z)$ and $G(z)$:}
Thus $F(z)$ and $G(z)$ have the same zeros. 
They also have the same quasi-periodicity properties with respect to $\frac{1}{n}\bZ+\bZ\eta$ so
their ratio is a doubly periodic meromorphic function without zeros or poles, and therefore a constant.
However, the formula for $F(\tfrac{s}{n}\eta)$ above gives
\begin{equation*}
F(\tfrac{s}{n}\eta)   \;=\;  n \, e\big(s\xi  -\tfrac{s^2}{2}\eta\big) \;=\;  n \, e\big(s\tau+\tfrac{s}{2}(1+\eta)  -\tfrac{s^2}{2}\eta\big) 
\end{equation*}
which equals $G(\tfrac{s}{n}\eta)$ so that constant is 1. The proof is complete.
\end{proof} 

\begin{corollary}
  $Q_{n,k}(E,\tau)^{\rm op}$ is the quotient of $TV$ by the ideal generated by the image of $S_k(-n\tau)$.
\end{corollary}

\begin{theorem}
\label{th.qybe2}
The  family of operators $R_\tau(z):V^{\otimes 2} \to V^{\otimes 2}$ in \cref{eq:odr} satisfies (\ref{QYBE2}).
\end{theorem}
\begin{proof}
  Since $S_k(z)=S^\phi(z)$ where
  \begin{equation*}
    \phi:H_n\to \End(V)
  \end{equation*}
  is the representation
  \begin{equation*}
    \chi\mapsto h^{-k'},\ \gamma\mapsto g,
  \end{equation*}
  and $\epsilon$ acts as multiplication by $\omega^{k'}$,
  \cref{cor:tracy-plus} tells us that $S_k(z)$ satisfies (\ref{QYBE1}) and hence so does $S_k(-nz)$. It now follows from \cref{eq.reln.Sk.R} that $PR(z)$ satisfies (\ref{QYBE1}) and therefore $R(z)$ satisfies (\ref{QYBE2}) by \cref{prop.R.R'.QYBE}.
\end{proof}

\section{Families of linear operators}\label{se.hol}

In \cref{sec.hilb.series,subsec.dual,se.ksz}, we need to determine the zeros, and their multiplicities, of the determinants of certain linear operators $G_{\tau}(z)$, $G^{+}_{\tau}(z)$ and $H_\tau(z)$, on $V^{\otimes d}$. These operators, which are analytic functions of $z$, are compositions of operators of the form $I^{\otimes i-1} \otimes R(w) \otimes I^{d-i+1}$ for various $w$'s and $i$'s.

\subsection{Dimensions of kernels and the multiplicity of the zeros of the determinant}
The multiplicity of a point $p \in \bC$ as a zero of a meromorphic function $f(z)$ is denoted\index{mult_p f(z)@$\mult_p f(z)$}
\begin{equation*}
\mult_p f(z).
\end{equation*}
If $f(z)$ is identically zero in a neighborhood of $p$ we set $\mult_p f(z)=\infty$.

The next result is used in \cref{sec.hilb.series,subsec.dual,se.ksz}.

\begin{lemma}\label{pr.hol-op}
  Let $V$ be a  finite-dimensional complex vector space and $A:D\to \End(V)$ a holomorphic map 
  defined on a domain $D\subseteq \bC$.
  For all $p\in D$, 
  \begin{equation*}
\mult_p(\det A(z))  \; \ge  \; \nullity A(p).
  \end{equation*}
\end{lemma}
\begin{proof}
This is trivial if $\det A(z)$ is identically zero in a neighborhood of $p$, so we 
assume that $p$ is an isolated zero of $\det A(z)$.
  Let $e_1, \ldots, e_\ell$ be an ordered basis for $\ker A(p)$, and extend it to an ordered basis $e_1, \ldots, e_\ell,\ldots$  for $V$. 
 With respect to this basis, the entries of the matrix $A(z)$ are holomorphic functions whose first $\ell$ columns are divisible by $z-p$ (in the ring of functions holomorphic in a neighborhood of $p$). Hence $\det A(z)$ is divisible by $(z-p)^\ell$, finishing the proof.
\end{proof}

In \cref{se.ksz}, we need a stronger version of \Cref{pr.hol-op}. First, some terminology. 
If $A:D\to \End(V)$ is a holomorphic map as above and $p$ is a fixed point in $D$, we define
\begin{equation*}
  A_m(z) \; : = \frac{A(z)}{(z-p)^m} 
\end{equation*}
with the convention that $A_{-1}(z) \equiv 0$. 

\begin{definition}\label{def.sing-part}
  The {\sf singularity partition of $A(z)$ at a point $p\in D$} is the tuple $\sigma_p(A):=(\lambda_0\ge \lambda_1\ge\cdots)$\index{sigma_p(A)@$\sigma_p(A)$} of non-negative integers defined by
  \begin{equation*}
    \lambda_m \; := \; \text{the dimension of the kernel of }A_m(p) \big|_{\ker A_{m-1}(p)}.
  \end{equation*}
The  {\sf size} of the singularity partition is the number $|\sigma_p(A)|:=\sum_i \lambda_i$\index{_sigma_p(A)_@$\vert \sigma_p(A) \vert$}.
\end{definition}

\begin{remark}\label{re.tail}
Since $\l_m=0$ for $m \gg 0$, we ignore those  zeros and regard the partition as a {\it finite} tuple. 
\end{remark}

The next result, which improves on \Cref{pr.hol-op}, is used in the proof of \cref{pr.h-part}.

\begin{lemma}\label{pr.hol-op-bis}
  Let $V$ be a finite-dimensional complex vector space and $A:D\to \End(V)$ a holomorphic map 
  for a domain $D\subseteq\bC$. For all $p\in D$,  
  \begin{equation*}
\mult_p(\det A(z)) \; \ge \;    |\sigma_p(A)|.
  \end{equation*}
\end{lemma}
\begin{proof}
Without loss of generality, we assume that $p=0$. 
A direct sum decomposition $V=\ker A(0)\oplus V_0$  leads to a decomposition
  \begin{equation}\label{eq:5}
    A_0\; =\; A:D\to \Hom(\ker A(0),V)\oplus \Hom(V_0,V)
  \end{equation}
  whose left-hand component is a multiple of $z$, contributing $z^{\lambda_0}$ to $\det A(z)$ where
  \begin{equation*}
    \sigma_0(A) \;=\; (\lambda_0\ge \lambda_1\ge\cdots).
  \end{equation*}
  Dividing the left hand component of \Cref{eq:5} by $z$ and using a splitting $\ker A_{0}(0) = \ker A_1(0)\oplus V_1$,
  we obtain
  \begin{equation*}
    A_1:D\to \Hom(\ker A_1(0),V)\oplus \Hom(V_1,V),
  \end{equation*}
  with the left-hand component once more a multiple of $z$ contributing $z^{\lambda_1}$ to the determinant. Simply repeat the procedure until the singularity partition has been exhausted, noting that if at any point any of the functions
  \begin{equation*}
    A_m(0)\big|_{\ker A_{m-1}(0)}
  \end{equation*}
  vanish identically then $\det A(z)$ does too, making the statement trivial. 
\end{proof}

\subsection{Theta operators}
\label{ssec.theta.operator}
Assume $\Lambda =\bZ\eta_1 + \bZ\eta_2$ is a lattice in $\bC$ such that $\Im(\eta_2/\eta_1)>0$.
A {\sf theta function of order $r$} with respect to $\Lambda$ is a holomorphic function $f:\bC\to\bC$ satisfying the 
quasi-periodicity conditions $f(z+\eta_1) =e(-az-b)f(z)$ and $f(z+\eta_2)= e(-cz-d)f(z)$ in \cref{lem.theta.fns} for some constants
$a,b,c,d$ such that $c\eta_1-a\eta_2=r$. If $f$ is not the zero function it has $r$ zeros in every fundamental parallelogram for 
$\Lambda$. 
For example, the functions that belong to the space $\Theta_{r,c}(\Lambda)$\index{Theta_r,c(Lambda)@$\Theta_{r,c}(\Lambda)$}, defined in \cite[\S2.1]{CKS1}, 
are theta functions of order $r$. In particular, $\theta_\a$ is a theta function of order $n$ with respect to $\Lambda$ and $\{\theta_0, \ldots, \theta_{n-1}\}$ is a basis for $\Theta_n(\Lambda)=\Theta_{n,\frac{n-1}{2}}(\Lambda)$\index{Theta_n(Lambda)@$\Theta_{n}(\Lambda)$} (\cite[Prop.~2.6]{CKS1}). 

\begin{definition}
\label{def.thetaop}
A holomorphic map $A:\bC\to \End(V)$ is a {\sf theta operator of order $r$} with respect to $\Lambda$ 
if $A(z+\eta_1) =e(-az-b)A(z)$ and $A(z+\eta_2) =e(-cz-d)A(z)$ for some constants $a,b,c,d$ such that $c\eta_1-a\eta_2=r$. 
Equivalently, if $\braket{v^*,A(-)v}$ is a theta function of order $r$ having the same quasi-periodicity properties for all 
$v\in V$ and all $v^*\in V^*$, where $V^{*}$\index{V^*@$V^{*}$} is the dual vector space of $V$. Equivalently, the matrix entries for $A(z)$ with respect to any basis for $V$ are
theta function of order $r$ having the same quasi-periodicity properties. 
  \end{definition}

For example, $R(z)$ is a theta operator of order $n^2$ with respect to $\Lambda$
because its matrix entries belong to $\Theta_{n^2,n\tau-n^{2}\eta}(\Lambda)$ (see \cite[\S2.1.2]{CKS1}). 

If $A_i(z)$, $i=1,2$, are theta operators whose matrix entries belong to $\Theta_{r_i,c_i}(\Lambda)$, then 
$A_1(z)A_2(z)$ is a theta operator of order $r_1+r_2$ because its matrix entries belong to $\Theta_{r_1+r_2,c_1+c_2}(\Lambda)$.

If $A(z)$ is a theta operator whose matrix entries belong to $\Theta_{r,c}(\Lambda)$ and $d\in\bC$, then $A(z+d)$ is a theta operator of order $r$ because its matrix entries belong to $\Theta_{r,c-rd}(\Lambda)$.

\subsubsection{``Determinants''}
\label{rmk.det}
We often encounter theta operators that preserve a fixed subspace $W\subseteq V$ or, more generally, map it to a 
fixed subspace $W'\subseteq V$ of the same dimension. 
If $A(z)(W)\subseteq W'$ for all $z\in\bC$ and  $  \dim W =m=  \dim W'$, then $A$ induces a holomorphic function\index{det_W W'(A(z))@$\det_{W\to W'}(A(z))$}
\begin{equation}\label{eq:detf}
\textstyle  \det_{W\to W'} (A(z)) := \bigwedge^m A(z):\bigwedge^m W \, \longrightarrow \, \bigwedge^m W'
\end{equation}
that is well-defined up to a non-zero scalar multiple (depending on a choice of bases for $W$ and $W'$). 
We will often be interested in the location and multiplicities of the zeros of this function. That data does not depend on the choice of
bases. These remarks, and the next result, apply to the theta operators $G_\tau(z)$ defined in \cref{sect.defn.Gz}
and $H_\tau(z)$ defined in \cref{prop.defn.H}.

\begin{proposition}
\label{prop.hol.det.zeros}
Assume $W$ and $W'$ are subspaces of $V$ of the same dimension. If $A:\bC\to \End(V)$ is a theta operator of order $N$ 
with respect to $\Lambda$ such that $A(z)(W) \subseteq W'$ for all $z$, then $\det_{W\to W'} (A(z))$
\begin{enumerate}
  \item\label{item.hol-det} 
 is a theta function of order $N\dim W$ and
  \item\label{item.hol-zeros} 
 has $N\dim W$ zeros in every fundamental parallelogram for $\Lambda$ if it is not identically zero.   
\end{enumerate}
In particular, $\det A(z)$ is a theta function of order $N \dim V$ with respect to $\Lambda$. 
\end{proposition}
\begin{proof}
\cref{item.hol-det} Composing $A$ with an automorphism of $V$ that maps $W'$ isomorphically onto $W$, we may as well assume $W=W'$,
  whence $\det_{W\to W} (A(z))$ becomes the usual determinant of $A(z)|_W$.

  Choose an ordered basis for $W$ and extend it to one for $V$. The operators $A(z)$ then have the shape
  \begin{equation*}
    \begin{pmatrix}
      A_{11}(z)& A_{12}(z)\\0 & A_{22}(z)
    \end{pmatrix},
  \end{equation*}
  and $\det( A(z)|_W)=\det A_{11}(z)$. Since the summands in the usual expression for $\det A_{11}(z)$ are products of 
  $\dim W$ theta functions of order $N$ having the same quasi-periodicity properties,
  those summands, and therefore their sum,  are theta functions of order $N \times \dim W$.
  
\cref{item.hol-zeros}
A non-zero theta function of order $r$ has $r$ zeros in a fundamental parallelogram (\cref{lem.theta.fns}).
\end{proof}

\subsection{Families of kernels and images}
Let $\Grass(d,W)$\index{Grass(d,W)@$\Grass(d,W)$} denote the Grassmannian of $d$-dimensional subspaces of a finite-dimensional $\bC$-vector space $W$.

In this subsection we consider algebraic or analytic morphisms $f_i$ from a complex variety (algebraic or analytic\footnote{We adopt the following convention: a (complex) \emph{algebraic variety} is a scheme over $\bC$ that is reduced, irreducible, separated, and of finite type. An \emph{analytic variety} is a (Hausdorff) analytic space that is reduced and irreducible.}, though we typically specialize to the algebraic case) to either
\begin{itemize}
\item the Grassmannian $\Grass(d,W)$, or
\item the space of linear maps $\Hom(W,W')$ for finite-dimensional $\bC$-vector spaces $W$ and $W'$.
\end{itemize}
We will be interested in the families of intersections (or sums) of $f_i(y)$ in the first case or $\ker f_i(y)$ (or $\im f_{i}(y)$) in the second.

\begin{proposition}
\cite[Prop.~13.4]{saltman}
\label{pr:optosp}
  Let $f:Y\to \Hom(W,W')$ be a morphism of algebraic or analytic varieties and write $r:=\max\{\rank f(y)\mid y\in Y\}$.
  \begin{enumerate}
  \item\label{item:1} 
  The set $U:=\{y\in Y\ |\ \rank f(y)=r\}$ is an open dense subset of $Y$. 
  \item\label{item:2} The map $\ker f: U\to\Grass(\dim W - r,W)$, $u\mapsto\ker f(u)$,
    is a morphism.
  \item\label{item:3}
    The map $\im f: U\to\Grass(r,W')$, $u\mapsto\im f(u)$, is a morphism.
  \end{enumerate}
\end{proposition}
\begin{proof}
As we said above, we focus on the algebraic situation. 

\Cref{item:1} Since $Y$ is a variety, and therefore irreducible, density follows from openness and non-emptiness. The latter holds by construction (since the maximal rank is, of course, achieved somewhere), so it remains to argue that $U\subseteq Y$ is open. This is clear from the fact that the condition rank $<r$ is expressible as a collection of algebraic equations (the vanishing of $r\times r$ minors).

Parts \Cref{item:2} and \Cref{item:3} are proved in \cite[Prop.~3.17]{CKS1}.
\end{proof}

To further strengthen the connection between families of subspaces and families of operators, we have a kind of converse to parts \Cref{item:2} and \Cref{item:3} of \Cref{pr:optosp}.

If $F:X \to  \Hom(W,W')$ is a function we write $\ker F$ and $\im F$ for the functions $x\mapsto \ker F(x)$ and $x \mapsto \im F(x)$,
respectively.

\begin{lemma}\label{pr:sptoop}
  Let $f:Y\to\Grass(d,W)$ be a morphism.
  \begin{enumerate}
    \item\label{item.sptoop.ker}
    Let $W'$ be a fixed vector space of dimension $\ge \dim W-d$, then $Y$ can be covered with open subvarieties $U$ for which there are morphisms 
  \begin{equation*}
    F_U:U\to \Hom(W,W')
  \end{equation*}
  such that $f|_U = \ker F_U$.
    \item\label{item.sptoop.im}
    Let $W''$ be a fixed vector space of dimension $\ge d$, then $Y$ can be covered with open subvarieties $U$ for which there are morphisms 
  \begin{equation*}
    G_{U}:U\to \Hom(W'',W)
  \end{equation*}
  such that $f|_{U} = \im G_{U}$.
  \end{enumerate}
\end{lemma}
\begin{proof}
We prove \cref{item.sptoop.ker}; the dual argument shows \cref{item.sptoop.im} (using possibly different $U$'s). Cover the Grassmannian $G:=\Grass(d,W)$ with the affine open sets $U_{\Gamma}$ described in \cite[\S 3.2.2]{3264}, consisting of the $d$-subspaces of $W$ that intersect a fixed $(\dim W-d)$-dimensional subspace $\Gamma$ trivially. Pulling these back to $Y$, we may as well assume the image of $f$ lies entirely within a single open patch $U_{\Gamma}\subset G$ for a fixed $\Gamma$;
i.e., we are now assuming that  $f(y) \cap \Gamma=\{0\}$ for all $y\in Y$.

  Now $\Gamma$ is naturally isomorphic to $W/f(y)$ via the quotient map $W\to W/f(y)$ so, for each $y\in Y$, define $F(y)\in\Hom(W,W')$ to be the composition
  \begin{equation*}
    W\;\to\;W/f(y)\;\cong\;\Gamma\;\to\;W'
  \end{equation*}
  where $\Gamma\;\to\;W'$ is some fixed embedding. This is the desired morphism 
$F:Y\to \Hom(W,W')$.
\end{proof}

\begin{proposition}
\cite[Prop.~13.5]{saltman}
\label{pr:genker}
  Let $f_i$, $1\le i\le r$, be morphisms $Y\to \Grass(d,W)$. Then,
  \begin{enumerate}
  \item\label{item:4} The sets
    \begin{align*}
      U  & \;:=\;\Big\{y\in Y \, \Big\vert \, \bigcap_{i} f_i(y)\text{ has minimal dimension } e\Big\} \qquad \text{and}
      \\
      U' & \;:=\;\Big \{y\in Y  \, \Big\vert \,  \sum_{i} f_i(y)\text{ has maximal dimension } e' \Big\}
    \end{align*}
    are open dense subsets of $Y$.
  \item\label{item:7} The maps 
  \begin{align*}
  \bigcap_{i}f_{i}:\ &U\to\Grass(e,W), \quad y\mapsto \bigcap_{i} f_i(y),  \qquad \text{and}
  \\
  \sum_{i}f_i:\ &U'\to\Grass(e',W), \quad y\mapsto \sum_{i} f_i(y),
  \end{align*}
    are morphisms.
  \end{enumerate}
\end{proposition}
\begin{proof}
Both  \Cref{item:4} and  \Cref{item:7} are true if they are true locally so, after \Cref{pr:sptoop}, 
we can assume that there are morphisms 
$F_i:Y\to \Hom(W,W')$ and $G_i:Y\to \Hom(W'',W)$ such that $\ker F_i=f_i=\im G_{i}$. Now  
  \begin{equation*}
    \bigcap_i f_i \;= \;  \ker \big(F_1\oplus \cdots\oplus F_r: Y\to \Hom(W,W'^{\oplus r})\big),
  \end{equation*}
  and  
    \begin{equation*}
    \sum_i f_i \;= \;  \im \big( (G_1, \ldots, G_r): Y\to \Hom(W''^{\oplus r},W)\big),
  \end{equation*}
so \Cref{item:4} and \Cref{item:7} follow from  \Cref{pr:optosp}.
\end{proof}

\subsubsection{Generically large and generically small functions}
Let $X$ be a topological space. We say that a function $f:X \to \bR$ is {\sf generically large} 
(resp., {\sf generically small}) if $f^{-1}(c)$ 
is an open dense set  for some $c \in \bR$ and $f(x)<c$  (resp., $f(x)>c$) for all $x \notin f^{-1}(c)$.

In the setting of \cref{pr:optosp}, the function $\rank f(x)$ on $X$ is generically large (\emph{images are generically large}) and $\nullity f(x)$ is generically small (\emph{kernels are generically small}).

We also say that \emph{sums are generically large} and \emph{intersections are generically small} based on the facts in \cref{pr:genker}.

\subsubsection{}
We will apply these ideas to situations where we have two functions $f_1,f_2:X \to \bR$ with the following properties: $f_1$ is generically small (e.g., nullity); $f_2$ is generically large (e.g., rank); $f_1(x) \le f_2(x)$ for all $x$; $f_1(x)=f_2(x)$ on an open dense subset of $X$. It follows that $f_1(x)=f_2(x)$ for all $x$.

\section{The determinant of $R_\tau(z)$ and the space of quadratic relations for $Q_{n,k}(E,\tau)$ }\label{se.det}

In  \cref{sect.H.series} we will show,  for all $\tau\in(\bC-\bigcup_{m\geq 1}\frac{1}{m}\Lambda)\cup\frac{1}{n}\Lambda$, that the Hilbert 
series for $Q_{n,k}(E,\tau)$ is the same as that for the polynomial ring $\bC[x_0,\ldots,x_{n-1}]$.
In this section we prove that, for all $\tau\in\bC-(\frac{1}{2n}\Lambda-\frac{1}{n}\Lambda)$, the degree-two components of $Q_{n,k}(E,\tau)$ and  $\bC[x_0,\ldots,x_{n-1}]$ have the same dimension, 
namely $\binom{n+1}{2}$.  
Since $\rel_{n,k}(E,\tau)$ is the image of $R_\tau(\tau)$ (see \cref{sssec.def.rel.Qnk}), it suffices to show that the nullity of $R_\tau(\tau)$ is $\binom{n+1}{2}$. 
That is what we will do.

Since we showed that $Q_{n,k}(E,\tau)$ has the same Hilbert series as the polynomial ring on 
$n$ variables when $\tau \in \frac{1}{n}\Lambda$ in \cite[\S5]{CKS1},   
we only have to prove the result for $\tau \in \bC-\frac{1}{2n}\Lambda$.

\subsection{The limit of $R_\tau(m\tau+\zeta)$ as $\tau \to 0$}
As a function of $z$, $R_\tau(z)$ is not defined when $\tau \in \frac{1}{n}\Lambda$. Nevertheless, as we observed in
\cite[\S3.3.2]{CKS1}, the holomorphic function $R_\tau(\tau)$   on $\bC - \frac{1}{n}\Lambda$ extends in a unique way to a
 holomorphic function on $\bC$. We need a slightly more general result here. 

\begin{lemma}
\label{lem.U_tau}
Fix $\zeta\in\frac{1}{n}\Lambda$ and $m\in\bZ$.  As a function of $\tau$, the operator $R_{\tau}(m\tau+\zeta)$ is holomorphic on $\bC-\frac{1}{n}\Lambda$, and its singularities at $\frac{1}{n}\Lambda$ are removable; i.e., $R_{\tau}(m\tau+\zeta)$ extends in a unique
 way to a holomorphic function of $\tau$ on the entire complex plane.
\end{lemma} \begin{proof}
 By definition, $R_{\tau}(m\tau+\zeta)(x_i\otimes x_j)$ is
\begin{equation}\label{eq:umts} 
  \frac{1}{\theta_1(0) \ldots \theta_{n-1}(0)} \Bigg(\prod_{i\in\bZ_n}\theta_{i}(-m\tau-\zeta)\Bigg)\sum_{r\in \bZ_n}
  \frac{\theta_{j-i+r(k-1)}((1-m)\tau-\zeta)}
  {\theta_{j-i-r}(-m\tau-\zeta)\theta_{kr}(\tau)}
  x_{j-r}\otimes x_{i+r}.
\end{equation}
  Suppose both theta functions in the denominator of a summand are zero at $\tau$. 
  Then $\tau=-\frac{kr}{n}$ and $-m\tau-\zeta=-\frac{j-i-r}{n}$ modulo $\frac{1}{n}\bZ +\bZ\eta$, so $(1-m)\tau-\zeta=-\frac{j-i+(k-1)}{n}$ modulo $\frac{1}{n}\bZ +\bZ\eta$; thus the numerator in the same summand is also zero; each summand therefore has at most 
  a pole of order one at $\tau$ and, since $m\tau+\zeta\in\frac{1}{n}\Lambda$,  
  such a pole is canceled out by the order-one zero in the term before the  $\Sigma$ sign.
 \end{proof}

\begin{proposition}\label{le.ulim0}
For all $m\in\bZ$,
\begin{equation}
\label{eq:lim.R(m.tau)}
\lim_{\tau \to 0} R_{\tau}(m\tau) \;=\; \sym_m
\end{equation}
where $\sym_m:V^{\otimes 2} \to V^{\otimes 2}$ is the skew-symmetrization operator\index{sym_m@$\sym_m$}  
\begin{equation*}
\sym_m(v \otimes v') \; :=\; v\otimes v' - m v'\otimes v.
\end{equation*}
\end{proposition}
\begin{proof}
The limit as $\tau \to 0$ of $R_{\tau}(m\tau)$ is the same as the limit as $\tau \to 0$ of the operator
\begin{equation*}
x_i \otimes x_j \; \mapsto \;  \theta_{0}(-m\tau) \sum_{r\in \bZ_n}   \frac{\theta_{j-i+r(k-1)}((1-m)\tau)}{\theta_{j-i-r}(-m\tau)\theta_{kr}(\tau)}
  x_{j-r}\otimes x_{i+r}.
\end{equation*}
As $\tau\to 0$, multiplication by $\theta_0(-m\tau)$ annihilates  those terms in the sum $\sum_{r\in \bZ_n}$ whose denominators do not vanish at $0$ so only the $r=0$ and $r=j-i$ terms contribute to $\lim_{\tau \to 0} R_\tau(m\tau)$. Hence 
$\lim_{\tau \to 0} R_\tau(m\tau) = \lim_{\tau \to 0} X_\tau(m\tau)$ where 
\begin{equation}
  \label{eq:xmt-ij}
X_\tau(m\tau)(x_i\otimes x_j) \; :=\; 
  \theta_0(-m\tau)\cdot
      \left(
    \frac{\theta_{k(j-i)}((1-m)\tau)}{\theta_0(-m\tau)\theta_{k(j-i)}(\tau)}x_i\otimes x_j
    \, + \, \frac{\theta_{j-i}((1-m)\tau)}{\theta_{j-i}(-m\tau)\theta_0(\tau)}
    x_j\otimes x_i
  \right)
\end{equation}
for $i \ne j$ and
\begin{equation}
  \label{eq:xmt-ii}
X_\tau(m\tau)(x_i\otimes x_i) \; :=\; 
    \theta_0(-m\tau)\cdot
    \frac{\theta_0((1-m)\tau)}{\theta_0(\tau)\theta_0(-m\tau)}
\,    x_i\otimes x_i.
\end{equation}

Assume $i\ne j$. The two $\theta_{k(j-i)}(\, \cdot \,)$ factors in the left-hand term of \Cref{eq:xmt-ij} cancel out as $\tau\to 0$ because 
both converge to $\theta_{k(j-i)}(0)$ which is non-zero; the two $\theta_0(-m\tau)$ terms also cancel out so the first term on
the right-hand side of \cref{eq:xmt-ij} converges to $x_i \otimes x_j$. 
Since $\theta_0(\tau)$ vanishes at $\tau=0$ with multiplicity 1, 
\begin{equation*}
\theta_0(\tau) = a_1\tau+a_2\tau^2+\cdots
\end{equation*}
with $a_1 \ne 0$. Hence the ratio $\frac{\theta_0(-m\tau)}{\theta_0(\tau)}$ converges to $-m$.  The two $\theta_{j-i}(\, \cdot \,)$ 
 factors in the right-hand term of \Cref{eq:xmt-ij} cancel out as $\tau\to 0$ because 
both converge to $\theta_{j-i}(0)$ which is non-zero. The second term on the right-hand side of \Cref{eq:xmt-ij} 
therefore converges to $-m x_j\otimes x_i$.
Thus, $X_\tau(m\tau) (x_i\otimes x_j)$ converges to $x_i\otimes x_j - m x_j\otimes x_i$ as $\tau \to 0$.

Assume $i=j$.
Similar analysis  shows that $X_\tau(m\tau) (x_i\otimes x_i)$ converges to $(1- m) x_i\otimes x_i$ as $\tau \to 0$. 

Combining the cases $i\ne j$ and $i=j$ gives the uniform result 
\begin{equation*}
X_{\tau}(m\tau)(x_i\otimes x_j) \; \longrightarrow \; x_i\otimes x_j - m x_j\otimes x_i
\end{equation*}
as $\tau \to 0$. The proof is now complete. 
\end{proof}

For each $\zeta\in\frac{1}{n}\Lambda$, we define\index{R_+(zeta), R_-(zeta)@$R_{+}(\zeta)$, $R_{-}(\zeta)$}
\begin{equation}\label{eq:mpm1}
	R_{+}(\zeta)\;:=\;\lim_{\tau\to 0}R_{\tau}(\tau+\zeta)\quad\text{and}\quad R_{-}(\zeta)\;:=\;\lim_{\tau\to 0}R_{\tau}(-\tau+\zeta).
\end{equation}

\begin{corollary}\label{cor.sym}
When $\zeta=0$, the operators in \Cref{eq:mpm1}  are $R_{\pm}(0)(x_i \otimes x_j)=x_i \otimes x_j \mp x_j \otimes x_i$.
\end{corollary}
\begin{proof}
  Apply \Cref{le.ulim0} with $m=1$ for $R_+(0)$ and with $m=-1$ for $R_-(0)$. 
\end{proof}

\begin{lemma}\label{lem.rz.dim.ker.tors}
	For all $\zeta\in\frac{1}{n}\Lambda$, 
	\begin{align*}
		\ker R_{+}(\zeta)  & \; = \; \im R_{-}(-\zeta),
		\\
		 \im R_{+}(\zeta) & \; =\;  \ker R_{-}(-\zeta),
		\\ 
		\nullity  R_{+}(\zeta)   & \; = \; \tbinom{n+1}{2},
		\\
		\nullity  R_{-}(\zeta)  & \; = \; \tbinom{n}{2}.
	\end{align*}
\end{lemma}
\begin{proof}
By \cref{le.ulim0},  
\begin{equation*}
R_{+}(0)(x_i \otimes x_j)=x_i \otimes x_j - x_j \otimes x_i 
\quad \text{and} \quad
R_{-}(0)(x_i \otimes x_j)=x_i \otimes x_j +  x_j \otimes x_i.
\end{equation*}
Therefore $\im R_+(0) = \ker R_-(0)$, $\im R_-(0) = \ker R_+(0)$,  $\nullity  R_{+}(0) =\tbinom{n+1}{2}$, and 
$\nullity  R_{-}(0) =\tbinom{n}{2}$. Thus, the lemma is true when $\zeta=0$.

We now consider an arbitrary $\zeta=\frac{a}{n}+\frac{b}{n}\eta$. 
The argument in the next two paragraphs will show that $R_+(\zeta) R_-(-\zeta)=R_-(-\zeta) R_+(\zeta)=0$.

Define $C:=T^bS^{ka}$, $C':=T^{-b}S^{-ka}$, $A:=C \otimes I$, 
$A'=C' \otimes I$, $B:=I \otimes C^{-1}$, and $B':=I \otimes C'^{-1}$.  By \cref{cor.R.transl.props},
\begin{equation*}
R_+(\zeta) \; = \; \lim_{\tau \to 0} f(\tau,\zeta,\tau) B R_\tau(\tau) A \; = \; f(0,\zeta,0) B  \big(\lim_{\tau \to 0} R_\tau(\tau) \big) A
 \; = \; f(0,\zeta,0) B  R_+(0) A 
\end{equation*}
and
\begin{equation*}
R_-(-\zeta) \; = \; \lim_{\tau \to 0} f(-\tau,-\zeta,\tau) B' R_\tau(-\tau) A'  \; = \; f(0,-\zeta,0) B' \big(\lim_{\tau \to 0} R_\tau(-\tau) \big) A'
 \; = \; f(0,-\zeta,0) B'  R_-(0) A'
\end{equation*}
where $f(z,\zeta,\tau)$ does not vanish at any point in $ \bC \times \frac{1}{n}\Lambda \times \bC$.  To show that $R_+(\zeta)R_-(-\zeta)=0$
it suffices to show that  $R_+(0)AB'R_-(0)=0$; i.e., that $R_+(0)  (C \otimes  C'^{-1})  R_-(0)  =0$. 

Let $\ve=e\big(\frac{1}{n}\big)$.  Since $ST=\ve TS$, $ C'^{-1}=S^{ka}T^b=\ve^{kab}T^bS^{ka} = \ve^{kab}C$.  Therefore
\begin{equation*}
R_+(0)  (C \otimes  C'^{-1})  R_-(0) \;=\; R_+(0)(I \otimes \ve^{kab})   (C \otimes  C)  R_-(0)
\;=\;  (C \otimes  C)  R_+(0)(I \otimes \ve^{kab})   R_-(0);
\end{equation*}
we used the fact that $S^{\otimes 2}$ and $T^{\otimes 2}$ commute with $R_\tau(z)$ and hence with limits of $R_\tau(m \tau + \zeta)$. Since $I \otimes \ve^{kab}$ is a scalar multiple of the identity, it also commutes with $R_{\pm}(0)$. Hence $R_+(0) (C \otimes C'^{-1}) R_-(0)=0$.  This completes the proof that $R_+(\zeta)R_-(-\zeta)=0$. A similar argument shows that $R_-(-\zeta)R_+(\zeta)=0$.

Since $A$ and $B$ are invertible operators and $f(0,\zeta,0)$ is a non-zero scalar, $\rank R_+(\zeta)=\rank R_+(0)$.  Similarly, $\rank R_-(\zeta)=\rank R_-(0)$.  The lemma therefore holds for all $\zeta \in \frac{1}{n}\Lambda$.
\end{proof}

\subsection{The ranks of $R_\tau(\tau)$ and $R_\tau(-\tau)$}

\begin{lemma}\label{lem.rz.dim.ker}
  Assume $\tau\in\bC-\frac{1}{n}\Lambda$. For all $\zeta\in\frac{1}{n}\Lambda$, 
  \begin{equation}\label{eq.lem.rz.dim.ker.incl}
    \im R_{\tau}(\tau+\zeta) \; \subseteq \; \ker R_{\tau}(-\tau-\zeta),
    \qquad
    \im R_{\tau}(-\tau-\zeta) \; \subseteq \; \ker R_{\tau}(\tau+\zeta),
  \end{equation}
  \begin{equation}
    \label{eq.lem.rz.dim.equal}
    \nullity R_{\tau}(\tau+\zeta) \; \geq \; \tbinom{n+1}{2}, \qquad\text{and}\qquad
    \nullity R_{\tau}(-\tau-\zeta)) \; \geq \; \tbinom{n}{2}.
  \end{equation}
\end{lemma}
\begin{proof}
  \cref{eq.lem.rz.dim.ker.incl} is an immediate consequence of \cref{lem.transf.props.R.tau.z}\cref{item.prop.r.pm.tors}.

  By \cref{lem.U_tau}, $R_{\tau}(\tau+\zeta)$ and $R_{\tau}(-\tau-\zeta)$ extend to holomorphic functions of $\tau$ on the whole complex plane. By \cref{lem.rz.dim.ker.tors},
  \begin{equation*}
  \nullity\big(\lim_{\tau \to 0} R_\tau(\tau+\zeta)\big) \;=\; \tbinom{n+1}{2} \;=\; \rank\big(\lim_{\tau \to 0} R_\tau(-\tau-\zeta)\big).
  \end{equation*}
  Since nullity is generically small and rank is generically large it follows that
  \begin{equation*}
    \nullity R_{\tau}(\tau+\zeta) \; \leq \; \tbinom{n+1}{2} \; \leq \; \rank R_{\tau}(-\tau-\zeta)
  \end{equation*}
  for generic $\tau$. But $\im R_{\tau}(-\tau-\zeta) \subseteq \ker R_{\tau}(\tau+\zeta)$ so
  \begin{equation*}
    \nullity R_{\tau}(\tau+\zeta) \; = \; \tbinom{n+1}{2} \; = \; \rank  R_{\tau}(-\tau-\zeta)
  \end{equation*}
  for generic $\tau$.  However, nullity is generically small and rank is generically large so
  \begin{equation*}
  \nullity R_{\tau}(\tau+\zeta) \; \geq \; \tbinom{n+1}{2} \; \geq \; \rank  R_{\tau}(-\tau-\zeta)
  \end{equation*}
  for \textit{all} $\tau \in\bC$.  This proves the first inequality in \cref{eq.lem.rz.dim.equal}. A similar argument proves the second.\footnote{Alternatively, the second inequality in \cref{eq.lem.rz.dim.equal} follows from the first because
    \begin{equation*}
      \nullity R_{\tau}(-\tau-\zeta)\;=\;\dim V^{\otimes 2}-\rank  R_{\tau}(-\tau-\zeta) \; \geq \; n^{2}-\tbinom{n+1}{2}\;=\;\tbinom{n}{2}.
    \end{equation*}}
\end{proof}

\subsection{The determinant of $R_\tau(z)$}
\label{ssec.det.r.z}
Since $R(z)$ is a theta operator of order $n^2$ with respect to $\Lambda$, $\det R(z)$ is a theta function of order $n^4$ with respect to $\Lambda$ (\cref{prop.hol.det.zeros}).  The next result improves on this.

\begin{proposition}
  \label{prop.qper} 
  \label{cor.zeros}
  The function $\det R_\tau(z)$ is a theta function of order $n^2$ with respect to $\frac{1}{n}\Lambda$.
  \begin{enumerate}
  \item\label{item.qper.real} $\det R_\tau\big(z+\frac 1n\big) = \det R_\tau(z)$.
  \item\label{item.qper.eta} $\det R_\tau\big(z+\frac 1n \eta\big) = b(z)^{n^2} \det R_\tau(z)$.
  \item\label{item.qper.zeros} $ \det R_\tau(z)$ has $n^2$ zeros in every fundamental parallelogram for $\frac{1}{n}\Lambda$.
  \end{enumerate}
\end{proposition}
\begin{proof}
By  \cref{eq.z+1/n} and  \cref{eq.z+1/n eta}, 
$R_\tau\big( z+\tfrac{1}{n}\big)=(-1)^{n-1} (I \otimes S^{-k}) \, R_\tau(z) \, (S^k \otimes I)$ and 
$R_\tau\big( z+\tfrac{1}{n}\eta\big) =b(z) (I \otimes T^{-1}) \, R_\tau(z) \, (T \otimes I)$.
Since $\dim(V^{\otimes 2})=n^2$ implies 
$\det(-I \otimes I)^{n-1} =  1$, it follows that $\det R_\tau\big(z+\frac 1n\big) = \det R_\tau(z)$ and $\det R_\tau\big(z+\tfrac{1}{n} \eta \big)= b(z)^{n^2} \det R_\tau(z)$.

\cref{item.qper.zeros}
We note that $b(z)^{n^2}=e(-n^3z-B)$ as functions of $z$ for a suitable $B \in \bC$.
Applying \cref{lem.theta.fns} to the function $f(z)=\det R(z)$, with $\eta_1=\frac{1}{n}$, $\eta_2=\frac{1}{n}\eta$, 
$a=b=0$, $c=n^3$, and $d=B$, we see that the number of zeros (counted with multiplicity) in each fundamental parallelogram for 
$\frac{1}{n}\Lambda$ is  $c \eta_1 - a\eta_2= n^3\times \frac{1}{n} = n^2$. 
\end{proof}

\begin{theorem}
\label{prop.nullities}
If $\tau \in \bC- \frac{1}{2n}\Lambda$, then
\begin{enumerate}
\item
\label{item.cor.r.rk.isom}  
$R_\tau(z)$ is an isomorphism if and only if $z \notin \pm \tau + \frac{1}{n}\Lambda$;
\item
\label{item.cor.r.rk.ex} 
$\im R_\tau(\tau+\zeta)=\ker R_\tau(-\tau-\zeta)$ and   $\im R_\tau(-\tau-\zeta)=\ker R_\tau(\tau+\zeta)$ for all $\zeta \in \frac{1}{n}\Lambda$;
\item
\label{item.cor.r.rk.det}  
if $p \in \bC$, then 
$$
\mult_p(\det R_{\tau}(z)) \;= \; \nullity R_{\tau}(p)  \; = \;
\begin{cases}
\tbinom{n+1}{2} & \text{if $p \in \tau+\frac{1}{n}\Lambda$},
      \\
\tbinom{n}{2} & \text{if $p \in -\tau+\frac{1}{n}\Lambda$}   \phantom{\bigg\vert},
      \\
0 & \text{otherwise;}
\end{cases}
$$
\end{enumerate}
If $\tau \in \frac{1}{2n}\Lambda-\frac{1}{n}\Lambda$, then $\nullity R_\tau(\tau) =\nullity R_\tau(-\tau) \ge  \tbinom{n+1}{2}$.
\end{theorem}
\begin{proof}
If $\tau \in \bC- \frac{1}{2n}\Lambda$, then $\tau$ and $-\tau$ are distinct points modulo $\frac{1}{n}\Lambda$.

\cref{lem.rz.dim.ker,pr.hol-op} imply that
\begin{equation*}
\mult_{\tau}(\det R_\tau(z)) \; \ge \; \nullity R_\tau(\tau) \; \ge \; \tbinom{n+1}{2}
\end{equation*} 
and 
\begin{equation*}
\mult_{-\tau}(\det R_\tau(z))  \; \ge \; \nullity R_\tau(-\tau) \;\ge \; \tbinom{n}{2}.
\end{equation*}
But  $ \det R_\tau(z)$ has exactly $n^2= \binom{n+1}{2}+ \binom{n}{2}$ zeros in every fundamental parallelogram for 
$\frac{1}{n}\Lambda$ so the displayed inequalities are equalities, which then implies \cref{item.cor.r.rk.isom} and \cref{item.cor.r.rk.ex}.
Since the rank of $R_\tau(z)$ is the same as that of $R_\tau(z+\zeta)$ (\cref{lem.transf.props.R.tau.z}\cref{item.prop.rank.r.tw}),
part \cref{item.cor.r.rk.det} follows from \cref{item.cor.r.rk.ex} and the inclusions in \cref{eq.lem.rz.dim.ker.incl}.

Finally, if $\tau \in \frac{1}{2n}\Lambda-\frac{1}{n}\Lambda$, then $-\tau=\tau+\zeta$ for some $\zeta \in \frac{1}{n}\Lambda$ so
$ R_\tau(\tau)$ and $R_\tau(-\tau)$ have the same nullity by \cref{lem.transf.props.R.tau.z}\cref{item.prop.rank.r.tw}, and this is $\ge
\tbinom{n+1}{2}$ by \cref{eq.lem.rz.dim.equal}
\end{proof}

\begin{proposition}\label{pr.det-zeros}
For all $\tau \in \bC-\frac{1}{n}\Lambda$, 
 \begin{equation}
   \label{eq:det.U(z)}
   \det R_{\tau}(z) \;=\; 
 \left(  \prod_{\alpha\in\bZ_{n}}\frac{\theta_\alpha (-z-\tau)}{\theta_{\alpha}(-\tau)}\right)^{\frac{n(n-1)}2} 
\left( \prod_{\alpha\in\bZ_{n}}\frac{\theta_\alpha (-z+\tau)}{\theta_{\alpha}(\tau)}\right)^{\frac{n(n+1)}2}.
 \end{equation}
In particular, $\det R_{\tau}(z)$ does not depend on $k$.
\end{proposition}
\begin{proof}
We first prove this under the assumption that $\tau \notin \frac{1}{2n}\Lambda$.

  Let $D(z)$ denote the function on the right-hand side of \cref{eq:det.U(z)}. Both $\det R(z)$ and $D(z)$ are holomorphic functions of $z$.
 It follows from  \Cref{prop.qper} and \cite[Prop.~2.6]{CKS1} that 
$\det R(z)$ and $D(z)$  have the same quasi-periodicity properties with respect to the lattice $\frac{1}{n}\Lambda$.
 It follows from \Cref{prop.nullities}\cref{item.cor.r.rk.det}  and \cite[Prop.~2.6]{CKS1} that $\det R(z)$ and $D(z)$ have the same zeros with the same multiplicities; the ratio $(\det R(z))/D(z)$ is therefore a meromorphic function on the elliptic curve $\bC/\frac{1}{n}\Lambda$ with neither zeros nor poles, and therefore a constant. Since $R(0)=I\otimes I$ by \cref{rmk.R(0)}, $\det R(0)=1=D(0)$. So the constant is $1$.
  
The result is therefore true when $\tau \notin \frac{1}{2n}\Lambda$. 
By continuity, it also holds when $\tau \in  \frac{1}{2n}\Lambda -  \frac{1}{n}\Lambda$.
\end{proof}

\begin{corollary}\label{cor.R.isom}
Let $\tau\in\bC-\frac{1}{n}\Lambda$. Then $R_{\tau}(z)$ is an isomorphism if and only if $z\notin\pm\tau+\frac{1}{n}\Lambda$.
\end{corollary}
\begin{proof}
\cref{pr.det-zeros} tells us that $\det R_{\tau}(z)=0$ if and only if $z\in\pm\tau+\frac{1}{n}\Lambda$. Thus the result follows.
\end{proof}

\subsection{The space of quadratic relations for $Q_{n,k}(E,\tau)$ has dimension $\tbinom{n}{2}$}\label{subsec.det.alt.proof}

This subsection completes the proof that the dimension of $\rel_{n,k}(E,\tau)$ is the same as that of the space of quadratic relations for the polynomial ring on $n$ variables when $\tau\notin(\frac{1}{2n}\Lambda-\frac{1}{n}\Lambda)$.

\subsubsection{Definition of $\rel_{n,k}(E,\tau)$ for all $\tau \in \bC$}
\label{sssec.def.rel.Qnk}
Since \cref{lem.U_tau} ensures that $R_{\tau}(\tau)$ extends in a unique way to a holomorphic function on $\bC$, we can now define, for all $\tau\in\bC$,\index{rel_n,k(E,tau)@$\rel_{n,k}(E,\tau)$}\index{Q_n,k(E,tau)@$Q_{n,k}(E,\tau)$}
\begin{align*}
\rel_{n,k}(E,\tau) & \; := \; \im R_\tau(\tau), 
\\
Q_{n,k}(E,\tau) & \; := \;  \frac{TV}{(\rel_{n,k}(E,\tau))}
\end{align*}
where, if $\tau\in\frac{1}{n}\Lambda$, then $R_{\tau}(\tau)$ is not defined so it should be regarded as the limit $R_{+}(\tau)$ in \cref{eq:mpm1}.\footnote{There are some other ways to extend the definition of $\rel_{n,k}(E,\tau)$ to all $\tau \in \bC$; see \cite[\S3.3]{CKS1} for more discussion.}

\begin{theorem}
\label{thm.quad.relns}
For all $\tau \in \bC-(\frac{1}{2n}\Lambda-\frac{1}{n}\Lambda)$, $\dim\rel_{n,k}(E,\tau) = \tbinom{n}{2}$.
\end{theorem}
\begin{proof}
We proved this in \cite[\S5]{CKS1} for $\tau \in \frac{1}{n}\Lambda$. Suppose $\tau \in \bC-\frac{1}{2n}\Lambda$. Then 
$\rel_{n,k}(E,\tau)$ is the image of $R_{\tau}(\tau)$ and, by \cref{prop.nullities}\cref{item.cor.r.rk.det},
\begin{equation*}
\rank R_{\tau}(\tau)\;=\;\dim V^{\otimes 2}-\nullity R_{\tau}(\tau)\;=\;n^{2}-\tbinom{n+1}{2}\;=\;\tbinom{n}{2}.\qedhere
\end{equation*}
\end{proof}

\subsection{Some twists of $Q_{n,k}(E,\tau)$}
By \cref{lem.transf.props.R.tau.z}\cref{item.prop.rank.r.tw}, $\rank R_\tau(\tau+\zeta) =\rank R_\tau(\tau)$ for all $\zeta \in \frac{1}{n}\Lambda$.
Since $Q_{n,k}(E,\tau)= TV/(\im R_\tau(\tau))$  it is reasonable to ask whether the algebras $TV/(\im R_\tau(\tau+\zeta))$
are, perhaps, ``new'' elliptic algebras.

\begin{proposition}
\label{prop.twist}
Assume $a,b \in \bZ$. Let $\tau\in\bC$, $\zeta =\tfrac{a}{n}+\tfrac{b}{n}\eta$, and $\phi:=S^{-ka}T^{-b} \in \GL(V)$.
\begin{enumerate}
\item\label{item.prop.qnk.tau.tw}
The map $\phi$  extends in a unique way to an algebra automorphism of $TV$ that descends to an automorphism
of $Q_{n,k}(E,\tau)$ (that we also denote by $\phi$) and
\begin{equation*}
\frac{TV}{(\im R_\tau(\tau+\zeta))} \; \cong \; Q_{n,k}(E,\tau)^\phi
\end{equation*}
where $Q_{n,k}(E,\tau)^\phi$\index{Q_n,k(E,tau)^phi@$Q_{n,k}(E,\tau)^\phi$} is the twist of  $Q_{n,k}(E,\tau)$ in the sense of \cite[\S4.1]{CKS1}. 
\item\label{item.prop.gr.morita.equiv}
The categories of $\bZ$-graded left modules over $TV/(\im R_\tau(\tau+\zeta))$ and $Q_{n,k}(E,\tau)$ are equivalent.
\item\label{item.prop.qnk.comparison}
If $k+1$ is a unit in $\bZ_n$, then $TV/(\im R_\tau(\tau+\zeta))$ is isomorphic to $Q_{n,k}(E,\tau+\zeta')$ where $\zeta'=-\frac{kc}{n}-\frac{kd}{n}\eta$ and
$c$ and $d$ are arbitrary integers such that $(k+1)c=a$ and $(k+1)d=b$ in $\bZ_n$.
\end{enumerate}
\end{proposition}
\begin{proof}
We use the convention that, for $\tau\in\frac{1}{n}\Lambda$, $R_{\tau}(\tau)$ and $R_{\tau}(\tau+\zeta)$ mean $R_{+}(\tau)$ and $R_{+}(\tau+\zeta)$, respectively.

\cref{item.prop.qnk.tau.tw}
Let $\fa$ denote the ideal in $TV$ generated by $\im R_\tau(\tau)$; thus  $Q_{n,k}(E,\tau)=TV/\fa$.

Certainly $\phi$ extends in a unique way to an algebra automorphism of $TV$ that we continue to denote by $\phi$.
By \cref{lem.transf.props.R.tau.z}\cref{item.prop.s.t.comm}, the operator $\phi \otimes \phi$ on $V\otimes V$ commutes with $R_\tau(\tau)$ so $\im R_\tau(\tau)$ is
 stable under the action of $\phi \otimes \phi$. 
The automorphism $\phi$ of $TV$ therefore preserves $\fa$, i.e., $\phi(\fa)=\fa$, and so descends to an automorphism of 
$Q_{n,k}(E,\tau)$ (that we continue to denote by $\phi$).
We define the algebra $Q_{n,k}(E,\tau)^\phi$ as in \cite[\S4]{CKS1}. 

By \cite[Lem.~4.1]{CKS1}, $(TV/\fa)^\phi \cong  TV/\phi'(\fa)$ where $\phi'(\fa)$ denotes the image of $\fa$ 
under the action of 
the linear map $\phi'$ that is $I \otimes \phi \otimes \cdots \otimes \phi^{d-1}$ on each $V^{\otimes d}$. Since our $\fa$ is generated by its degree-two
component, $\fa_2$, $\phi'(\fa)$ is generated by $(I \otimes \phi)(\fa_2)$.
But 
\begin{align*}
(I \otimes \phi)(\fa_2) & \; = \; (I \otimes \phi)R_\tau(\tau)(V^{\otimes 2})
\\
& \; = \; (I \otimes \phi)R_\tau(\tau) (\phi^{-1} \otimes I)(V^{\otimes 2})
\\
& \; = \;  \im R_\tau(\tau+\zeta)     \qquad \text{by \cref{cor.R.transl.props}}.
\end{align*}
Hence
\begin{equation*}
\frac{TV}{(\im R_\tau(\tau+\zeta))} \;= \; \frac{TV}{(I \otimes \phi)(\fa_2)} \; =\;  \frac{TV}{\phi'(\fa)}  \; \cong \; Q_{n,k}(E,\tau)^\phi.
\end{equation*}

\cref{item.prop.gr.morita.equiv}
This is an immediate  consequence of \cite[Cor.~8.5]{ATV2}.

\cref{item.prop.qnk.comparison}
Assume $k+1$ is a unit in $\bZ_n$ (equivalently, $k'+1$ is a unit in $\bZ_n$).

Let $c,d \in \bZ$ be such that $c=(k+1)^{-1}a$ and $d=(k+1)^{-1}b$ in $\bZ_n$.
Define $\sigma:=S^{-kc}T^{-d}$. Then $\phi =\s^{k+1}$, and 
if $\zeta'=-\frac{kc}{n} - \frac{kd}{n}$,
\begin{equation*}
Q_{n,k}(E,\tau)^{\phi} \; = \;  Q_{n,k}(E,\tau)^{\s^{k+1}} \; \cong \;  Q_{n,k}(E,\tau + \zeta') 
\end{equation*}
by \cite[Thm.~4.3]{CKS1}.
\end{proof}

\subsection{The relation between $R_\tau(z)$ and Odesskii's $R$-matrix}
\label{subsec.rel.r.od}

Odesskii defines a family of operators that we will denote by $R^{\rm Od}(z)$\index{R^Od(z)@$R^{\rm Od}(z)$} at \cite[p.~1145]{Od-survey}.  
His family is related to ours by the formula
\begin{equation}
\label{eq:Od.p(z)}
R_\tau(z) \; =\; \left( \prod_{\a \in \bZ_n} \frac{\theta_\a(-z+\tau)}{\theta_\a(\tau)}\right) \, R^{\rm Od}(z) .  
\end{equation}

Odesskii says that $R^{\rm Od}(z)$ satisfies (\ref{QYBE2}) but no proof is provided or indicated.  
Nevertheless, it is easy to see that $R^{\rm Od}(z)$ satisfies (\ref{QYBE2}) if and only if $R_{\tau}(z)$ does so Odesskii's statement is
correct. However, $R^{\rm Od}(z)$ is not defined for all $z\in \bC$ so one must place some restrictions on the points 
$u,v \in \bC$ when saying that $R^{\rm Od}(z)$ satisfies (\ref{QYBE2}); in contrast, $R_\tau(z)$ satisfies (\ref{QYBE2}) for all $u,v \in \bC$.

There were other reasons we chose $R_{\tau}(z)$ as we did. We wanted $R_\tau(0)$ to be the identity (\cref{rmk.R(0)}).  We also wanted the equality $\lim_{\tau \to 0} R_{\tau}(m\tau)=\sym_m$ in \cref{le.ulim0} and the equalities in \cref{pr.flim}; other choices of $R_{\tau}(z)$ 
would only give those equalities up to a non-zero scalar multiple.

\begin{proposition}
\label{prop.O's.formula.locn.zeros}
The determinant of the operator $R^{\rm Od}(z)$ defined at \cite[p.~1145]{Od-survey} is
  \begin{equation}
  \label{eq:almost.det.R(z)}
(-1)^{\frac{n^{2}(n-1)}{2}} e\left(    \tfrac{n^3(n-1)}{2} \, \tau   \right)
 \left( \frac{\theta_0(-z-\tau) \cdots \theta_{n-1}(-z-\tau)}{\theta_0(-z+\tau) \cdots \theta_{n-1}(-z+\tau)} \right)^{\!\frac{n(n-1)}{2}}.
  \end{equation}
  \end{proposition}
  \begin{proof}
  It follows from \cref{eq:det.U(z)} that  
\begin{equation*}
\det R^{\rm Od}(z) \;=\;   
\left(  \prod_{\a \in \bZ_n} \frac{\theta_\a(\tau)}{  \theta_\a(-z+\tau)}\right)^{\!n^2} \,
 \left(  \prod_{\alpha\in\bZ_{n}}\frac{\theta_\alpha (-z-\tau)}{\theta_{\alpha}(-\tau)}\right)^{\!\frac{n(n-1)}2} 
\left( \prod_{\alpha\in\bZ_{n}}\frac{\theta_\alpha (-z+\tau)}{\theta_{\alpha}(\tau)}\right)^{\!\frac{n(n+1)}2}.
\end{equation*}
We leave to the reader the pleasant task of showing this equals the expression in \cref{eq:almost.det.R(z)}.
\end{proof}

The formula for  $\det R^{\rm Od}(z)$  at \cite[p.~1145]{Od-survey} omits the term 
$(-1)^{\frac{n^{2}(n-1)}{2}} e\left(    \tfrac{n^3(n-1)}{2} \, \tau   \right)$.

\section{The Hilbert series of $Q_{n,k}(E,\tau)$}
\label{sec.hilb.series}

This section shows $Q_{n,k}(E,\tau)$ has the same Hilbert series as the polynomial ring $SV$  for all $\tau\in(\bC-\bigcup_{m\geq 1}\frac{1}{m}\Lambda)\cup\frac{1}{n}\Lambda$.

\subsection{Introduction}
\Cref{se.det} showed that, for all $\tau\in\bC-(\frac{1}{2n}\Lambda-\frac{1}{n}\Lambda)$, the degree-two component of $Q_{n,k}(E,\tau)$ has the ``right'' dimension, namely $\tbinom{n+1}{2}$, by showing in \cref{prop.nullities}\cref{item.cor.r.rk.ex} that $Q_{n,k}(E,\tau)$'s space of quadratic relations, which is, by definition, the image of $R_\tau(\tau)$, is equal to the kernel of $R_\tau(-\tau)$.  A similar idea is used in this section: we show, for all $\tau\in\bC-\bigcup_{m=1}^{d}\frac{1}{mn}\Lambda$, that the space of degree-$d$ relations
for $Q_{n,k}(E,\tau)$, which is $\sum_{s+t+2=d} V^{\otimes s} \otimes \im R_\tau(\tau) \otimes V^{\otimes t}$, is the kernel of the operator $F_d(-\tau): V^{\otimes d} \to V^{\otimes d}$ defined in \cref{defn.F_d} below.

The analogy with the polynomial ring $SV$ is helpful. The space of degree-$d$ relations for $SV$, i.e., the kernel of the natural map $V^{\otimes d} \to S^dV$, is the kernel of the symmetrization operator $ \sum_{\s \in S_d} \s$ acting in the natural way on $V^{\otimes d}$.  \cref{pr.flim} shows that $\lim_{\tau \to 0} F_d(-\tau)$ is a non-zero scalar multiple of $\sum_{\s \in S_d} \s$.

\subsection{The linear operators $T_d$, $S_{d \to 1}$, $S^{\rm rev}_{d \to 1}$, $F_d$, etc.,  on $V^{\otimes d}$}
\label{se:symmetrizers}
The results in this subsection apply to any family of linear operators $R(z):V^{\otimes 2} \to V^{\otimes 2}$, $z\in\bC$, satisfying (\ref{QYBE2}), which is the parametrized braid relation
\begin{equation*}
R(u)_{12}R(u+v)_{23}R(v)_{12} \;=\; R(v)_{23}R(u+v)_{12}R(u)_{23}
\end{equation*}
for all $u,v\in\bC$.

If $t_p,\ldots,t_q \in \bC$ we will write $\Sigma_p^q:=t_p+\cdots + t_q$\index{Sigma_p^q@$\Sigma_p^q$}. Let $i$, $j$, and $d$ be positive integers with $i\le j\le d$.  
We will use the following operators on $V^{\otimes d}$:\index{S_ij, S^rev_ij@$S_{i \to j}$, $S^{\rm rev}_{i \to j}$}
\begin{align*}
S_{i\to j}(t_i,\cdots,t_{j-1})    & \; :=\;  R\big(\Sigma_i^{j-1}\big)_{\!i,i+1} \ldots R\big(\Sigma_{q}^{j-1}\big)_{\! q,q+1}
\ldots R\big(t_{j-1}\big)_{\! j-1,j},
\\
S_{j\to i}(t_{j-1}, \ldots, t_{i}) & \; :=\; R\big(\Sigma_{i}^{j-1}\big)_{\! j-1,j} \ldots R\big(\Sigma_{i}^{q}\big)_{\! q,q+1} 
\ldots R\big(t_{i}\big)_{\! i,i+1},
\\
S_{i\to j}^{\rm rev}(t_i,\cdots,t_{j-1})   & \; :=\; R\big(t_i\big)_{\!i,i+1}  \ldots  R\big(\Sigma_{i}^{q}\big)_{\! q,q+1}
\ldots R\big(\Sigma_i^{j-1}\big)_{\! j-1,j},
\\
S_{j\to i}^{\rm rev}(t_{j-1}, \ldots, t_{i}) &  \; :=\; R\big(t_{j-1}\big)_{\! j-1,j} \ldots R\big(\Sigma^{j-1}_{q}\big)_{\! q,q+1}
 \ldots  R\big(\Sigma^{j-1}_i\big)_{\! i,i+1},
\end{align*}
with the convention that these are the identity operators when $i=j$. For example,
\begin{equation*}
	S_{1\to 4}(t_{1},t_{2},t_{3})\;=\;R(t_{1}+t_{2}+t_{3})_{1,2}R(t_{2}+t_{3})_{2,3}R(t_{3})_{3,4}.
\end{equation*}
Each of these is a theta operator (\cref{def.thetaop}) of order $(i-j)n^2$ with respect to $\Lambda$. 
We also define\index{T_d(z_1,...,z_d-1)@$T_d(z_{1},\ldots,z_{d-1})$}
\begin{equation}\label{eq.def.td}
  T_{d}(z_{1},\ldots,z_{d-1}) \; := \; S_{2\to 1}(z_{1})S_{3\to 1}(z_{1},z_{2})\cdots S_{d\to 1}(z_{1},\ldots,z_{d-1}).
\end{equation}
and\index{F_d@$F_d(z)$}
\begin{equation}
\label{defn.F_d}
	F_{d}(z) \; := \; T_{d}(z,\ldots,z).
\end{equation}
When $d=0,1$ we declare that these operators are the identity.

The choice of labeling for the arguments in the $S$-operators allows the elegant factorizations 
\begin{align*}
S_{i\to k}(t_i,\cdots,t_{k-1})  & \; = \;  S_{i\to j}(t_i,\cdots,t_{j-2},\Sigma_{j-1}^{k-1}\big) S_{j\to k}(t_j,\cdots,t_{k-1}),
\\
S_{k \to i}(t_{k-1},\ldots,t_{i})  &  \;=\; S_{k \to j}\big(t_{k-1},\ldots,t_{j+1},\Sigma_{i}^{j}\big) S_{j\to i}\big(t_{j-1},\ldots,t_{i}\big), 
\\
S^{\rm rev}_{i\to k}(t_i,\cdots,t_{k-1})  & \; = \;   S^{\rm rev}_{i\to j}(t_i,\cdots,t_{j-1}) S^{\rm rev}_{j\to k}(\Sigma_{i}^{j},t _{j+1},\cdots,t_{k-1}),
\\
S^{\rm rev}_{k \to i}(t_{k-1},\ldots,t_{i})  &  \;=\;  S_{k\to j}^{\rm rev}(t_{k-1}, \ldots, t_{j}) 
S^{\rm rev}_{j\to i}\big(\Sigma^{k-1}_{j-1}, t_{j-2},\ldots,t_{i}\big),
\end{align*}
when $1\le i\le j\le k$. These factorizations make no use of the Yang-Baxter equation.
We also note that 
\begin{align*}
S_{i\to j}(t_i,\cdots,t_{j-1})    & \; =\;     S_{i \to j}^{\rm rev}\big(\Sigma_i^{j-1}, -t_{i},\ldots , -t_{j-2}\big).
\numberthis  \label{eq.S.rev.2}
\\
S_{j\to i}(t_{j-1}, \ldots, t_{i}) &  \;=\; S_{j \to i}^{\rm rev}\big(\Sigma_i^{j-1}, -t_{j-1},\ldots , -t_{i+1}\big)
\numberthis  \label{eq.S.rev.1}
\end{align*}
These equalities make no use of the Yang-Baxter equation either.

Recall the notation in \cref{sect.notn}: $T^L_{d-1}$ (resp., $T^R_{d-1}$) denotes $T_{d-1}$ applied to the left-most (resp., right-most) 
  $d-1$ tensorands of $V^{\otimes d}$. The identities in the next result will be used repeatedly in   subsequent sections.

\begin{lemma}\label{le.f-id}
For all $d \ge 2$ and all $z_{1},\ldots,z_{d-1}\in\bC$,
\begin{align*}
T_{d}(z_{1},&\ldots,z_{d-1}) 
\\
&\;=\; T^{L}_{d-1}(z_{1},\ldots,z_{d-2})S_{d\to 1}(z_{1},\ldots,z_{d-1})
\;=\; T^{R}_{d-1}(z_{2},\ldots,z_{d-1})S_{1\to d}(z_{d-1},\ldots,z_{1})
\\
&\;=\; S^{\mathrm{rev}}_{1\to d}(z_{1},\ldots,z_{d-1})T^{L}_{d-1}(z_{2},\ldots,z_{d-1})
\;=\; S^{\mathrm{rev}}_{d\to 1}(z_{d-1},\ldots,z_{1})T^{R}_{d-1}(z_{1},\ldots,z_{d-2}).
\end{align*} 
\end{lemma}
\begin{proof}
(1)
The equality $T_{d}(z_{1},\ldots,z_{d-1})= T^L_{d-1}(\cdots)S_{d \to 1}(\cdots)$  follows at once from the definition of $T_{d}(z_{1},\ldots,z_{d-1})$.
	
(2) We will now show that $T_{d}(z_{1},\ldots,z_{d-1})S^{\mathrm{rev}}_{1\to d}(z_{1},\ldots,z_{d-1})T^{L}_{d-1}(z_{2},\ldots,z_{d-1})$.
	
We first replace each factor $S_{i\to 1}(z_{1},\ldots,z_{i-1})$ in \cref{eq.def.td} by
	\begin{equation*} 
		R(z_{1}+\cdots+z_{i-1})_{i-1,i}S_{i-1\to 1}(z_{2},\ldots,z_{i-1}).
	\end{equation*}
Since $S_{i-1\to 1}(z_{2},\ldots,z_{i-1})$ acts on the left-most $i-1$ tensorands of $V^{\otimes d}$, it commutes with $R(z_{1}+\cdots+z_{j-1})_{j-1,j}$ whenever $i<j$. Therefore
	\begin{align*}
		T_{d}(z_{1},\ldots,z_{d-1})
		&\; =\; R(z_{1})_{12}\cdot R(z_{1}+z_{2})_{23}S_{2\to 1}(z_{2})\cdots R(z_{1}+\cdots+z_{d-1})_{d-1,d}S_{d-1\to 1}(z_{2},\ldots,z_{d-1})\\
		&\;= \; R(z_{1})_{12}\cdots R(z_{1}+\cdots+z_{d-1})_{d-1,d}\cdot S_{2\to 1}(z_{2})\cdots S_{d-1\to 1}(z_{2},\ldots,z_{d-1})\\
		&\; =\; S^{\mathrm{rev}}_{1\to d}(z_{1},\ldots,z_{d-1})T^{L}_{d-1}(z_{2},\ldots,z_{d-1}).
	\end{align*}
	
(3)
We now prove $T_{d}(z_{1},\ldots,z_{d-1}) = T^{R}_{d-1}(\cdots)S_{1\to d}(\cdots)$ by induction on $d$. 
The case $d=2$ is trivial, so we assume the equality holds for all integers $\le d$ and prove it for $d+1$.
  
By (1) and the induction hypothesis, 
	\begin{align*}
		T_{d+1}(z_{1},\ldots,z_{d})
		&\;=\;T^{L}_{d}(z_{1},\ldots,z_{d-1})S_{d+1\to 1}(z_{1},\ldots,z_{d})\numberthis\label{eq.lem.ts.tss}\\
		&\;=\;T_{d-1}^{M}(z_{2},\ldots,z_{d-1})S_{1\to d}(z_{d-1},\ldots,z_{1})S_{d+1\to 1}(z_{1},\ldots,z_{d})
	\end{align*}
	where $T_{d-1}^{M}$ denotes $T_{d-1}$ applied to the middle $d-1$ tensorands of $V^{\otimes (d+1)}$. 
	
	The product $S_{1\to d}(z_{d-1},\ldots,z_{1}) S_{d+1\to 1}(z_{1},\ldots,z_{d}) $ equals
	\begin{align*}
		  & S_{1\to d}(z_{d-1},\ldots,z_{1})R(z_{1}+\cdots+z_{d})_{d,d+1}S_{d\to 1}(z_{2},\ldots,z_{d})
		 \numberthis\label{eq.lem.ts.srs}
		 \\
		\;=\; &S_{1\to d-1}(z_{d-1},\ldots,z_{3},z_{1}+z_{2}) 
		 \, R(z_{1})_{d-1,d} \, R\bigg(\sum_{i=1}^d z_{i}\bigg)_{\!d,d+1}R\bigg(\sum_{i=2}^d z_{i}\bigg)_{\!d-1,d} 
		\, S_{d-1\to 1}(z_{3},\ldots,z_{d}).
	\end{align*}
	By \Cref{QYBE2}, the product of the three $R$'s in the middle is
  \begin{equation}\label{eq:rdd}
    R(z_{2}+\cdots+z_{d})_{d,d+1}R(z_{1}+\cdots+z_{d})_{d-1,d}R(z_{1})_{d,d+1}.
  \end{equation}
	Since $R_{d,d+1}$ commutes with $S_{1\to d-1}$ and $S_{d-1\to 1}$, \cref{eq.lem.ts.srs} equals
	\begin{equation*}
		R(z_{2}+\cdots+z_{d})_{d,d+1}S_{1\to d-1}(z_{d-1},\ldots,z_{3},z_{1}+z_{2})R(z_{1}+\cdots+z_{d})_{d-1,d}S_{d-1\to 1}(z_{3},\ldots,z_{d})R(z_{1})_{d,d+1}.
	\end{equation*}
	The product of the three factors in the middle has the same form as the second line of 
	\cref{eq.lem.ts.srs}, so we can repeat the procedure. Eventually we see that \cref{eq.lem.ts.srs} equals
	$S_{d+1\to 2}(z_{2},\ldots,z_{d})S_{1\to d+1}(z_{d},\ldots,z_{1})$.
	Hence \cref{eq.lem.ts.tss} equals
	\begin{equation*}
		T_{d-1}^{M}(z_{2},\ldots,z_{d-1})S_{d+1\to 2}(z_{2},\ldots,z_{d})S_{1\to d+1}(z_{d},\ldots,z_{1})\;=\;T^{R}_{d}(z_{2},\ldots,z_{d})S_{1\to d+1}(z_{d},\ldots,z_{1})
	\end{equation*}
	as desired.
	
(4)
An argument like that in (2) shows that $T^{R}_{d-1}(\cdots)S_{1\to d}(\cdots) = S^{\mathrm{rev}}_{d\to 1}(\cdots)T^{R}_{d-1}(\cdots)$.
\end{proof}

\begin{lemma}\label{le.shfl.gen}
	For every $1\le i\le d-1$,     
	\begin{equation*}
		T_d(z_{1},\ldots,z_{d-1}) \;=\;  R(z_{i})_{i,i+1}Q_{i} \; =\; Q_{i}'R(z_{d-i})_{i,i+1}
	\end{equation*}
	where $Q_{i}$ and $Q_{i}'$ are products of terms $R(w)_{j,j+1}$ for various integers $1\leq j\leq d-1$ and partial sums $w$ of $z_{1},\ldots,z_{d-1}$.
\end{lemma}
\begin{proof}
We argue by induction on $d$. The case $d=2$ is trivial. 
Assuming the claim up to and including $d-1$, the first equality follows from the first and second identities in \Cref{le.f-id}, 
and second equality follows from the third and fourth identities in \Cref{le.f-id}.
\end{proof}

\begin{proposition}\label{le.shfl}
Assume $1\le i\le d-1$.
\begin{enumerate}
\item\label{item.shfl.Q.m}
$F_{d}(-\tau)=Q R(-\tau)_{i,i+1} =R(-\tau)_{i,i+1}Q'$ for some $Q$ and $Q'$ that are products of terms $R(m\tau)_{j,j+1}$ for various integers $1\leq j\leq d-1$ and $m$.
\item\label{item.shfl.Q.p}
$F_{d}(\tau)=Q R(\tau)_{i,i+1} =R(\tau)_{i,i+1}Q'$ for some $Q$ and $Q'$ that are products of terms $R(m\tau)_{j,j+1}$ for various integers $1\leq j\leq d-1$ and $m$.
\item\label{item.shfl.im.int}
If $R(\tau)R(-\tau)=0$, then 
\begin{equation*}
\im F_{d}(-\tau) \; \subseteq \;  \bigcap_{s+t+2=d} V^{\otimes s} \otimes \ker R(\tau) \otimes V^{\otimes t}.
\end{equation*}
\item\label{item.shfl.ker.sum}
If $R(-\tau)R(\tau)=0$, then 
\begin{equation*}
\ker F_{d}(-\tau) \; \supseteq \;  \sum_{s+t+2=d} V^{\otimes s} \otimes \im R(\tau) \otimes V^{\otimes t}.
\end{equation*}
\end{enumerate}
\end{proposition}
\begin{proof}
Both \cref{item.shfl.Q.m} and \cref{item.shfl.Q.p} are special cases of \cref{le.shfl.gen}.

\cref{item.shfl.im.int}
Since $F_{d}(-\tau)=R(-\tau)_{i,i+1}Q'$, $R(\tau)_{i,i+1}F_d(-\tau)=0$; i.e., $\im F_d(-\tau) \subseteq \ker R(\tau)_{i,i+1}$.
But
\begin{equation*}
	\ker R(\tau)_{i,i+1}\;=\;V^{\otimes (i-1)} \otimes \ker R(\tau) \otimes V^{\otimes (d-i-1)}
\end{equation*} 
so the result follows.

\cref{item.shfl.ker.sum}
Since $F_{d}(-\tau)=Q R(-\tau)_{i,i+1}$, $F_{d}(-\tau)R(\tau)_{i,i+1}=0$; i.e., $\ker F_d(-\tau) \supseteq \im R(\tau)_{i,i+1}$.
\end{proof}

In proving that $Q_{n,k}(E,\tau)$ has the ``right'' Hilbert series we will show that the inclusions in parts \cref{item.shfl.im.int} and \cref{item.shfl.ker.sum} of \cref{le.shfl}
are equalities when $R(z)$ is the  operator in \cref{eq:odr}.

\subsection{The limit $F_d(\pm \tau)$ as $\tau \to 0$}\label{subse.lim}

To determine the Hilbert series of $Q_{n,k}(E,\tau)$ and $Q_{n,k}(E,\tau)^!$ we must understand the limits of 
$F_{d}(\pm\tau)$ as $\tau\to 0$.

\begin{proposition}\label{pr.flim}
If $d$ is an integer $\ge 2$, then
  \begin{align*}
    \lim_{\tau\to 0}F_{d}(-\tau)  & \; = \; \prod_{m=1}^{d-1}m!\cdot\sum_{\sigma\in S_d} \sigma 
 \qquad\qquad \text{and} 
  \\
    \lim_{\tau\to 0}F_{d}(\tau) &  \; =\;  \prod_{m=1}^{d-1}m!\cdot\sum_{\sigma\in S_d} \sgn(\sigma)\sigma
  \end{align*}
where the symmetric group $S_d$\index{S_d@$S_{d}$} acts on $V^{\otimes d}$ by permuting tensorands.
\end{proposition}
\begin{proof}
  We prove the proposition for $F_{d}(-\tau)$. The argument for $F_{d}(\tau)$  is virtually identical. 

We argue by induction on $d$. Since $F_2(z)=R(z)$, the $d=2$ case is a consequence of \Cref{cor.sym}. 
Assume $d \ge 3$. Since $ F_{d}(-\tau) = T_d(-\tau,\ldots,-\tau)$, it follows from \Cref{le.f-id,le.ulim0}
that
  \begin{align*}\label{eq:fusf}
\lim_{\tau \to 0}    F_{d}(-\tau) & \; =\;  \lim_{\tau \to 0}  R(-\tau)_{12}\cdots R(-(d-1)\tau)_{d-1,d}F_{d-1}^{L}(-\tau) \numberthis
    \\
               & \; = \; \sym^{1,2}_{-1}\cdots \sym^{d-1,d}_{-(d-1)} \cdot  \lim_{\tau \to 0}  F_{d-1}^{L}(-\tau)
  \end{align*}
  where the superscripts on the operators $\sym_m$ indicate which tensorands of $V^{\otimes d}$ they apply to.
  By the induction hypothesis, the factor $ \lim_{\tau \to 0}  F^{L}_{d-1}(-\tau)$ in \cref{eq:fusf} is the map
  \begin{equation}\label{eq:shfl(d-1)1}
    v_1\cdots v_{d-1} v_d \; \mapsto \;  \prod_{m=1}^{d-2}m!\cdot\sum_{\sigma\in S_{d-1}}v_{\sigma(1)}\cdots v_{\sigma(d-1)} v_d
  \end{equation}
where we have suppressed tensor symbols for readability. 

\textbf{Claim:} the product of the $\sym$ factors in \Cref{eq:fusf} sends the sum part of \Cref{eq:shfl(d-1)1} to
\begin{equation*}
  (d-1)!\sum_{\sigma\in S_d} v_{\sigma(1)}\cdots v_{\sigma(d)}. 
\end{equation*}
To see this, we will count, for every $t$ and every $\sigma'\in S_{d-1}$, the number of times the term
\begin{equation}\label{eq:dinter}
  v_{\s'(1)}\cdots v_{\s'(t-1)} \, v_d \, v_{\s'(t)}\cdots v_{\s'(d-1)}
\end{equation}
appears in 
\begin{equation}\label{eq:syms}
\sym^{1,2}_{-1}\cdots \sym^{d-1,d}_{-(d-1)} \Bigg(\sum_{\sigma\in S_{d-1}}v_{\sigma(1)}\cdots v_{\sigma(d-1)}v_d\Bigg).
\end{equation}

If we write $\sym_{-m} = I + mP$ using $I=\id_{V}$ and the flip $P\in \End(V\otimes V)$, then $\sym^{m,m+1}_{-m} = I^{\otimes d} + mP_{m,m+1}$,
 where $P_{m,m+1}$ interchanges the $m^{\rm th}$ and $(m+1)^{\rm th}$ tensorands
 in $V^{\otimes d}$. 
 Since $v_d$ starts out on the extreme right in \Cref{eq:syms} and crosses leftward past $d-t$ tensorands to reach its position in \Cref{eq:dinter}, the latter must be the result of applying the summands
\begin{equation*}
  (d-1)P_{d-1,d},\; (d-2)P_{d-2,d-1},\; \ldots,\; tP_{t,t+1} 
\end{equation*}
of the rightmost $(d-t)$ $\sym$ operators in \Cref{eq:syms}. These yield a factor of
\begin{equation}\label{eq:fact-incomplete}
(d-1)\cdots   (t+1)t \;=  \; \frac{(d-1)!}{(t-1)!}
\end{equation}
in front of \Cref{eq:dinter}, and we must show that the remaining $\sym$ operators
\begin{equation}
\label{eq:sym.factors}
  \sym^{1,2}_{-1}\ldots \sym^{t-1,t}_{-(t-1)}
\end{equation}
contribute the missing $(t-1)!$ factor to produce the requisite $(d-1)!$. 
Only the $I^{\otimes d}$ term in
\begin{equation*}
  \sym^{t-1,t}_{-(t-1)} \; = \; I^{\otimes d} \, + \,  (t-1)P_{t-1,t}
\end{equation*}
can contribute a \Cref{eq:dinter} term because $P_{t-1,t}$ would slide $v_d$ further to the left. On the other hand,
each of the product of the remaining $\sym$ factors in the product \cref{eq:sym.factors}, namely
\begin{equation*}
  \sym^{1,2}_{-1}\cdots \sym^{t-2,t-1}_{-(t-2)},
\end{equation*}
contributes terms of the form  \Cref{eq:dinter} via both its $I^{\otimes d}$ and $P$ terms, whence each
\begin{equation*}
  \sym^{j,j+1}_{-j} \;=\;  I^{\otimes d} + j P_{j,j+1},\quad 1\le j\le t-2,
\end{equation*}
supplements \Cref{eq:fact-incomplete} with an additional factor of $1+j$  scaling the \Cref{eq:dinter} term. 
The overall coefficient of \Cref{eq:dinter} is therefore
\begin{equation*}
  \frac{(d-1)!}{(t-1)!} \cdot (t-1)\cdots 2 \; =\;  (d-1)!
\end{equation*}
as claimed.
\end{proof}

\subsection{Computation of the Hilbert series for $Q_{n,k}(E,\tau)$}
\label{sect.H.series}

\begin{theorem}\label{thm.dim.rel.sp}
Let $d\geq 2$. Assume $\tau\in\bC-\bigcup_{m=1}^{d}\frac{1}{mn}\Lambda$.
  \begin{enumerate}
  \item\label{thm.dim.rel.sp.ker} We have
    \begin{equation}\label{eq.dim.rel.sp.ker}
      \ker F_{d}(-\tau)\; = \; \sum_{s+t+2=d}V^{\otimes s}\otimes\im R_{\tau}(\tau) \otimes V^{\otimes t}
    \end{equation}
    and its dimension is the same as the dimension of the space of degree-$d$ relations of a polynomial algebra in $n$ variables,
    namely $n^{d}-\tbinom{n+d-1}{d}$.
  \item\label{thm.dim.rel.sp.im} We have
    \begin{equation}\label{eq.dim.rel.sp.im}
      \im F_{d}(-\tau)\;=\;\bigcap_{s+t+2=d}V^{\otimes s}\otimes\ker R_{\tau}(\tau) \otimes V^{\otimes t}
    \end{equation}
    and its dimension is the same as the dimension of the degree-$d$ component of a polynomial algebra in $n$ variables, 
    namely $\tbinom{n+d-1}{d}$.
  \end{enumerate}
\end{theorem}

We assume that $\tau\in\bC-\bigcup_{m=1}^{d}\frac{1}{mn}\Lambda$ until the end of the proof, i.e., $\pm\tau,\pm 2\tau,\ldots,\pm d\tau\notin\frac{1}{n}\Lambda$.

When $d=2$, \cref{thm.dim.rel.sp} follows from \cref{prop.nullities} since $F_2(\tau)=R(\tau)$. We now argue by induction on $d$.

\subsubsection{The operators $G_{\tau}(z)$}
\label{sect.defn.Gz}

Taking $(z_1,\ldots,z_{d-1})=(z,-\tau,\ldots,-\tau)$ in \cref{le.f-id}  we see that
\begin{align*}
	T_{d}(z,-\tau,\ldots,-\tau)
	&\;=\;   S^{\mathrm{rev}}_{1\to d}(z,-\tau,\ldots,-\tau)T^{L}_{d-1}(-\tau,\ldots,-\tau)\\
	&\; =\; T^{R}_{d-1}(-\tau,\ldots,-\tau)S_{1\to d}(-\tau,,\ldots,-\tau,z)
\end{align*}
which implies that the operator
\begin{equation}\label{eq.comp.r.subsp}
  S^{\mathrm{rev}}_{1\to d}(z,-\tau,\ldots,-\tau) \; = \; R(z)_{12}R(z-\tau)_{23}\cdots R(z-(d-2)\tau)_{d-1,d}
\end{equation}
on $V^{\otimes d}$ restricts to a linear map\index{G_tau(z)@$G_\tau(z)$}
\begin{equation}\label{eq:fww'}
G_{\tau}(z):  \im F_{d-1}(-\tau)\otimes V \, \longrightarrow \, V\otimes\im F_{d-1}(-\tau).
\end{equation}
Since $T_{d}(z,-\tau,\ldots,-\tau)=S^{\mathrm{rev}}_{1\to d}(z,-\tau,\ldots,-\tau)T^{L}_{d-1}(-\tau,\ldots,-\tau)$,
\begin{equation}
\label{eq:imG.imF}
\im G_{\tau}(z)  \;=\;  \im T_{d}(z,-\tau,\ldots,-\tau).
\end{equation} 
In particular,  $\im G_{\tau}(-\tau) =\im F_{d}(-\tau)$. 

Since $R_\tau(z)$ (resp., $\det R_\tau(z)$) is a theta operator (resp., function) of order $n^2$ with respect to $\Lambda$
(resp., $\frac{1}{n}\Lambda$), the operator (resp., determinant of the operator) in \cref{eq.comp.r.subsp}
is a theta operator (resp., function) of order $(d-1)n^2$ with respect to $\Lambda$ (resp., $\frac{1}{n}\Lambda$).

\subsubsection{The ``determinant'' of $G_\tau(z)$}
By the induction hypothesis, $\rank F_{d-1}(-\tau)=\tbinom{n+d-2}{d-1}$ so
\begin{equation*}
\dim(\im F_{d-1}(-\tau)\otimes V) \;=\;  \dim(V\otimes\im F_{d-1}(-\tau)) \;=\; n\tbinom{n+d-2}{d-1}.
\end{equation*}
We fix arbitrary bases for the subspaces $\im F^L_{d-1}(-\tau)$ and $\im F^R_{d-1}(-\tau)$ of $V^{\otimes d}$
and write $\det G_{\tau}(z)$ for the determinant of the matrix for $G_{\tau}(z)$ with respect to those bases; 
although $\det G_{\tau}(z)$ depends on the choice of bases, the location and multiplicities of its zeros do not
(see \cref{rmk.det}).

\begin{proposition}\label{lem.gz.zeros}
  $\det G_{\tau}(z)$ is a theta function with respect to $\frac{1}{n}\Lambda$ and has exactly $(d-1)n\tbinom{n+d-2}{d-1}$ zeros in every fundamental parallelogram for $\frac{1}{n}\Lambda$.
\end{proposition}
\begin{proof}
Let $W$ and $W'$ denote the  domain and codomain in \Cref{eq:fww'}.
By the induction hypothesis,
\begin{equation*}
\dim W \;=\; \dim W' \;=\;    n\tbinom{n+d-2}{d-1}.
\end{equation*}
As remarked above,  the operator in \cref{eq.comp.r.subsp} is a theta operator of order $(d-1)n^2$ with respect to $\Lambda$. 
It now follows from \Cref{prop.hol.det.zeros} applied to $A(z)=G_\tau(z)$ that $\det G_{\tau}(z)$ is a theta function of order
$(d-1)n^3\tbinom{n+d-2}{d-1}$ with respect to $\Lambda$. However, because the determinant of the operator in \cref{eq.comp.r.subsp}
 is a theta function with respect to $\frac{1}{n}\Lambda$ so is $\det G_\tau(z)$, and it has exactly 
 $(d-1)n\tbinom{n+d-2}{d-1}$ zeros in every fundamental parallelogram for $\frac{1}{n}\Lambda$.
(Note that $\det G_{\tau}(z)$ does not vanish identically 
because each factor in \cref{eq.comp.r.subsp} is an isomorphism for all but finitely many $z$.)
\end{proof}

\begin{lemma}\label{lem.dim.ker.ineq.p}
For all $m=1,\ldots,d-2$, $G_{\tau}(m\tau)=0$.
\end{lemma}
\begin{proof}
By definition, $G_\tau(m\tau)$ is the restriction of $R(m\tau)_{12}R((m-1)\tau)_{23}\ldots R((m-d+2)\tau)_{d-1,d}$ to 
$ \im F_{d-1}(-\tau)\otimes V$ so, to prove the lemma, it suffices to show that
\begin{equation}
\label{eq:G(FoI)}
R(m\tau)_{12}R((m-1)\tau)_{23}\cdots R((m-d+2)\tau)_{d-1,d} \, \cdot \, ( F_{d-1}(-\tau) \otimes I) \;=\; 0.
\end{equation}

	Assume $1 \le m\le d-3$. The product  $R(\tau)_{m,m+1}R(0)_{m+1,m+2}R(-\tau)_{m+2,m+3}$ is a factor of the operator 
	in \cref{eq:G(FoI)}. But $R(0)_{m+1,m+2}$ is the identity so
	\begin{align*}
	R(\tau)_{m,m+1}R(0)_{m+1,m+2}R(-\tau)_{m+2,m+3} & \;=\; R(0)_{m+1,m+2}R(\tau)_{m,m+1}R(-\tau)_{m+2,m+3}
	\\
	& \;= \; R(0)_{m+1,m+2}R(-\tau)_{m+2,m+3}R(\tau)_{m,m+1}.
	\end{align*}
	In fact, $R(\tau)_{m,m+1}$ commutes with all the $R$-factors to the right of it in \cref{eq:G(FoI)} so
	\cref{eq:G(FoI)} can be written as  $Q'' R(\tau)_{m,m+1} \cdot \, ( F_{d-1}(-\tau) \otimes I)$ for some $Q''$.
	However, by \cref{le.shfl}, there is a factorization of the form $F_{d-1}(-\tau)=R(-\tau)_{m,m+1}Q'$ and hence a factorization $F_{d-1}(-\tau) \otimes I=R(-\tau)_{m,m+1}(Q' \otimes I)$. 
	The product in \cref{eq:G(FoI)} is therefore of the form
	\begin{equation*}
	Q'' R(\tau)_{m,m+1} \,\cdot \,  R(-\tau)_{m,m+1}(Q' \otimes I),
	\end{equation*}
	and this product is zero (\cref{lem.rz.dim.ker} tells us that $R(\tau)_{i,i+1}  R(-\tau)_{i,i+1}=0$ for all 	$i$).

	Assume $m=d-2$. The left-hand side of  \cref{eq:G(FoI)} is now 
\begin{equation*}
R((d-2)\tau)_{12}R((d-3)\tau)_{23}\cdots  R(\tau)_{d-2,d-1}   R(0)_{d-1,d} \, \cdot \, ( F_{d-1}(-\tau) \otimes I)
\end{equation*}
which equals $Q''R(\tau)_{d-2,d-1} \, \cdot \, ( F_{d-1}(-\tau) \otimes I)$ for some $Q''$.
However, by \cref{le.shfl}, $F_{d-1}(-\tau) = R(-\tau)_{d-2,d-1} Q$ for some $Q$ so, as before,  \cref{eq:G(FoI)} is zero.
\end{proof}

\begin{lemma}\label{lem.dim.ker.ineq.h}
$\nullity  G_{\tau}((d-1)\tau) \geq\tbinom{n+d-1}{d}$.
\end{lemma}
\begin{proof}
  When $z=(d-1)\tau$, the right-most factor in \cref{eq.comp.r.subsp} is $R(\tau)_{d-1,d}$ so
  \begin{align*}
    \ker G_{\tau}((d-1)\tau)
    &\; \supseteq \; \ker R(\tau)_{d-1,d}\cap(\im F_{d-1}(-\tau)\otimes V)\\
    & \; =\; \bigcap_{s+t+2=d}V^{\otimes s}\otimes\ker R_{\tau}(\tau)\otimes V^{\otimes t}\\
    &\; \supseteq \; \im F_{d}(-\tau)
  \end{align*}
  where the equality comes from the induction hypothesis \cref{eq.dim.rel.sp.im} applied to $\im F_{d-1}(-\tau)$. 
  The operator $F_{d}(-\tau)$ can be extended to all $\tau$ and its rank is generically large. Thus the desired inequality follows from the fact that 
 $\rank(\lim_{\tau\to 0}F_{d}(-\tau)) \;=\; \tbinom{n+d-1}{d}$,
  which is a consequence of \Cref{pr.flim}.
\end{proof}

\begin{lemma}\label{lem.dim.ker.ineq.m}
  $\nullity G_{\tau}(-\tau)  \geq n\tbinom{n+d-2}{d-1}-\tbinom{n+d-1}{d}$.
\end{lemma}
\begin{proof}
  By \cref{le.shfl}, 
  \begin{equation}\label{eq.lem.dim.ker.ineq.m}
    \sum_{s+t+2=d}V^{\otimes s}\otimes\im R(\tau)\otimes V^{\otimes t} \; \subseteq \; \ker F_{d}(-\tau).
  \end{equation}
  Extending the function $R_\tau(\tau)$ to all $\tau\in\bC$, the left-hand side of \cref{eq.lem.dim.ker.ineq.m} makes sense for all $\tau$ and its dimension is generically large by \Cref{pr:optosp,pr:genker}. When $\tau=0$ the left-hand side of \cref{eq.lem.dim.ker.ineq.m} is the space of degree-$d$ relations for the polynomial ring $SV$ by \Cref{cor.sym}, so its dimension is $n^{d}-\tbinom{n+d-1}{d}$.  Hence $\dim(\ker F_{d}(-\tau)) \ge n^{d}-\tbinom{n+d-1}{d}$ for generic $\tau$; but the dimension of this kernel is generically small by \Cref{pr:optosp}, so this inequality holds for all $\tau$. Therefore
  \begin{align*}
    \dim(\ker G_{\tau}(-\tau))
    &\;=\; \dim(\im F_{d-1}(-\tau)\otimes V) \, -\, \dim(\im G_{\tau}(-\tau)) \\
    &\; =\; n\dim(\im F_{d-1}(-\tau)) \, -\, \dim(\im F_{d}(-\tau))\\
    &\;\geq \; n\tbinom{n+d-2}{d-1}-\tbinom{n+d-1}{d}
  \end{align*}
  where the inequality comes from the induction hypothesis $\dim(\ker F_{d-1}(-\tau))=n^{d-1}-\tbinom{n+d-2}{d-1}$.
\end{proof}

\begin{lemma}
  \label{lem.G.nullity}
  For all $\zeta \in \frac{1}{n}\Lambda$,  
  $G_\tau(z+\zeta) $ and $G_\tau(z)$ have the same nullity. 
\end{lemma}
\begin{proof} 
  Assume $\zeta =\tfrac{a}{n}+\tfrac{b}{n}\eta$ where $a,b \in \bZ$, and let $C=T^bS^{ka}:V \to V$.
  In this proof we use the notation $C_i:= I^{\otimes (i-1)} \otimes C \otimes I^{\otimes (d-i)}$. 
  
  By \cref{cor.R.transl.props},   $R_\tau(z+\zeta) = f(z,\zeta,\tau) C_2^{-1} R_\tau(z) C_1$ 
  where $f(z,\zeta,\tau)$ is a nowhere vanishing function.  

  By definition, $G_\tau(z)$ is the restriction of $S^{\rm rev}_{1 \to d}(z,-\tau,\ldots,-\tau)$ to the image of $F^L_{d-1}(-\tau)$.
  Therefore $G_\tau(z+\zeta)$  is the restriction of 
  \begin{align*}
    & R(z+\zeta)_{12}R(z-\tau+\zeta)_{23}\ldots R(z-(d-2)\tau+\zeta)_{d-1,d}
    \\
    & \;=\; g(z,\zeta,\tau) C_2^{-1}R(z)_{12}C_1\cdot C_3^{-1}R(z-\tau)_{23}C_2\cdots C_d^{-1}R(z-(d-2)\tau)_{d-1,d}C_{d-1}
    \\
    & \;=\; g(z,\zeta,\tau) (C_2C_3 \ldots C_d)^{-1} \, R(z)_{12} (z-\tau)_{23}\cdots R(z-(d-2)\tau)_{d-1,d}\, (C_1 \ldots C_{d-1})
    \\
    & \;=\; g(z,\zeta,\tau) (C_2C_3 \ldots C_d)^{-1} \, S^{\rm rev}_{1 \to d}(z,-\tau,\ldots,-\tau) \, (C_1 \ldots C_{d-1})
  \end{align*}
  to  the image of $F^L_{d-1}(-\tau)$, where the function $g(z,\zeta,\tau)$ is a product of various $f(\cdot,\cdot,\cdot)$'s and therefore 
  never vanishes.
  It follows that $G_\tau(z+\zeta)$ and the restriction of 
  $S^{\rm rev}_{1 \to d}(z,-\tau,\ldots,-\tau) \, (C_1 \ldots C_{d-1})$ to  the image of $F^L_{d-1}(-\tau)$
  have the same nullity. 
  Equivalently, the nullity of $G_\tau(z+\zeta)$  equals the nullity of the restriction of 
  $S^{\rm rev}_{1 \to d}(z,-\tau,\ldots,-\tau) $ to the image of $C_1 \ldots C_{d-1} F^L_{d-1}(-\tau)$. 

  By \cref{lem.transf.props.R.tau.z}\cref{item.prop.s.t.comm}, $C^{\otimes {d-1}}$ commutes with $R(z)_{i,i+1}$ for all $z$ and all $1\le i \le d-2$,
  and therefore with $F_{d-1}(z)$. The image of $C_1 \ldots C_{d-1} F^L_{d-1}(-\tau)$ is therefore the same as
  the image of $F^L_{d-1}(-\tau)C_1 \ldots C_{d-1}$ which is, since $C$ is an automorphism of $V$, the same as the image of 
  $F^L_{d-1}(-\tau)$. Thus $G_\tau(z+\zeta)$ has the same nullity as $G_\tau(z)$.
\end{proof}

By \cref{lem.G.nullity}, the results in \cref{lem.dim.ker.ineq.p,lem.dim.ker.ineq.m,lem.dim.ker.ineq.h} hold when $\zeta \in \frac{1}{n}\Lambda$ is added to the input of $G_{\tau}(z)$.
Therefore
\begin{equation}\label{eq:dim.ker.ineq}
  \nullity G_{\tau}(z) \; \geq \;
  \begin{cases}
    \tbinom{n+d-1}{d} & \text{if $z\in (d-1)\tau+\frac{1}{n}\Lambda$,} \\
    n\tbinom{n+d-2}{d-1} & \text{if $z\in m\tau+\frac{1}{n}\Lambda$ for some $m=1,\ldots,d-2$,} \\
    n\tbinom{n+d-2}{d-1}-\tbinom{n+d-1}{d} & \text{if $z\in-\tau+\frac{1}{n}\Lambda$,}\\
    0 & \text{otherwise.}
  \end{cases}
\end{equation}

\begin{lemma}\label{lem.dim.ker.eq}
	Equality holds in \cref{eq:dim.ker.ineq}. Consequently,
	\begin{equation*}
		\rank  T_{d}(z,-\tau,\ldots,-\tau)\;=\;
		\begin{cases}
			n\tbinom{n+d-2}{d-1}-\tbinom{n+d-1}{d} & \text{if $z\in (d-1)\tau+\frac{1}{n}\Lambda$,} \\
			0 & \text{if $z\in m\tau+\frac{1}{n}\Lambda$ for some $m=1,\ldots,d-2$,} \\
			\tbinom{n+d-1}{d} & \text{if $z\in-\tau+\frac{1}{n}\Lambda$,}\\
			n\tbinom{n+d-2}{d-1} & \text{otherwise.}
		\end{cases}
	\end{equation*}
\end{lemma}
\begin{proof}
	By \cref{pr.hol-op}, the sum of the dimensions of $\ker G_{\tau}(z)$, as $z$ runs over all zeros of $\det G_{\tau}(z)$ in a 
	fundamental parallelogram for $\Lambda$, does not exceed the number of zeros of $\det G_{\tau}(z)$ in that parallelogram, which is $(d-1)n^{3}\tbinom{n+d-2}{d-1}$ by \cref{lem.gz.zeros}. 
	By the assumption $\tau\in\bC-\bigcup_{m=1}^{d}\frac{1}{mn}\Lambda$, the cosets $m\tau+\frac{1}{n}\Lambda$, $-1\leq m\leq d-1$, are pairwise disjoint.
	Since the sum of the three non-zero numbers appearing on the right-hand side of \cref{eq:dim.ker.ineq} is
	\begin{equation*}
		\tbinom{n+d-1}{d}+n\tbinom{n+d-2}{d-1}\cdot(d-2)+\big(n\tbinom{n+d-2}{d-1}-\tbinom{n+d-1}{d}\big)=(d-1)n\tbinom{n+d-2}{d-1}
	\end{equation*}
	and the fundamental parallelogram for $\Lambda$ contains exactly 
	$n^{2}$ points in $\frac{1}{n}\Lambda$, every inequality in \cref{eq:dim.ker.ineq} is an equality.
	Therefore
	\begin{align*}
		\dim(\im T_{d}(z,-\tau,\ldots,-\tau)) &\;=\; \dim(\im G_{\tau}(z))\qquad\text{by \cref{eq:imG.imF}}
		\\
		& \; =\; \dim(\im F_{d-1}(-\tau)\otimes V) \, -\, \dim(\ker G_{\tau}(z))
		\\
		&\;=\; n\tbinom{n+d-2}{d-1} \, - \, \dim (\ker G_{\tau}(z))
	\end{align*}
	where the last equality follows from the induction hypothesis.
\end{proof}

\begin{proof}[Proof of \cref{thm.dim.rel.sp}\cref{thm.dim.rel.sp.im}] 
	We observed in the proof of \cref{lem.dim.ker.ineq.h} that
	\begin{equation*}
	 \ker G_{\tau}((d-1)\tau)  \; \supseteq \; 
	 \bigcap_{s+t+2=d}V^{\otimes s}\otimes\ker R_{\tau}(\tau)\otimes V^{\otimes t}
	  \; \supseteq \; \im F_{d}(-\tau).
	 \end{equation*}
	 But $\dim(\ker G_{\tau}((d-1)\tau)) = \tbinom{n+d-1}{d}$ and
	 $\dim(\im F_{d}(-\tau))=\tbinom{n+d-1}{d}$ by \cref{lem.dim.ker.eq} since 
	 $F_{d}(-\tau)=T_{d}(-\tau,\ldots,-\tau)$, so these three subspaces of $V^{\otimes d}$ coincide. 
\end{proof}	

\begin{proof}[Proof of \cref{thm.dim.rel.sp}\cref{thm.dim.rel.sp.ker}] 	
We must show that the inclusion 
	\begin{equation}\label{eq:desired.equality}
		\sum_{s+t+2=d}V^{\otimes s}\otimes\im R(\tau)\otimes V^{\otimes t} \; \subseteq \; \ker F_{d}(-\tau)
	\end{equation}
(in \cref{le.shfl}\cref{item.shfl.ker.sum}) is an equality. We will do this by showing that the dimension of the left-hand side is $\ge$ the dimension of the right-hand side.  

By the induction hypothesis, the left-hand side of \cref{eq:desired.equality} equals 
$
\ker F_{d-1}^{R}(-\tau)+\im R(\tau)_{12}
$
and its dimension is 
\begin{align*}
		& \dim(\ker F_{d-1}^{R}(-\tau))+\dim(\im R(\tau)_{12}) \,-\, \dim(\ker F_{d-1}^{R}(-\tau)
		\cap\im R(\tau)_{12})
		\\
		& \quad \;=\; n\big(n^{d-1}\,-\,\tbinom{n+d-2}{d-1}\big)\, +\, n^{d-2}\tbinom{n}{2}
		\,-\, \dim(\ker F_{d-1}^{R}(-\tau) \cap\im R(\tau)_{12}),
\end{align*}
so it suffices to show that 
\begin{align*} 
		n\big(n^{d-1}\,-\,\tbinom{n+d-2}{d-1}\big) & \, +\, n^{d-2}\tbinom{n}{2}
		\,-\, \dim(\ker F_{d-1}^{R}(-\tau) \cap\im R(\tau)_{12})
		\\
		&  \; \ge \;  \dim(\ker F_{d}(-\tau))
		\\
		&  \; = \;  n^d \,-\, \dim(\im F_{d}(-\tau))
		\\
		&  \; = \;  n^d \,-\, \tbinom{n+d-1}{d}
\end{align*}
where the last equality appears in  the proof of \cref{thm.dim.rel.sp.im}. Equivalently, it suffices to show that  
\begin{align*}
 \dim(\ker F_{d-1}^{R}(-\tau) \cap\im R(\tau)_{12}) &  \; \le \;  n\big(n^{d-1}\,-\,\tbinom{n+d-2}{d-1}\big)\, +\, n^{d-2}\tbinom{n}{2} 
 \,-\, \big(n^d-\tbinom{n+d-1}{d} \big)
 \\
 & \;=\; \,-\,n\tbinom{n+d-2}{d-1}\, +\, n^{d-2}\tbinom{n}{2}  \,+\, \tbinom{n+d-1}{d}.
\end{align*}
Since
\begin{align*}
\dim(\ker F_{d-1}^{R}(-\tau)\cap \im R(\tau)_{12}) 
&  \;=\; \dim(\ker F_{d-1}^{R}(-\tau)R(\tau)_{12}) \,-\,  \dim(\ker R(\tau)_{12})
\\
&  \;=\; \dim(\ker F_{d-1}^{R}(-\tau)R(\tau)_{12}) \,-\, \big(n^d-n^{d-2} \tbinom{n}{2}),
\end{align*}
it suffices to show that 
\begin{align*}
 \dim(\ker F_{d-1}^{R}(-\tau)R(\tau)_{12})
 & \; \le \;
 \,-\,n\tbinom{n+d-2}{d-1}\, +\, n^{d-2}\tbinom{n}{2}  \,+\, \tbinom{n+d-1}{d}  \, + \,  \big(n^d-n^{d-2} \tbinom{n}{2}\big)
 \\
  & \; = \;
n^d \,-\,n\tbinom{n+d-2}{d-1}  \,+\, \tbinom{n+d-1}{d}  \,  .
\end{align*}

By \cref{le.f-id},
\begin{equation*}
T_d((d-1)\tau,-\tau,\ldots,-\tau) \;=\; F^R_{d-1}(-\tau) R(\tau)_{12}R(2\tau)_{23} \ldots R((d-1)\tau)_{d-1,d}
\end{equation*}
so 
\begin{equation*}
 \im F^R_{d-1}(-\tau) R(\tau)_{12} \; \supseteq \; \im T_d((d-1)\tau,-\tau,\ldots,-\tau)
\end{equation*}
and
\begin{align*}
\dim(\ker F^R_{d-1}(-\tau) R(\tau)_{12}) & \; \le \;  \dim(\ker T_d((d-1)\tau,-\tau,\ldots,-\tau))
\\
& \;=\; n^{d}\, - \, n\tbinom{n+d-2}{d-1}\,+\,\tbinom{n+d-1}{d}
\end{align*}
where the equality comes from \cref{lem.dim.ker.eq}.
The proof is now complete.
\end{proof}

\begin{theorem}
\label{thm.H.series}
For all $\tau\in(\bC-\bigcup_{m\geq 1}\frac{1}{m}\Lambda)\cup\frac{1}{n}\Lambda$, the Hilbert series of $Q_{n,k}(E,\tau)$ is the same as that of the polynomial ring on $n$ variables.
\end{theorem}
\begin{proof}
We proved this in \cite[\S5]{CKS1} for $\tau\in\frac{1}{n}\Lambda$, so we assume that $\tau\in\bC-\bigcup_{m\geq 1}\frac{1}{m}\Lambda$. By \cref{thm.dim.rel.sp}\cref{thm.dim.rel.sp.ker}, the space of degree-$d$ relations for $Q_{n,k}(E,\tau)$ has the same dimension as that for a polynomial algebra in $n$ variables. Hence the result follows.
\end{proof}

\subsubsection{Remarks}
\label{sect.H.series.remark}

Our proof  that $Q_{n,k}(E,\tau)$ has the ``right'' Hilbert series has nothing in common with earlier 
proofs that  $Q_{n,1}(E,\tau)$ has the ``right'' Hilbert series. 
The proofs for $n=3$ and $n=4$, and $k=1$, by Artin-Tate-Van den Bergh \cite{ATV1} and Smith-Stafford  \cite{SS92}, respectively,
relied on the following facts, none of which is guaranteed to hold for other $Q_{n,k}(E,\tau)$'s: (1)
$Q_{3,1}(E,\tau)$ has a central element of degree 3 (\cite[p.~211]{AS87}, \cite[Thm.~4.4]{dL})\footnote{Artin-Schelter's proof is ``by computer''. De Laet's is ``by algebra''.} and the quotient by it is a twisted homogeneous coordinate ring for $E$;
$Q_{4,1}(E,\tau)$ has a regular sequence consisting of two degree-2 central elements (\cite[Thm.~2]{Skl82}, \cite[Thm.~6.5]{dL})
  and the quotient by them is a twisted homogeneous coordinate ring for $E$;\footnote{These facts are analogues of the fact that when $E$ is embedded in $\bP^2$ or $\bP^3$ as an elliptic
normal curve of degree 3, or 4, respectively, it is a complete intersection. However, when $E$ is embedded in $\bP^{n-1}$ 
as an  elliptic normal curve of degree $n \ge 5$ it is not a complete intersection.} (2) the Riemann-Roch theorem for curves allows one to 
determine the Hilbert series of these two twisted homogeneous coordinate rings; (3) a tricky induction argument 
then allows one to ``climb up the regular sequence'' to show that the dimension of the degree-$i$ component $Q_{n,1}(E,\tau)_i$ 
(for $n=3,4$) is ``right'' and, simultaneously, that the central elements form a regular sequence with respect to homogeneous elements of degree $<i$. In  \cite{TvdB96}, Tate-Van den Bergh proved that 
$Q_{n,1}(E,\tau)$ has the ``right'' Hilbert series when $n \ge 5$. Their argument relies on modules of $I$-type (a notion they introduce) 
and a geometric definition of the defining relations (\cite[(4.2)]{TvdB96} (at the end of their section 1,  
they  suggest that $Q_{n,k}(E,\tau)$ might be amenable to their techniques).

It would be good to know whether the methods in this paper apply to other graded algebras whose defining 
relations are the image of a specialization of a family of operators $R(z)$ satisfying the QYBE.

\section{The Hilbert series for $Q_{n,k}(E,\tau)^!$}
\label{subsec.dual}

The argument we use  in this section to show that the Hilbert series of $Q_{n,k}(E,\tau)^!$ is $(1+t)^n$ is modeled on   the argument 
we used to show  that the Hilbert series of $Q_{n,k}(E,\tau)$ is $(1-t)^{-n}$.

\subsection{The algebras $S_{n,k}(E,\tau)$}
\label{sect.algs.Snk}

By definition, $Q_{n,k}(E,\tau)=TV/(\im R_{\tau}(\tau))$ with the convention that $R_{\tau}(\tau)$ is replaced by $R_{+}(\tau)$ in \cref{eq:mpm1} when $\tau\in\frac{1}{n}\Lambda$. We now define\index{S_n,k(E,tau)@$S_{n,k}(E,\tau)$}
\begin{equation*}
  S_{n,k}(E,\tau) \; := \; \frac{TV}{(\ker R_{\tau}(\tau))}
\end{equation*}
with the same convention. \cref{th:is-kszdl}  shows that $Q_{n,k}(E,\tau)^! \cong S_{n,n-k}(E,\tau)$.

By \cite[Prop.~3.22]{CKS1},  the automorphism $N(x_\a)=x_{-\a}$ of $V$ extends to algebra isomorphisms 
$Q_{n,k}(E,\tau) \to Q_{n,k}(E,-\tau)$ and  $Q_{n,k}(E,\tau) \to Q_{n,k}(E,\tau)^{\rm op}$.  
There is a similar result for $S_{n,k}(E,\tau)$.

\begin{proposition}\label{pr:sop}
Let $N \in \GL(V)$ be the map $N(x_\a)=x_{-\a}$. For all $\tau \in \bC$, $N$ extends to algebra isomorphisms
$S_{n,k}(E,\tau)\to S_{n,k}(E,-\tau)$ and
 $S_{n,k}(E,\tau)\to S_{n,k}(E,\tau)^{\rm op}$. In particular,
   \begin{equation*}
  S_{n,k}(E,\tau)  \; \cong \;  S_{n,k}(E,\tau)^{\rm op}  \;=\;  S_{n,k}(E,-\tau).
  \end{equation*}
\end{proposition}
\begin{proof}
It is clear that
\begin{equation*}
S_{n,k}(E,\tau)^{\rm op} \; =\; \frac{TV}{(\ker R_\tau(\tau)P)}
\end{equation*}
where $P$ is the operator $P(u \otimes v)=v \otimes u$.
By \cref{eq.-z}, $R_\tau(-\tau)= e(n^2\tau) PR_{-\tau}(\tau) P$ for all $\tau \in \bC-\frac{1}{n}\Lambda$, and thus for all $\tau\in\bC$ (we define $R_\tau(-\tau)$ and $R_{-\tau}(\tau)$ as limits when $\tau\in\frac{1}{n}\Lambda$).
Hence  $\ker R_\tau(-\tau)P = \ker R_{-\tau}(\tau)$ for all $\tau \in \bC$. This concludes that $S_{n,k}(E,\tau)^{\rm op} =  S_{n,k}(E,-\tau)$.

By \cref{eq.-tau.-z}, 
\begin{equation*}
(N \otimes N) R_\tau(\tau) \;=\; e(-n^2\tau) R_{-\tau}(-\tau)\, (N \otimes N)
\end{equation*} 
for all $\tau\in\bC$.
It follows that $\ker R_{-\tau}(-\tau)\, (N \otimes N) = \ker R_{\tau}(\tau)$.

Thus, $N \otimes N$ is an automorphism of  $V^{\otimes 2}$ that sends $\ker R_{\tau}(\tau)$, the space of quadratic relations for $S_{n,k}(E,\tau)$, to $\ker R_{-\tau}(-\tau)$, the space of quadratic relations for $S_{n,k}(E,-\tau)$, isomorphically. Therefore $N$ induces an isomorphism $S_{n,k}(E,\tau)\to S_{n,k}(E,-\tau)$.
\end{proof}

\subsection{The quadratic dual of $Q_{n,k}(E,\tau)$}
Let $\langle \, \cdot \, ,  \, \cdot \, \rangle:V \times V \to \bC$\index{<.,.>@$\langle\,\cdot\,,\,\cdot\,\rangle$} and 
$\langle \, \cdot \, , \, \cdot \, \rangle:V^{\otimes 2} \times V^{\otimes 2} \to \bC$
be the non-degenerate symmetric bilinear forms $\langle x_i, x_j \rangle = \delta_{ij}$ and
$\langle x_i \otimes x_k, x_j \otimes x_\ell \rangle = \delta_{ij}\delta_{k\ell}$.
The maps 
\begin{align*}
V \to V^*,  \; & \qquad v \mapsto \langle v ,  \, \cdot \, \rangle,
\\
V\otimes V \to (V\otimes V)^*,  \; & \qquad u \otimes v \mapsto \langle u \otimes v ,  \, \cdot \, \rangle,
 \\
 V^* \otimes V^* \to (V\otimes V)^*,  \; & \qquad
   \langle u,  \, \cdot \, \rangle \otimes \langle v ,  \, \cdot \, \rangle 
    \mapsto
    \langle u \otimes v ,  \, \cdot \, \rangle,
 \end{align*}
are isomorphisms. We will treat them as identifications. The third isomorphism is the composition of $V^{*}\otimes V^{*}\to V\otimes V$, induced by the first, and the second. We also define the isomorphism $V^{\otimes d}\to (V^{*})^{\otimes d}$ for each $d\geq 3$ in the same way as $d=2$ and identify $TV$ with $TV^{*}$.

If $R$ is a subspace of $V\otimes V$, then the \textsf{quadratic dual} $A=TV/(R)$ is defined to be $A^{!}:=TV^{*}/(R^{\perp})$\index{A^"!@$A^{"!}$}, where $R^{\perp}$\index{R^perp@$R^{\perp}$} is the annihilator with respect to the form $\langle\,\cdot\,,\,\cdot\,\rangle:V^{\otimes 2} \times V^{\otimes 2} \to \bC$ and 
is regarded as a subspace of $V^{*}\otimes V^{*}$.

Given a linear operator $R: V^{\otimes 2} \to V^{\otimes 2}$, we denote by  
$R^{\hdot}: V^{\otimes 2} \to V^{\otimes 2}$\index{R^.@$R^{\hdot}$} the unique linear map 
such that $\langle R^{\hdot}x, y \rangle=\langle x, Ry \rangle$ for all $x,y \in V^{\otimes 2}$.

\begin{lemma}
\label{lem.quad.dual}
If $V$ is a finite-dimensional vector space and $R^{*}:(V\otimes V)^* \to (V\otimes V)^*$ is the dual of a linear map
$R:V^{\otimes 2} \to V^{\otimes 2}$, then
\begin{enumerate}
  \item\label{item.quad.dual.imker}
  $(\im R)^\perp= \ker  R^{*}$ and
  \item\label{item.quad.dual.kerim}
  $(\ker R)^\perp = \im R^{*}$.
\end{enumerate}
With the conventions stated just before this lemma, 
\begin{equation*}
\left(  \frac{TV}{(\im R)} \right)^{\!!} \; = \;  \frac{TV}{(\ker R^\hdot)} \, .
\end{equation*}
\end{lemma}
\begin{proof}
Parts \cref{item.quad.dual.imker} and \cref{item.quad.dual.kerim} are basic linear algebra. The displayed equality follows because the left-hand side of it equals $TV^{*}/(\ker R^{*})$ which 
we are identifying with the right-hand side (by convention).
\end{proof}

\begin{lemma}
\label{lem.R.transpose}
For all $\tau\in\bC-\frac{1}{n}\Lambda$ and $z\in\bC$,
\begin{equation*}
R_{n,k,\tau}(z)^\hdot \;=\;  e(-n^{2}z)R_{n,n-k,-\tau}(-z).
\end{equation*}
\end{lemma}
\begin{proof}
Let $\cB=\{x_i \otimes x_j \; | \; i,j \in \bZ_n\}$. This is a basis for $V^{\otimes 2}$.
We must show, for all $p,q,s,t \in \bZ_n$, that
\begin{equation}\label{eq.pairing.pqrs}
\langle x_p \otimes x_q,R_{n,k,\tau}(z)(x_s \otimes x_t) \rangle
\;=\;  
e(-n^{2}z)\langle R_{n,n-k,-\tau}(-z)(x_p \otimes x_q), x_s \otimes x_t  \rangle.
\end{equation}
or, equivalently, that the coefficient of $x_p \otimes x_q$ in $R_{n,k,\tau}(z)(x_s \otimes x_t)$ with respect to
$\cB$ equals $e(-n^2z)$ times the coefficient of $x_s \otimes x_t$ in $R_{n,n-k,-\tau}(-z)(x_p \otimes x_q)$ with respect to $\cB$.

If $p+q \ne s+t$ then both sides of \cref{eq.pairing.pqrs} are zero so we assume $p+q=s+t$ for the remainder of the proof.

By \cref{eq.-tau.-z}, the right-hand side of \cref{eq.pairing.pqrs} is equal to
\begin{equation*}
\langle R_{n,n-k,\tau}(z)(x_q \otimes x_p), x_t \otimes x_s  \rangle
\;=\;\frac{\theta_0(-z)\cdots\theta_{n-1}(-z)}{\theta_1(0)\cdots\theta_{n-1}(0)}\frac{\theta_{p-q+r(-k-1)}(-z+\tau)}{\theta_{p-q-r}(-z)\theta_{-kr}(\tau)}
\end{equation*}
where $r\in\bZ_{n}$ is determined by $p-r=t$ or, equivalently, by $q+r=s$. Since $p-q=t-s+2r$,
\begin{equation*}
\langle R_{n,n-k,\tau}(z)(x_q \otimes x_p), x_t \otimes x_s  \rangle
\;=\;\frac{\theta_0(-z)\cdots\theta_{n-1}(-z)}{\theta_1(0)\cdots\theta_{n-1}(0)}\frac{\theta_{t-s+(-r)(k-1)}(-z+\tau)}{\theta_{t-s-(-r)}(-z)\theta_{k(-r)}(\tau)},
\end{equation*}
which equals the left-hand side of \cref{eq.pairing.pqrs}.
\end{proof}

\begin{theorem}
\label{th:is-kszdl}
For all $\tau\in\bC$,
\begin{equation}
Q_{n,k}(E,\tau)^! \; \cong \; S_{n,n-k}(E,\tau).
\end{equation}
\end{theorem}
\begin{proof}
\cref{lem.R.transpose} implies that $R_{n,k,\tau}(\tau)^\hdot \;=\;  e(-n^{2}\tau)R_{n,n-k,-\tau}(-\tau)$ for all $\tau\in\bC-\frac{1}{n}\Lambda$, and thus for all $\tau\in\bC$ if we define both sides as limits when $\tau\in\frac{1}{n}\Lambda$.
Therefore, by \cref{lem.quad.dual},
\begin{align*}
Q_{n,k}(E,\tau)^!
\;=\; \frac{TV}{(\ker R_{n,k,\tau}(\tau)^{\hdot})}
\;=\; \frac{TV}{(\ker R_{n,n-k,-\tau}(-\tau))}
\;=\; S_{n,n-k}(E,-\tau).
\end{align*}
The desired isomorphism now follows from \cref{pr:sop}.
\end{proof}

\begin{corollary}\label{cor:cls-ext}
  For all $\tau\in\bC$, $S_{n,1}(E,\tau)\cong \Lambda V$\index{LambdaV@$\Lambda V$}, the exterior algebra on $V$.
\end{corollary}
\begin{proof}
By \cref{th:is-kszdl}, $S_{n,1}(E,\tau)$ is isomorphic to the quadratic dual of $Q_{n,n-1}(E,\tau)$.
However, $Q_{n,n-1}(E,\tau) \cong SV$ by \cite[Prop.~5.5]{CKS1} so the result follows from the well-known fact that 
the quadratic dual of the polynomial algebra $SV$ is the exterior algebra $\Lambda V^*$.
\end{proof}

\subsection{Computation of the Hilbert series for $Q_{n,k}(E,\tau)^!$}

In this subsection we prove \cref{thm.H.series.dual}. After \cref{th:is-kszdl} it suffices to show that the Hilbert series of $S_{n,k}(E,\tau)$ is $(1+t)^n$ for all $k$ and all $\tau\in (\bC-\bigcup_{m=1}^{n+1}\frac{1}{mn}\Lambda) \cup \frac{1}{n}\Lambda$.
The arguments we use to show this are similar to those in \cref{sect.H.series} with the essential change that images and kernels of $R_{\tau}(\tau)$ are replaced by images and kernels of $R_{\tau}(-\tau)$. 

We adopt the convention that $\tbinom{n}{e}=0$ when $e>n$.

\begin{theorem}\label{thm.dim.rel+.sp}
Let $d\geq 2$. Assume $\tau\in\bC-\bigcup_{m=1}^{d}\frac{1}{mn}\Lambda$.
  \begin{enumerate}
  \item\label{thm.dim.rel+.sp.ker} We have
    \begin{equation}\label{eq.dim.rel+.sp.ker}
      \ker F_{d}(\tau) \; =\; \sum_{s+t+2=d}V^{\otimes s}\otimes\im R_{\tau}(-\tau)\otimes V^{\otimes t}
    \end{equation}
    and its dimension is the same as the dimension of the space of degree-$d$ relations for the exterior algebra in $n$ variables,
     namely $n^{d}-\tbinom{n}{d}$.
  \item\label{thm.dim.rel+.sp.im} We have
    \begin{equation}\label{eq.dim.rel+.sp.im}
      \im F_{d}(\tau) \; =\; \bigcap_{s+t+2=d}V^{\otimes s}\otimes\ker R_{\tau}(-\tau)\otimes V^{\otimes t}
    \end{equation}
    and its dimension is the same as the dimension of the degree-$d$ component of an exterior algebra in $n$ variables, 
    namely $\tbinom{n}{d}$.
  \end{enumerate}  
\end{theorem}
\begin{proof}
The proof is like that for \Cref{thm.dim.rel.sp} with some natural changes. 
The binomial coefficients $\tbinom{n+e-1}{e}$ are replaced by $\tbinom{n}{e}$. 
The operator 
\begin{equation*}
G_{\tau}(z):  \im F_{d-1}(-\tau)\otimes V \, \longrightarrow \, V\otimes\im F_{d-1}(-\tau)
\end{equation*}
that is the restriction of  $S^{\mathrm{rev}}_{1\to d}(z,-\tau,\ldots,-\tau)$ is replaced by\index{G^+_tau(z)@$G^+_{\tau}(z)$}
  \begin{equation*}
    G^+_{\tau}(z):V\otimes \im F_{d-1}(\tau) \, \longrightarrow \, \im F_{d-1}(\tau)\otimes V
  \end{equation*}
which is the restriction of  $S^{\mathrm{rev}}_{d\to 1}(z,\tau,\ldots,\tau)$. The result in \cref{lem.gz.zeros} showing that 
  $\det G_{\tau}(z)$ has exactly $(d-1)n^{3}\tbinom{n+d-2}{d-1}$ zeros in a fundamental parallelogram for $\Lambda$
  is replaced by the result that $\det G^+_{\tau}(z)$ 
has exactly $(d-1)n^{3}\tbinom{n}{d-1}$ zeros in a fundamental parallelogram for $\Lambda$. The analogue of \Cref{eq:dim.ker.ineq}
is now
\begin{equation}\label{lem.dim+.ker.ineq}
  \dim(\ker G^+_{\tau}(z)) \; \ge\; 
  \begin{cases}
    \tbinom{n}{d} & \text{if $z\in -(d-1)\tau+\frac{1}{n}\Lambda$,} \\
    n\tbinom{n}{d-1} & \text{if $z\in -m\tau+\frac{1}{n}\Lambda$ for some $m=1,\ldots,d-2$,} \\
    n\tbinom{n}{d-1}-\tbinom{n}{d} & \text{if $z\in\tau+\frac{1}{n}\Lambda$,}\\
    0 & \text{otherwise.}
  \end{cases}
\end{equation}
(The hypothesis that $\tau \notin \bigcup_{m=1}^{d}\frac{1}{mn}\Lambda$ ensures that the four cases
in \cref{lem.dim+.ker.ineq} are pairwise disjoint.)
After these changes, the argument then proceeds as before since 
\begin{equation*}
  \rank \Big( \lim_{\tau\to 0}F_{d}(\tau)\Big) \;=\;  \tbinom{n}{d}
\end{equation*}
by \Cref{pr.flim}.
\end{proof}

In analogy with \cref{lem.dim.ker.eq}, 
\begin{equation}\label{eq.dim.ker.eq.plus}
  \dim(\im T_{d}(\tau,\ldots,\tau,z))\; = \; 
  \begin{cases}
    n\tbinom{n}{d-1}-\tbinom{n}{d} & \text{if $z\in -(d-1)\tau+\frac{1}{n}\Lambda$,} \\
    0 & \text{if $z\in -m\tau+\frac{1}{n}\Lambda$ for some $m=1,\ldots,d-2$,} \\
    \tbinom{n}{d} & \text{if $z\in\tau+\frac{1}{n}\Lambda$,}\\
    n\tbinom{n}{d-1} & \text{otherwise.}
  \end{cases}
\end{equation}

\begin{theorem}
\label{thm.H.series.dual}
If $\tau\in (\bC-\bigcup_{m=1}^{n+1}\frac{1}{mn}\Lambda) \cup \frac{1}{n}\Lambda$, 
then $Q_{n,k}(E,\tau)^!$ has the same Hilbert series as the exterior algebra on 
$n$ variables.
\end{theorem}
\begin{proof}
If $\tau \in \frac{1}{n}\Lambda$, then $Q_{n,k}(E,\tau)$ is a twist of the polynomial ring on $n$ variables (\cite[Cor.~5.2]{CKS1}) so its category of graded 
modules is equivalent to the category of graded modules over that polynomial ring. 
But the Koszulity of a finitely generated connected graded algebra one depends only on its
category of graded modules (see the argument preceding \cite[Prop.~5.7]{Z-twist}). Since the polynomial ring is a Koszul algebra so is every twist of it. In particular, $Q_{n,k}(E,\tau)$ is a Koszul
algebra. The Hilbert series for $Q$ and $Q^!$ therefore satisfy the functional equation $H_{Q^!}(t) H_Q(-t)=1$. The Hilbert series for 
$Q_{n,k}(E,\tau)^!$ is therefore $(1+t)^n$.

For the rest of the proof we assume that $\tau \notin \bigcup_{m=1}^{n+1}\frac{1}{mn}\Lambda$.

Since  $\tau \notin \frac{1}{2n}\Lambda$, $\im R_{\tau}(-\tau)=\ker R_{\tau}(\tau)$ by \cref{prop.nullities}\cref{item.cor.r.rk.ex}.
By \cref{thm.dim.rel+.sp}\cref{thm.dim.rel+.sp.ker}, the degree-$d$ part of $S_{n,k}(E,\tau)$ has the same dimension as the degree-$d$ part of the exterior algebra on $n$ variables for all $0\leq d\leq n+1$. In particular, $S_{n,k}(E,\tau)_{n+1}=0$; since $S_{n,k}(E,\tau)$ is generated in degree-one, $S_{n,k}(E,\tau)_{d}=0$ for all $d\geq n+1$. Thus $S_{n,k}(E,\tau)$ has the same Hilbert series as the exterior algebra on $n$ variables for all $k$. The result now follows from \cref{th:is-kszdl}.
\end{proof}

\section{Multiplication in $Q_{n,k}(E,\tau)$}
\label{sect.mult.in.Q}

Let $\Sym^d V$\index{Sym^d V@$\Sym^d V$} (resp., $\Alt^d V$\index{Alt^d V@$\Alt^d V$}) denote the subspace of $V^{\otimes d}$ consisting of the symmetric (resp., anti-symmetric) tensors.
The restriction to $\Sym^d V$ of the natural map from the tensor algebra $TV$ to the symmetric algebra
 $SV:=TV/(\Alt^{2}V)$
is an isomorphism onto its image, $S^dV$,
 the degree-$d$ component of $SV$. The multiplication on $SV$ can therefore be transferred in a canonical way to a multiplication on
 $\Sym V := \bigoplus_{d \ge 0} \Sym^d V$\index{Sym V@$\Sym V$}. The induced multiplication is called the {\sf shuffle product}.

 In a similar way, the equality in \cref{eq.dim.rel.sp.ker} leads to a canonical isomorphism from $Q_{n,k}(E,\tau)$ to the subspace of $TV$
 that is the direct sum of the images of the operators $F_d(-\tau)$ which are, by \cref{pr.flim}, elliptic analogues of the symmetrization
 operators. Following this line of reasoning, the multiplication on $Q_{n,k}(E,\tau)$ can be transferred in a canonical way to this graded 
 subspace of $TV$.
 
 In this section we make this multiplication explicit in terms of certain operators, those in
 \cref{prop.ell.shuffle.prod}\cref{item.prop.ell.shuffle.prod.bilin}, that should be thought of as elliptic analogues of  the shuffle operators.

\subsection{The operators $M_{b,a}:V^{\otimes a} \otimes V^{\otimes b}  \to V^{\otimes (a+b)}$ }
At first sight, the calculations in this section might appear mysterious. They have been guided by a desire to find
an elliptic analogue of  the equality \cref{eq:shuffle.decomp} which says that the product on $\Sym V$ induced by 
the usual product on $SV$ {\it is} the shuffle product. We need some notation to explain this.

Let $a,b \in\bZ_{\ge 0}$. Let $S_{a+b}$ denote the group of permutations of $\{1,\ldots,a+b\}$. 
Define
\begin{align*}
S_{a|b}  &  \; :=\; \{ \s \in S_{a+b} \; | \; \s(1)<\cdots <\s(a) \text{ and }\s(a+1)<\cdots <\s(a+b)\},
\\
S_{a|\circ } & \; :=\;  \{ \s \in S_{a+b} \; | \; \s(i)=i \text{ for all } i\ge a+1\} ,
\\
S_{\circ | b} & \; :=\;  \{ \s \in S_{a+b} \; | \; \s(i)=i \text{  for all } i\le a \}.
\end{align*} 
Elements in $S_{a|b}$ are called {\sf shuffles}. 
If $\s \in S_{a+b}$, then there are unique elements $\omega \in S_{a|b}$, $\a \in S_{a |\circ}$, $\b \in S_{\circ |b}$ such that
$\s=\omega\a\b$. Hence, in the group algebra $\bC S_{a+b}$,  we have
\begin{equation}
\label{eq:shuffle.decomp}
\left( \sum_{\omega \in S_{a|b}} \omega \right) \left(  \sum_{\a \in S_{a|\circ}} \a \right)  \left(  \sum_{\b \in S_{\circ|b}} \b \right)
  \;=\; \sum_{\s \in S_{a+b}}   \s .
\end{equation}
The shuffle product $u \otimes v \mapsto u *v$ on $V^{\otimes (a+b)}$, and its restriction $\Sym^{a}V\otimes\Sym^{b}V\to\Sym^{a+b} V$, is given by 
$\frac{a!b!}{(a+b)!}$ times the  left-most term in \cref{eq:shuffle.decomp}. 

The second equality in \cref{le.mtt}, which is one of the main results in this section, namely
\begin{equation*}
M_{b,a}(-\tau) \cdot  \big(F_{a}(-\tau)\otimes F_{b}(-\tau)\big) \; =\;  F_{a+b}(-\tau),
\end{equation*}
is analogous to \cref{eq:shuffle.decomp} (with some factorial terms thrown in).
By \cref{pr.flim},  $\lim_{\tau\to 0}F_{a}(-\tau) = \sum_{\sigma\in S_a} \sigma$; i.e., $F_{a}(-\tau)$ is analogous  to the 
the middle term on the left-hand side of \cref{eq:shuffle.decomp}.  We now introduce the operator $M_{b,a}(-\tau)$ that 
will be analogous to the left-most term in \cref{eq:shuffle.decomp}. 

\begin{definition}
	Let $a,b\in\bZ_{\geq 0}$. Define the operator\index{M_a,b(z;x;y)@$M_{a,b}(z;\sfx;\sfy)$}
	\begin{equation*}
	M_{a,b}(z;x_{1},\ldots,x_{a-1};y_{1},\ldots,y_{b-1}):V^{\otimes (a+b)} \to V^{\otimes (a+b)}
	\end{equation*}
	to be
	\begin{equation*}
		\begin{matrix}
			R(z)_{a,a+1} & R(z+y_{1})_{a+1,a+2} & \cdots & R(z+\sum_{k}y_{k})_{a+b-1,a+b} \\
			R(z+x_{1})_{a-1,a} & R(z+x_{1}+y_{1})_{a,a+1} & \cdots & R(z+x_{1}+\sum_{k}y_{k})_{a+b-2,a+b-1} \\
			\vdots & \vdots & \ddots & \vdots \\
			R(z+\sum_{j}x_{j})_{12} & R(z+\sum_{j}x_{j}+y_{1})_{23} & \cdots & 
			R(z+\sum_{j}x_{j}+\sum_{k}y_{k})_{b,b+1}
		\end{matrix}
	\end{equation*}
	interpreted as either
	\begin{enumerate}
		\item the downward product of the rightward products along rows, or
		\item the rightward product of the downward products along columns.
	\end{enumerate}
	If $a=0$ or $b=0$, we regard the operator as the identity.
\end{definition}

For example,  $M_{2,3}(z;x;y_1,y_2)$ is
\begin{align*}
 R(z)_{23}&R(z+y_1)_{34}R(z+y_1+y_2)_{45} R(z+x)_{12}R(z+x+y_1)_{23} R(z+x+y_1+y_2)_{34}
\\
 =\;  R(z)_{23}&R(z+x)_{12}R(z+y_1)_{34} R(z+x+y_1)_{23}         R(z+y_1+y_2)_{45}     R(z+x+y_1+y_2)_{34}.
\end{align*}

Let $\sfx:=(x_{1},\ldots,x_{a-1})$ and $\sfy:=(y_{1},\ldots,y_{b-1})$. The interpretations in (1) and (2) yield
\begin{align}
	M_{a,b}(z;\sfx;\sfy)
	&\label{eq.M.incr}\textstyle \; =\; S^{\mathrm{rev}}_{a\to a+b}(z,\sfy)S^{\mathrm{rev}}_{a-1\to a+b-1}(z+x_{1},\sfy)\cdots S^{\mathrm{rev}}_{1\to b+1}(z+\sum_{j}x_{j},\sfy)\\
	&\label{eq.M.decr}\textstyle \; =\; S^{\mathrm{rev}}_{a+1\to 1}(z,\sfx)S^{\mathrm{rev}}_{a+2\to 2}(z+y_{1},\sfx)\cdots S^{\mathrm{rev}}_{a+b\to b}(z+\sum_{k}y_{k},\sfx)
\end{align}
respectively.
We leave the reader to verify that the procedures in (1) and (2) produce the same result (this does not involve using the Yang-Baxter
equation). The first step in verifying this is to notice that if one starts with the product produced by (1), then the factors 
$R(z+x_1+\cdots+x_j)_{a-j,a-j+1}$ coming from the left-most column commute with all the entries in the array that appear to the northeast 
of that factor; thus the product produced by (1) is equal to $S^{\rm rev}_{a+1 \to 1}(z, \sfx)$ times 
\begin{equation*}
S^{\mathrm{rev}}_{a+1\to a+b}(z+y_1,\ldots, y_{b-1})S^{\mathrm{rev}}_{a\to a+b-1}(z+x_{1}+y_1,y_2,\ldots, y_{b-1})\cdots 
S^{\mathrm{rev}}_{2\to b+1}(z+\sum_{j}x_{j}+y_1,y_2,\ldots, y_{b-1}).
\end{equation*}
One then treats this product in the same way, and so on.

We write $M_{a,b}(z)$\index{M_a,b(z)@$M_{a,b}(z)$} for $M_{a,b}(z;\sfx;\sfy)$ if $z=x_1 =\cdots = x_{a-1} = y_1 = \cdots y_{b-1}$.

\begin{lemma}\label{le.TM.MT}
	Let $\sfx=(x_{1},\ldots,x_{a-1})$ and $\sfy=(y_{1},\ldots,y_{b-1})$. As operators on $V^{\otimes(a+b)}$,
	\begin{equation}\label{eq.TLM.MTR}
		\textstyle T^{L}_{a}(\sfx)M_{a,b}(z+\sum_{j}x_{j};-\sfx;\sfy)\; =\; M_{a,b}(z;\sfx^{\mathrm{rev}};\sfy)T^{R}_{a}(\sfx)
	\end{equation}
	and
	\begin{equation}\label{eq.TRM.MTL}
		\textstyle T^{R}_{b}(\sfy)M_{a,b}(z+\sum_{k}y_{k};\sfx^{\mathrm{rev}};-\sfy^{\mathrm{rev}})\; =\; M_{a,b}(z;\sfx^{\mathrm{rev}};\sfy)T^{L}_{b}(\sfy)
	\end{equation}
	where $\sfx^{\mathrm{rev}}:=(x_{a-1},\ldots,x_{1})$ and $\sfy^{\mathrm{rev}}:=(y_{b-1},\ldots,y_{1})$.
\end{lemma}

\begin{proof}
	By \cref{eq.M.decr} and \cref{eq.S.rev.1},  
	\begin{align*}
		&\textstyle M_{a,b}(z+\sum_{j}x_{j};-\sfx;\sfy)\\
		&\textstyle =\; S^{\mathrm{rev}}_{a+1\to 1}(z+\sum_{j}x_{j},-\sfx)S^{\mathrm{rev}}_{a+2\to 2}(z+\sum_{j}x_{j}+y_{1},-\sfx)\cdots S^{\mathrm{rev}}_{a+b\to b}(z+\sum_{j}x_{j}+\sum_{k}y_{k},-\sfx)\\
		&\textstyle =\; S_{a+1\to 1}(\sfx,z)S_{a+2\to 2}(\sfx,z+y_{1})\cdots S_{a+b\to b}(\sfx,z+\sum_{k}y_{k}).				
	\end{align*}
	By \cref{le.f-id}, 
\begin{equation}
\label{eq:T^L.S=Srev.T^R}
T^{L}_{d-1}(z_{1},\ldots,z_{d-2})S_{d\to 1}(z_{1},\ldots,z_{d-1}) \;=\; S^{\mathrm{rev}}_{d\to 1}(z_{d-1},\ldots,z_{1})T^{R}_{d-1}(z_{1},\ldots,z_{d-2}).
\end{equation}
Hence $T^{L}_{a}(\sfx)M_{a,b}(z+\sum_{j}x_{j};-\sfx;\sfy)$ equals
\begin{align*}
 & T^{L}_{a}(\sfx) S_{a+1\to 1}(\sfx,z)S_{a+2\to 2}(\sfx,z+y_{1})S_{a+3\to 3}(\sfx,z+y_{1}+y_2)\cdots 
\\
& \;=\; S^{\mathrm{rev}}_{a+1\to 1}(z,\sfx^{\rm rev})(I\otimes T_{a}(\sfx))^{L} S_{a+2\to 2}(\sfx,z+y_{1})S_{a+3\to 3}(\sfx,z+y_{1}+y_2) \cdots
\\
& \;=\; S^{\mathrm{rev}}_{a+1\to 1}(z,\sfx^{\rm rev})S^{\mathrm{rev}}_{a+1\to 1}(z+y_1,\sfx^{\rm rev}) (I^{\otimes 2}\otimes T_{a}(\sfx))^{L} S_{a+3\to 3}(\sfx,z+y_{1}+y_2) \cdots
\end{align*}
where the last equality is obtained by applying \cref{eq:T^L.S=Srev.T^R} after observing that the previous 
$(I \otimes T_{a}(\sfx))^L$ is of the form $T^{L}_{a}(\sfx)$ with respect to $S_{a+2\to 2}(\sfx,z+y_{1})$. Repeating this procedure we
eventually see that $T^{L}_{a}(\sfx)M_{a,b}(z+\sum_{j}x_{j};-\sfx;\sfy)$ equals
\begin{equation*}
\textstyle S^{\mathrm{rev}}_{a+1\to 1}(z,\sfx^{\rm rev})S^{\mathrm{rev}}_{a+2\to 2}(z+y_{1},\sfx^{\rm rev})\cdots S^{\mathrm{rev}}_{a+b\to b}(z+\sum_{k}y_{k},\sfx^{\rm rev})T^{R}_{a}(\sfx)
\end{equation*}
which is $M_{a,b}(z;\sfx^{\rm rev};\sfy)T^{R}_{a}(\sfx)$.

A similar argument proves \cref{eq.TRM.MTL}.
\end{proof}

\begin{lemma}\label{le.t-big}
For positive integers  $a,b$,  
  \begin{multline}
    \label{eq:tss}
    T^{L}_a(z_1,\ldots,z_{a-1}) S_{a+1\to 1}(z_1,\ldots,z_a) S_{a+2\to 2}(z_1,\ldots,z_{a-1},z_a+z_{a+1})\cdots\\
    \cdots S_{a+b\to b}(z_1,\ldots,z_{a-1},z_{a}+\cdots+z_{a+b-1})T^{L}_{b}(z_{a+1},\ldots,z_{a+b-1})  \;=\; 
    T_{a+b}(z_1,\ldots,z_{a+b-1}).   
  \end{multline}
\end{lemma}
\begin{proof}
Since $T^{L}_{d-1}(z_{1},\ldots,z_{d-2})S_{d\to 1}(z_{1},\ldots,z_{d-1}) =T_{d}(z_{1},\ldots,z_{d-1} )$ (by \cref{le.f-id}),
the product of the two left-most factors on the left-hand side of \Cref{eq:tss} equals
\begin{equation*}
T^{L}_{a+1}(z_1,\ldots,z_a).
\end{equation*}
By definition \Cref{eq.def.td}, the right-most factor $T^{L}_{b}(z_{a+1},\ldots,z_{a+b-1})$ on the left-hand side is
  \begin{equation}\label{eq:ts-pl}
    T^{L}_{b}(z_{a+1},\ldots,z_{a+b-1})  = S_{2\to 1}(z_{a+1})S_{3\to 1}(z_{a+1},z_{a+2})\ldots S_{b\to 1}(z_{a+1},\ldots,z_{a+b-1}).
  \end{equation}
  Each of the $b-1$ resulting $S_{j\to 1}$ factors commutes with all $S_{a+k\to k}$ factors ending the left-hand side of \Cref{eq:tss} for $k>j$. Implementing this commutation for each of the $S$ factors in \Cref{eq:ts-pl} means attaching $S_{j\to 1}$ in \Cref{eq:ts-pl} to $S_{a+j\to j}$ in \Cref{eq:tss} to produce
  \begin{equation*}
    S_{a+j\to j}(z_1,\ldots,z_{a-1},z_a+\cdots+z_{a+j-1}) S_{j\to 1}(z_{a+1},\ldots,z_{a+j-1}) =
    S_{a+j\to 1}(z_1,\ldots,z_{a+j-1}).
  \end{equation*}
  Multiplying these by the $T^{L}_{a+1}(z_1,\ldots,z_a)$ we already have and applying \Cref{le.f-id} successively now yields the right-hand side $T_{a+b}(z_1,\ldots,z_{a+b-1})$ of \Cref{eq:tss}, as claimed.
\end{proof}

\begin{lemma}\label{le.TMT}
	If $\sfx=(x_{1},\ldots,x_{a-1})$ and $\sfy=(y_{1},\ldots,y_{b-1})$, then
	\begin{equation*}
		\textstyle T_{a+b}(\sfx,z,\sfy) \; =\; M_{a,b}\big(z;\sfx^{\mathrm{rev}};\sfy\big)\cdot T^{R}_{a}(\sfx)\cdot T^{L}_{b}(\sfy).
	\end{equation*}
\end{lemma}
\begin{proof}
\cref{le.t-big} can be re-stated as
\begin{equation*}
\textstyle T_{a+b}(\sfx,z,\sfy) \; =\; T^{L}_{a}(\sfx)\cdot M_{a,b}\big(z+\sum_{j}x_{j};-\sfx;\sfy\big)\cdot T^{L}_{b}(\sfy).
\end{equation*}
So the result follows from \cref{eq.TLM.MTR}.
\end{proof}

\subsection{Multiplication in $Q_{n,k}(E,\tau)$}\label{subse.descr.qnk}

Assume $\tau\in\bC-\bigcup_{m\geq 1}\frac{1}{m}\Lambda$.
By \cref{eq.dim.rel.sp.ker},
\begin{equation*}
Q_{n,k}(E,\tau)_d\;=\;\frac{V^{\otimes d}}{\ker F_{d}(-\tau)},
\end{equation*}
which is canonically isomorphic to $\im F_{d}(-\tau)$. Thus 
\begin{equation*}
Q_{n,k}(E,\tau) \; \cong \; \bigoplus_{d\geq 0}\im F_{d}(-\tau)
\end{equation*}
as graded vector spaces, so the multiplication on $Q_{n,k}(E,\tau)$ induces a multiplication on the right-hand space 
making it a graded $\bC$-algebra. \cref{prop.ell.shuffle.prod} describes the induced multiplication.

\begin{lemma}\label{le.mtt}
	With the notation above,
	\begin{align*}
		M_{b,a}(\tau) \cdot (F_{a}(\tau)\otimes F_{b}(\tau)) & \; =\; F_{a+b}(\tau),
		\\ 
		M_{b,a}(-\tau) \cdot  (F_{a}(-\tau)\otimes F_{b}(-\tau))& \; =\;  F_{a+b}(-\tau).
	\end{align*}
\end{lemma}

\begin{proof}
	By \cref{le.TMT},
	\begin{align*}
		F_{a+b}(\pm\tau)\;  
		& =\; M_{b,a}(\pm\tau;(\pm\tau)^{b-1};(\pm\tau)^{a-1}) T^{R}_{b}((\pm\tau)^{b-1}) T^{L}_{a}((\pm\tau)^{a-1})
		\\
		& =\; M_{b,a}(\pm \tau)F^{R}_{b}(\pm\tau)F^{L}_{a}(\pm\tau)
		\\
		& =\; M_{b,a}(\pm \tau) \cdot ( F_{a}(\pm\tau)  \otimes F_{b}(\pm\tau))
	\end{align*}
	as desired.  
\end{proof}

\begin{proposition}
\label{prop.ell.shuffle.prod}
Let $\tau\in\bC-\bigcup_{m\geq 1}\frac{1}{m}\Lambda$.
\begin{enumerate}
\item\label{item.prop.ell.shuffle.prod.bilin}
	Define a bilinear multiplication on $A:=\bigoplus_{d\geq 0}\im F_{d}(-\tau)$ by the maps
	\begin{equation*}
		\im F_{a}(-\tau)\otimes\im F_{b}(-\tau) \, \longrightarrow \,  \im F_{a+b}(-\tau)
	\end{equation*}
	induced from $M_{b,a}(-\tau)$ for all $a,b\geq 0$. Then $A$ is a graded algebra isomorphic to $Q_{n,k}(E,\tau)$.
\item\label{item.prop.ell.shuffle.prod.dual}	
	Similarly, the maps induced from $M_{b,a}(\tau)$ make $\bigoplus_{a+b \geq 0}\im F_{a+b}(\tau)$ a graded algebra isomorphic to the quadratic dual $Q_{n,k}(E,\tau)^{!}$.
\end{enumerate}	
\end{proposition}
\begin{proof}
	Write $Q$ for $Q_{n,k}(E,\tau)$ and $Q_i$ for its degree-$i$ component. Consider the diagram
	\begin{equation*}
		\begin{tikzcd}
			V^{\otimes a}\otimes V^{\otimes b}\ar[d,equal]\ar[r,twoheadrightarrow] & Q_{a}\otimes Q_{b}\ar[d]\ar[r,"\sim"] & A_{a}\otimes A_{b}\ar[r,hook] & V^{\otimes a}\otimes V^{\otimes b}\ar[d,"M_{b,a}(-\tau)"] \\
			V^{\otimes (a+b)}\ar[r,twoheadrightarrow] & Q_{a+b}\ar[r,"\sim"] & A_{a+b}\ar[r,hook] & V^{\otimes (a+b)}
		\end{tikzcd}
	\end{equation*}
	where the first and second rows are factorizations of $F_{a}(-\tau)\otimes F_{b}(-\tau)$ and $F_{a+b}(-\tau)$, respectively. 
	The left-hand square commutes since $Q$ is defined as a quotient of the tensor algebra, and the commutativity of the outer square follows from \cref{le.mtt}. Thus the right-hand square also commutes, whence $M_{b,a}(-\tau)$ induces a map $A_{a}\otimes A_{b}\to A_{a+b}$, which is equal to the one induced from the multiplication of $Q$. This proves the first statement.
	
	If we replace all $-\tau$'s in the above argument by $\tau$ we get a proof of the second statement using \cref{eq.dim.rel+.sp.im}.
	\end{proof}

\section{Koszulity of $Q_{n,k}(E,\tau)$}
\label{se.ksz}

Throughout this section, we assume that $\tau\in\bC-\bigcup_{m\geq 1}\frac{1}{m}\Lambda$ to ensure that the necessary results such as \Cref{thm.dim.rel.sp,thm.dim.rel+.sp} hold.

Let $\Lat(V^{\otimes d})$ denote the lattice of subspaces of $V^{\otimes d}$.

We will use the following result to show that $Q_{n,k}(E,\tau)$ is a Koszul algebra.

\begin{lemma}[Backelin]\cite[Thm.~2.4.1]{PP05}.
Let $\tau\in\bC$. $Q_{n,k}(E,\tau)$ is a Koszul algebra if and only if the sublattice of $\Lat(V^{\otimes d})$\index{Lat(V)@$\Lat(V)$} generated by\index{W_i@$W_{i}$}
\begin{equation}\label{eq:wi}
  W_i \; := \; V^{\otimes( i-1)}\otimes \rel_{n,k}(E,\tau)  \otimes V^{\otimes (d-i-1)},
  \qquad i=1,\ldots,d-1,
\end{equation}
is distributive for all $d \ge 2$.
\end{lemma}

\subsection{Distributive lattices}
Recall that a lattice $(\cL,\vee,\wedge)$ is \textsf{distributive} if
\begin{align}
x\vee(y\wedge z)&=(x\vee y)\wedge(x\vee z)\quad\text{and}\label{eq.dist.cond1}\\
x\wedge(y\vee z)&=(x\wedge y)\vee(x\wedge z)\label{eq.dist.cond2}
\end{align}
for all $x,y,z\in\cL$. Condition \cref{eq.dist.cond1} holds for all $x,y,z$ if only if \cref{eq.dist.cond2} holds for all $x,y,z$.

A lattice $\cL$ is \textsf{modular} if \cref{eq.dist.cond1} (or equivalently, \cref{eq.dist.cond2}) holds for all triples $(x,y,z)$ satisfying $x\geq z$. As explained in \cite[Lem.~1.6.1]{PP05}, if $\cL$ is modular, then \cref{eq.dist.cond1} and \cref{eq.dist.cond2} are equivalent for \emph{each} triple $(x,y,z)$ and these conditions are invariant under permutations of $x$, $y$ and $z$. If those equivalent conditions hold, we say that the triple $(x,y,z)$ is \textsf{distributive}.

We write $\Lat(d)$\index{Lat(d)@$\Lat(d)$} for the sublattice of 
$\Lat(V^{\otimes d})$ generated by $W_1,\ldots, W_{d-1}$. Like the lattice of subspaces of any vector space, $\Lat(V^{\otimes d})$ is modular, and so is $\Lat(d)$.

We say that $X\in \Lat(d)$  has {\sf classical dimension} if it 
has the same dimension as its counterpart for the polynomial ring $SV$. 
This terminology is not really for $X$, an element of $\Lat(d)$, but rather for an expression of $X$ using join and meet.

Since $\dim\rel_{n,k}(E,\tau) = \binom{n}{2} = \dim(\Alt^2 V)$, 
every $W_i$ has classical dimension; its classical counterpart is
$V^{\otimes (i-1)} \otimes \Alt^2 V \otimes V^{\otimes (d-i-1)}$. 
The subspaces\index{Sigma_s@$\Sigma_s$}\index{I_t@$I_t$}
\begin{equation*}
  \Sigma_s\;:=\;\sum_{i=1}^s W_i, \qquad \text{and} \qquad I_t\;:=\;\bigcap_{j=d-t}^{d-1} W_j
\end{equation*}
also have classical dimension for all $s$ and $t$ by \Cref{thm.dim.rel.sp,thm.dim.rel+.sp}.  
It follows that $\Sigma_s \cap I_{t}$ has classical dimension
if and only if $\Sigma_s + I_{t}$ does.

Because $\Lat(V^{\otimes d})$ is modular, the second half of \cite[Thm.~1.6.3]{PP05} tells us the following.

\begin{proposition}\label{th.dist}
\label{th.kosz}
  Let $d\ge 3$ and let $W_i$, $1\le i \le d-1$, be the subspaces of $V^{\otimes d}$ defined in \Cref{eq:wi}. 
  If $\Lat(2),\ldots,\Lat(d-1)$ are distributive and, for $1\le\ell\le d-1$, the triple 
  \begin{equation}
  \label{eq:dist.lattice}
   \left(  \sum_{i=1}^{\ell-1} W_i,\; W_{\ell},\; \bigcap_{j=\ell+1}^{d-1}W_j \right)
  \end{equation}
 is distributive, then $\Lat(d)$ is distributive and $Q_{n,k}(E,\tau)$ is a Koszul algebra. 
\end{proposition}

We will prove that $\Lat(d)$ is distributive by induction on $d$.

\begin{lemma}
Fix $d \ge 3$. Assume $\Lat(2),\ldots,\Lat(d-1)$ are distributive.
If $\Sigma_i \cap I_{d-i-1}$ has classical dimension for all $i=0,\ldots,d-1$, then $\Lat(d)$ is distributive.
\end{lemma}
\begin{proof}
It suffices to show that $(\Sigma_{\ell-1},W_\ell,I_{d-\ell-1})$ is a distributive triple for all  integers $\ell$
in $[1,d-1]$.

Fix $\ell$ and write $r:=d-\ell-1$.

Since $\Sigma_{\ell-1}+W_\ell = \Sigma_\ell$ and $W_\ell \cap I_r = I_{r+1}$, the distributivity condition  
\begin{equation*}
  \Sigma_{\ell-1}+(W_\ell\cap I_r) \;=\; (\Sigma_{\ell-1}+ W_\ell) \cap (\Sigma_{\ell-1}+ I_r). 
\end{equation*}
is equivalent to the condition
\begin{equation}\label{eq:wlr}
  \Sigma_{\ell-1}+ I_{r+1} \;=\; \Sigma_\ell \cap (\Sigma_{\ell-1}+ I_r).
\end{equation}
The two terms on the right-hand side of \Cref{eq:wlr} have classical dimensions:
\begin{itemize}
\item[$\hdot$] $\Sigma_\ell$ does by \Cref{thm.dim.rel.sp};
\item[$\hdot$] $\Sigma_{\ell-1}+ I_r$ does by \cref{thm.dim.rel.sp} and 
the observation that $\Sigma_{\ell-1}=X\otimes V^{\otimes d-\ell}$ and $I_r=V^{\otimes \ell}\otimes Y$ where
  \begin{equation*}
    X=\Sigma_{\ell-1}\subseteq V^{\otimes \ell},\quad \text{ and } \quad Y=I_r\subseteq V^{\otimes (d-\ell)}
  \end{equation*}
whence
$\dim (\Sigma_{\ell-1}\cap I_r) = \dim(X \otimes Y) = \dim X \cdot \dim Y = \text{the classical dimension}$.
\end{itemize}

Since $\Sigma_\ell$ and $I_{r}$ have classical dimension, and $\Sigma_\ell \cap I_{r}$ has classical dimension by assumption,
$\Sigma_\ell + I_{r}$ also has classical dimension. But this is true for all $\ell$ so the left-hand side of  \Cref{eq:wlr} has classical dimension. So does the right-hand side because 
\begin{align*}
\dim(\Sigma_\ell \cap (\Sigma_{\ell-1}+ I_r)) & \;=\;  \dim\Sigma_\ell + \dim(\Sigma_{\ell-1}+ I_r) - \dim(\Sigma_\ell + (\Sigma_{\ell-1}+I_r))
\\
& \;=\;  \dim\Sigma_\ell + \dim(\Sigma_{\ell-1}+ I_r) - \dim(\Sigma_\ell + I_r).
\end{align*}
Thus the left- and right-hand sides of \Cref{eq:wlr} have classical dimensions. However,  
\begin{equation*}
  \Sigma_{\ell-1}+ I_{r+1}\;  \subseteq  \; \Sigma_{\ell} \cap (\Sigma_{\ell-1}+ I_r)
\end{equation*}
and this inclusion is an equality in the case of the polynomial ring so both the left- and right-hand sides have the same dimension for
the polynomial ring {\it and} in the present situation too because they have classical dimensions (by hypothesis).
Hence this inclusion is an equality in our case too; i.e., \cref{eq:wlr} holds, and the proof is complete. 
\end{proof}

Thus, to show that $Q_{n,k}(E,\tau)$ is a Koszul algebra it suffices to show that
$\Sigma_\ell\cap I_{r}$ (with $r=d-\ell-1$) has classical dimension
for all $d\geq 3$ and $0\leq\ell\leq d-1$. We will achieve this goal in  \Cref{pr.dist-right-dim}.

Let $r=d-\ell-1$. If $\ell \in \{0,1,d-1\}$, then $\Sigma_\ell\cap I_{r}$ has classical dimension so we can assume that $2 \le \ell \le d-2$ (i.e., $1 \le r \le d-3$), but for now we also allow the case $\ell=1$ (i.e., $r=d-2$) to show some necessary results for induction and exclude this case later.

To show that $\Sigma_\ell \cap I_r$ has classical dimension we first convert the problem into a 
question about the rank of the operator $F^L_{\ell+1}(-\tau)F^{R}_{r+1}(\tau): V^{\otimes d} \to V^{\otimes d}$.

\begin{lemma}
\label{eq:7}
Let $F_p(z)$ be the operator on $V^{\otimes p}$ defined in \cref{defn.F_d}. We have 
\begin{equation*}
\dim\!\left(\Sigma_\ell \cap I_r\right) \;=\; \dim\!\left( \ker F^L_{\ell+1}(-\tau)F^{R}_{r+1}(\tau)\right)
\, -\, n^{\ell}\left(n^{r+1}-\tbinom{n}{r+1}\right).
\end{equation*}
\end{lemma}
\begin{proof}
Since $\Sigma_\ell \cap I_r = \sum_{i=1}^{\ell}W_i \,  \cap\, \bigcap_{j=\ell+1}^{d-1} W_j$  (by definition),
\begin{align*}
\Sigma_\ell \cap I_r
&  \;=\;
\ker\!\big(F^L_{\ell+1}(-\tau): V^{\otimes d} \to V^{\otimes d}\big)  \, \cap \, \im\!\big(F^{R}_{r+1}(\tau):V^{\otimes d} \to V^{\otimes d}\big)
\end{align*}
by  \Cref{thm.dim.rel.sp,thm.dim.rel+.sp}. Therefore
\begin{equation*}
\dim \left(\Sigma_\ell \cap I_r\right) \;=\; \dim\!\left( \ker F^L_{\ell+1}(-\tau)F^{R}_{r+1}(\tau)\right)
\, -\, \dim\!\left(\ker F^{R}_{r+1}(\tau)\right).
\end{equation*}
But
$
\dim\!\left(\ker F^{R}_{r+1}(\tau)\right) =   n^{\ell}\left(n^{r+1}-\tbinom{n}{r+1}\right)
$
by \Cref{thm.dim.rel+.sp}, so the proof is complete.
\end{proof}

\subsubsection{Notation for the classical case} 
\label{sect.cl.notn}
For the symmetric algebra, $SV$, the analogue of $W_i$ is\index{Lambda_i,i+1@$\Lambda_{i,i+1}$}
\begin{equation*}
 \Lambda_{i,i+1} \; :=\;  V^{\otimes (i-1)} \otimes \Alt^2 V \otimes V^{\otimes (d-i-1)}
\end{equation*}
The  classical analogue of $\Sigma_\ell \cap I_r$ is therefore the space\index{W^l,r@$W^{\ell,r}$}
\begin{equation*}
W^{\ell+1,r+1}  \; :=  \;   \sum_{i=1}^{\ell}\Lambda_{i,i+1}\cap\bigcap_{i=\ell+1}^{d-1}\Lambda_{i,i+1}.
\end{equation*}

\subsection{The operators $T_{\ell,r}(z)$ and $H_\tau(z)$ on $V^{\otimes d}$. }\label{subse.cnt}


When $\ell\geq 2$, we will use the notation\index{T_ell,r(z)@$T_{\ell,r}(z)$}
\begin{equation*}
  T_{\ell,r}(z) \;  := \;  T_d(\tau^r,-(r+1)\tau,(-\tau)^{\ell -2},z).
\end{equation*}
When $\ell=1$, we only define
\begin{equation*}
  T_{1,r}(-\tau) \;  := \;  T_d(\tau^r,-(r+1)\tau).
\end{equation*}

\begin{lemma}
\label{lem.MTT}
With the above notation, 
\begin{equation}
\label{MTT}
T_{\ell,r}(z)  \;=\; M_{r+1,\ell}(-(r+1)\tau; \tau^r; (-\tau)^{\ell-2},z) \, \cdot \,  T^L_{\ell}((-\tau)^{\ell-2},z)  \, \cdot \, T^R_{r+1}(\tau^r) .
\end{equation}
 \end{lemma}
\begin{proof}
This follows from \cref{le.TMT}.
\end{proof}

We now define $H_\tau(z)$.

\begin{proposition}
\label{prop.defn.H}
If $\ell\geq 2$, then the theta operator
\begin{equation}
  \label{eq:sw} 
 S^{\mathrm{rev}}_{d\to 1}(z, \, (-\tau)^{\ell -2}, \,  -(r+1)\tau, \, \tau^r)
\end{equation}
on $V^{\otimes d}$ restricts to a theta operator\index{H_tau(z)@$H_\tau(z)$}
\begin{equation*}
  H_\tau(z): V\otimes\im T_{\ell-1,r}(-\tau) \, \longrightarrow \, \im T_{\ell-1,r}(-\tau)\otimes V
\end{equation*}
and $\im H_\tau(z)=\im T_{\ell,r}(z)$.
\end{proposition}
\begin{proof}
By \Cref{le.f-id}, 
\begin{align*}
T_{\ell,r}(z)    &  \;=\; 
S^{\mathrm{rev}}_{d\to 1}(z, \, (-\tau)^{\ell-2},\,  -(r+1)\tau, \, \tau^r) \, \cdot \,
  T^R_{d-1}(\tau^r, \,  -(r+1)\tau, (-\tau)^{\ell-2}) \numberthis \label{eq:tstts}
  \\
& \; = \; T^L_{d-1}( \tau^r, \, -(r+1)\tau, \, (-\tau)^{\ell-2}) \, \cdot \, S_{d\to 1}(\tau^r, \,  -(r+1)\tau, \, (-\tau)^{\ell-2}, \, z).
\end{align*}
In other words,
\begin{align*}
T_{\ell,r}(z)
\numberthis \label{eq:tt-tw}
  &  \;=\; S^{\mathrm{rev}}_{d\to 1}(z, \, (-\tau)^{\ell-2},\,  -(r+1)\tau, \, \tau^r) \, \cdot \, T^R_{\ell-1,r}(-\tau) 
  \\
& \; = \; T^L_{\ell-1,r}(-\tau) \, \cdot \, S_{d\to 1}(\tau^r, \,  -(r+1)\tau, \, (-\tau)^{\ell-2}, \, z)
\end{align*}
 The results follow because $\im T^R_{\ell-1,r}(-\tau)  = V\otimes\im T_{\ell-1,r}(-\tau)$ and
 $\im T^L_{\ell-1,r}(-\tau)=\im T_{\ell-1,r}(-\tau)\otimes V$.
\end{proof}

\begin{lemma}
\label{lem.new.goal}
Let $1\leq\ell\leq d-2$. Then $\Sigma_\ell \cap I_r$ has classical dimension if and only if
\begin{equation}\label{eq:alt-goalt}
 \dim\!\big(\!\im T_{\ell,r}(-\tau)\big) \;=\;  n^{\ell}\tbinom{n}{r+1} \, -\, \dim W^{\ell+1,r+1}.
\end{equation}
\end{lemma}
\begin{proof}
Setting $z=-\tau$ and  using \cref{eq.M.incr} to factor the $M_{r+1,\ell}$ term in \cref{MTT} as 
$A\cdot S^{\mathrm{rev}}_{1\to\ell+1}((-\tau)^\ell)$, we have
\begin{align*}
T_{\ell,r}(-\tau) &  \;=\; A \cdot S^{\mathrm{rev}}_{1\to\ell+1}((-\tau)^\ell) T^L_{\ell}((-\tau)^{\ell-1})T^{R}_{r+1}(\tau^r)
\\
&  \;=\; A \cdot T^L_{\ell+1}((-\tau)^\ell)T^{R}_{r+1}(\tau^r) \qquad \qquad\text{by \cref{le.f-id}}
\\
&  \;=\; A \cdot F^L_{\ell+1}(-\tau)F^{R}_{r+1}(\tau).
\end{align*}
Since the $R$'s appearing in $A$ belong to $\{R(-2\tau),\ldots,R(-(d-1)\tau)\}$, the assumption $\tau\in\bC-\bigcup_{m\geq 1}\frac{1}{m}\Lambda$
implies that $A$ is an isomorphism. Hence, by \cref{eq:7},
\begin{align*}
\dim\!\left(\Sigma_\ell \cap I_r\right)
&\;=\;\dim\!\left( \ker F^L_{\ell+1}(-\tau)F^{R}_{r+1}(\tau)\right)
\, -\, n^{\ell}\left(n^{r+1}-\tbinom{n}{r+1}\right)\\
&\;=\;\dim V^{\otimes d}-\dim(\im T_{\ell,r}(-\tau))\, -\, n^{\ell}\left(n^{r+1}-\tbinom{n}{r+1}\right)\\
&\;=\;n^{\ell}\tbinom{n}{r+1}\,-\,\dim(\im T_{\ell,r}(-\tau)).
\end{align*}
Since $W^{\ell+1,r+1}$ is the classical analogue of $\Sigma_\ell \cap I_r$, the result follows.
\end{proof}

\subsubsection{The induction hypothesis}
We will prove that \cref{eq:alt-goalt} holds by induction on $d$. 
Thus, we assume \cref{eq:alt-goalt} is true for $d-1$ or fewer tensorands. If $\ell=1$, then \cref{eq:alt-goalt} follows from \cref{lem.new.goal} since $\Sigma_1 \cap I_r=I_{r+1}$ has classical dimension. So we also assume $2 \le \ell \le d-2$, i.e., $1 \le r \le d-3$.
The induction hypothesis implies that
\begin{equation}
\label{eq:dim.domain.H}
\dim \!\big(\text{the domain of $H_\tau(z)$}\big) \;=\; n^\ell \tbinom{n}{r+1} \,-\, n\dim W^{\ell,r+1}.
\end{equation}

The function $\det H_\tau(z)$ in the next result is only defined up to a non-zero scalar multiple (see \cref{rmk.det}).

\begin{proposition}\label{le.hdet}
The  function $\det H_\tau(z)$ is a  theta function with respect to  $\frac{1}{n}\Lambda$ having
  \begin{equation}\label{eq:hdet}
   (d-1)\dim\!\big(V \otimes \im T_{\ell-1,r}(-\tau)\big)  \; = \;  (d-1)\! \left(n^{\ell}\tbinom{n}{r+1}\,-\, n\dim W^{\ell,r+1}\right)
  \end{equation} 
  zeros in each fundamental parallelogram for $\frac{1}{n}\Lambda$, all of which belong to 
  \begin{equation}
  \label{eq:possible.zeros}
  \{-\tau,0,\tau,\ldots, (d-1)\tau\} + \tfrac{1}{n}\Lambda.
  \end{equation}
\end{proposition}
\begin{proof}
Since $R(z)$ is a theta operator of order $n^2$ with respect to $\Lambda$,  $S^{\rm rev}_{1 \to d}(z,\ldots)$ in \cref{eq:sw} is a 
theta operator of order $(d-1)n^2$ with respect to $\Lambda$.   
Since $H_\tau(z)$ is the restriction of $S^{\rm rev}_{1 \to d}(z,\ldots)$ to $ \im T^R_{\ell-1,r}(-\tau)$,  \Cref{prop.hol.det.zeros} tells us that
$\det H_\tau(z)$ has 
\begin{align*}
  (d-1)n^2\dim \!\big(\! \im T^R_{\ell-1,r}(-\tau)\big)
\end{align*}
zeros in every fundamental parallelogram for $\Lambda$. However, since $\det R(z)$ is a theta function with respect to $\frac{1}{n}\Lambda$, so is $\det H_\tau(z)$. Hence $\det H_\tau(z)$ has $(d-1) \,  \dim \!\big(\! \im T^R_{\ell-1,r}(-\tau)\big)$
zeros in every fundamental parallelogram for $\frac{1}{n}\Lambda$.

By the induction hypothesis, \cref{eq:alt-goalt} tells us that  
$\dim(\im T_{\ell-1,r}(-\tau)) =  n^{\ell-1}\tbinom{n}{r+1} \, -\, \dim W^{\ell,r+1}$. The equality in \cref{eq:hdet} 
now follows once we observe that $\det H_\tau(z)$ is not identically zero:
it isn't because the factors of $S^{\rm rev}_{1 \to d}(z,\ldots)$ are invertible for all but finitely many $z$'s.  

Since $H_\tau(z)$ is the restriction of a product of terms of the form $R(z-m\tau)_{i,i+1}$ for various $i$'s and
 $m=0,\ldots, d-2$, the zeros of $\det H_\tau(z)$ belong to 
\begin{equation*}
\{z \; | \; \det R(z-m\tau)=0 \text{ for some } m \in [0,d-2] \}.
\end{equation*}
But $\det R(z)=0$ if and only if $z \in \pm \tau+\frac{1}{n}\Lambda$, so this set is $\{-\tau,0,\tau,\ldots, (d-1)\tau\} + \tfrac{1}{n}\Lambda$.
\end{proof}
 
We now examine $\mult_p(\det H_\tau(z))$ for the $p$'s in \cref{eq:possible.zeros}. 
In truth, we will only examine $\mult_p(\det H_\tau(z))$ when $p \in  \{-\tau,0,\tau,\ldots, (d-1)\tau\}$ and then apply
\cref{lem.H.nullity}. 

As  the next result shows, $H_\tau(p)=0$  for some of these $p$'s.

\begin{lemma}\label{le.hzero}
If $m \in \bZ \cap [1,d-3]$,  then $\dim(\ker H_\tau(m\tau)) \ge  n^\ell \tbinom{n}{r+1} \,-\, n\dim W^{\ell,r+1}$.
\end{lemma}
\begin{proof}
By \cref{eq:dim.domain.H}, that dimension is $n^\ell \tbinom{n}{r+1} \,-\, n\dim W^{\ell,r+1}$. 
We will therefore show that $H_\tau(m\tau)=0$.

Since $\im H_\tau(m\tau)=\im T_{\ell,r}(m\tau)$ by \cref{prop.defn.H},
$H_\tau(m\tau)=0$ if $T_{\ell,r}(m\tau)=0$. Thus we will prove the lemma by showing that $T_{\ell,r}(m\tau)=0$
for the $m$'s in $[1,d-3]$. We split the proof into two parts.

(1)  Assume  $1\le m\le \ell-2$. (This case is vacuous if $\ell=2$ so we assume $\ell \ge 3$.)
By \cref{lem.MTT}, $T^L_{\ell}((-\tau)^{\ell-2},m\tau)$ is a factor of $T_{\ell,r}(m\tau)$ so it suffices to 
show that $T_{\ell}((-\tau)^{\ell-2},m\tau)=0$. By \Cref{le.f-id},
  \begin{equation*}
    T_{\ell}((-\tau)^{\ell-2},m\tau) \; =\;  S^{\mathrm{rev}}_{\ell\to 1}(m\tau, (-\tau)^{\ell-2}) \cdot     T^R_{\ell-1}((-\tau)^{\ell-2}).
 \end{equation*}
But
 \begin{equation*}
   S^{\mathrm{rev}}_{\ell\to 1}(m\tau,(-\tau)^{\ell-2}) \;=\; R(m\tau)_{\ell-1,\ell} R((m-1)\tau)_{\ell-2,\ell-1}\cdots R((m-\ell+2)\tau)_{12}
 \end{equation*}
 and this has two consecutive factors   of the form
 \begin{equation*}
   R(\tau)_{j,j+1} \, R(0)_{j-1,j}
 \end{equation*}
 for some $j \ge 2$. The $R(\tau)_{j,j+1}$ commutes past the $R(0)=I \otimes I$ term, and also commutes past the other factors to 
 its right, to ultimately annihilate the term $T^R_{\ell-1}((-\tau)^{\ell-2}) = R(-\tau)_{j,j+1} Q'$ (where the latter equality uses \Cref{le.shfl}\cref{item.shfl.Q.m}).

(2) Assume $\ell-1\le m\le d-3$. By \cref{eq.M.decr}, $M_{r+1,\ell}(-(r+1)\tau; \tau^r; (-\tau)^{\ell-2},m\tau)$ is right-divisible by
\begin{equation*}
S^{\rm rev}_{d \to \ell}((m-d+2)\tau,\tau^r) \;=\; R((m-d+2)\tau)_{d-1,d} R((m-d+3)\tau)_{d-2,d-1} \ldots R((m-\ell+1)\tau)_{\ell,\ell+1},
\end{equation*}
and $m-d+2 \le -1 \le 0 \le m-\ell+1$. Hence it has two consecutive factors of the form $R(-\tau)_{j,j+1}R(0)_{j-1,j}$ for some $j\ge \ell+1$. Since $R(0)=I \otimes I$, $R(-\tau)_{j,j+1}$ commutes with all factors to the right of it in $S^{\rm rev}_{d \to \ell}((m-d+2)\tau,\tau^r)$. After moving it all the way to the right in the expression \cref{MTT}, we conclude that
 $T_{\ell,r}(m\tau)$ is right-divisible by $R(-\tau)_{j,j+1}T^R_{r+1}(\tau^r)$. However, by \cref{le.shfl}(2), $T^R_{r+1}(\tau^r)=R(\tau)_{j,j+1}Q$ 
 for some $Q$ so $R(-\tau)_{j,j+1}T^R_{r+1}(\tau^r)=0$. 
\end{proof}

\begin{lemma}\label{le.null.mult.two}
If $T_{\ell-1,r}^{R}(-\tau)$ denotes $I \otimes T_{\ell-1,r}(-\tau)$  acting on $V^{\otimes d}$, then
\begin{equation*}
\rank T_{\ell-1,r}^{R}(-\tau)S_{1\to\ell}((-\tau)^{\ell-1}) \; =\; \tbinom{n+\ell-1}{\ell}\tbinom{n}{r+1}.
\end{equation*}
\end{lemma}
\begin{proof}
As operators on $V^{\otimes  (d-1)}$, 
  \begin{align*}
    T_{\ell-1,r}(-\tau)
    & \; =\; M_{r+1,\ell-1}(-(r+1)\tau;\tau^{r};(-\tau)^{\ell-2}) \cdot T^{L}_{\ell-1}((-\tau)^{\ell-2}) \cdot T^{R}_{r+1}(\tau^{r})\\
    & \; =\; M^R_{r,\ell-1}(-(r+1)\tau;\tau^{r-1};(-\tau)^{\ell-2})\cdot S^{\rm rev}_{1\to\ell}((-\tau)^{\ell-1})\cdot T^{L}_{\ell-1}((-\tau)^{\ell-2})\cdot T^{R}_{r+1}(\tau^{r})\\
    & \; =\;  M^R_{r,\ell-1}(-(r+1)\tau;\tau^{r-1};(-\tau)^{\ell-2})\cdot T^{L}_{\ell}((-\tau)^{\ell-1})\cdot T^{R}_{r+1}(\tau^{r})
  \end{align*}
  where $M^R_{r,\ell-1}(\cdots)$ is acting on the $(d-2)$ right-most tensorands of  $V^{\otimes  (d-1)}$;
   the second equality comes from   \cref{eq.M.incr}, and the third  from \cref{le.f-id}. 
We now view this as an equality of operators on $V^{\otimes d}=V \otimes V^{\otimes  (d-1)}$ by considering each of the four operators
in it as acting  on the  right-most $(d-1)$ tensorands; i.e., we replace each operator
by ($I \, \otimes \,$itself).  But $I \otimes  T_{\ell-1,r}(-\tau)=T^R_{\ell-1,r}(-\tau)$,  so the equality implies 
\begin{equation*}
T^R_{\ell-1,r}(-\tau)\cdot S_{1\to\ell}((-\tau)^{\ell-1}) \;=\; M_{r,\ell-1}^R(\cdots)  \cdot (I \otimes T^{L}_{\ell}((-\tau)^{\ell-1}))\cdot 
T^{R}_{r+1}(\tau^{r}) \cdot S_{1\to\ell}((-\tau)^{\ell-1}).
\end{equation*} 
By the assumption $\tau\in\bC-\bigcup_{m\geq 1}\frac{1}{m}\Lambda$, all the $R$'s appearing in the term $M^R_{r,\ell-1}(\cdots)$ are isomorphisms. The rank of 
$T^R_{\ell-1,r}(-\tau)S_{1\to\ell}((-\tau)^{\ell-1})$ therefore equals that of
\begin{equation}
\label{eq:new.image}
\big(I \otimes T^{L}_{\ell}((-\tau)^{\ell-1})\big)\cdot T^{R}_{r+1}(\tau^{r}) \cdot S_{1\to\ell}((-\tau)^{\ell-1}).
\end{equation} 
The left-most $\ell$ tensorands of $V^{\otimes d}$ are ``disjoint'' from the right-most $(r+1)$ tensorands so
\begin{equation*}
\text{\cref{eq:new.image}}
\;=\; 
\big(I \otimes T^{L}_{\ell}((-\tau)^{\ell-1})\big)\cdot  S_{1\to\ell}((-\tau)^{\ell-1}) \cdot T^{R}_{r+1}(\tau^{r}).
\end{equation*} 
But $S_{1 \to \ell}(z_1,\ldots,z_{\ell-1}) = S_{1 \to \ell+1}(z_1,\ldots,z_{\ell-1},0)$ and
\begin{equation*}
T^R_{\ell}(z_2,\ldots,z_{\ell}) \cdot S_{1 \to \ell+1}(z_{\ell},\ldots, z_1) \;=\; S^{\rm rev}_{1 \to \ell+1}(z_1,\ldots, z_\ell)\cdot T^L_{\ell}(z_2,\ldots,z_{\ell}) 
\end{equation*}
by \Cref{le.f-id}, so, as operators on $V^{\otimes( \ell+1)}$,  
\begin{align*}
\text{\cref{eq:new.image}} & \;= \; \big(I \otimes T^{L}_{\ell}((-\tau)^{\ell-1})\big) \cdot S_{1 \to \ell}((-\tau)^{\ell-1})  \cdot T^{R}_{r+1}(\tau^{r})
\\
&  \;=\; \big(I \otimes T^{L}_{\ell}((-\tau)^{\ell-1})\big) \cdot S_{1 \to \ell+1}((-\tau)^{\ell-1},0)  \cdot T^{R}_{r+1}(\tau^{r})
\\
&  \;=\;  S^{\rm rev}_{1 \to \ell+1}(0,(-\tau)^{\ell-1}) \cdot T^{L}_{\ell}((-\tau)^{\ell-1})  \cdot T^{R}_{r+1}(\tau^{r}).
\end{align*}

By \cref{lem.dim.ker.eq}, $G_\tau(0)$ is an isomorphism. But $G_\tau(0)$ is the 
restriction of $S^{\mathrm{rev}}_{1\to\ell+1}(0,(-\tau)^{\ell-1})$ to the image of $T^{L}_{\ell}((-\tau)^{\ell-1})$ so
$S^{\mathrm{rev}}_{1\to\ell+1}(0,(-\tau)^{\ell-1})$ acts injectively on the image of $T^{L}_{\ell}((-\tau)^{\ell-1})$, whence
\begin{equation*}
\rank\text{\cref{eq:new.image}} \;=\; \rank T^{L}_{\ell}((-\tau)^{\ell-1}) T^{R}_{r+1}(\tau^{r}).
\end{equation*} 
As operators on $V^{\otimes d}=V^{\otimes \ell} \otimes V^{\otimes (r+1)}$,  $T^{L}_{\ell}((-\tau)^{\ell-1}) T^{R}_{r+1}(\tau^{r})
=T_{\ell}((-\tau)^{\ell-1})  \otimes T_{r+1}(\tau^{r})$ so
\begin{align*}
\rank T^{L}_{\ell}((-\tau)^{\ell-1}) T^{R}_{r+1}(\tau^{r}) & \;=\; \rank T^{L}_{\ell}((-\tau)^{\ell-1}) \cdot \rank T^{R}_{r+1}(\tau^{r})
\\
& \;=\; \rank F_{\ell}(-\tau) \cdot \rank F_{r+1}(\tau)
\end{align*}
which is $\tbinom{n+\ell-1}{\ell}\tbinom{n}{r+1}$ by \cref{thm.dim.rel.sp,thm.dim.rel+.sp}.
\end{proof}

\subsubsection{Terminology}
A family of linear operators $A(z)$ {\sf has a zero of multiplicity $m$ at a point $p\in \bC$} if $A(z)=(z-p)^mB(z)$ where $B(z)$
is an operator such that $\det B(z)$ has neither a zero nor a pole at $p$. 

\begin{lemma}\label{le.ker.mult.two}
  The  restriction of $H_\tau(z)$ to  $\im T_{\ell-1,r}^{R}(-\tau)S_{1\to\ell}((-\tau)^{\ell-1})$ has a zero of multiplicity $\ge 2$  
  at $z=(\ell-1)\tau$.
\end{lemma}

\begin{proof}
Since $H_\tau(z):\im T^R_{\ell-1,r}(-\tau)  \to \im T^L_{\ell-1,r}(-\tau)$ is the restriction of 
$S^{\mathrm{rev}}_{d\to 1}(z, \, (-\tau)^{\ell -2}, \,  -(r+1)\tau, \, \tau^r)$, it suffices to prove that
	\begin{equation}\label{eq.SrevTR}
		S^{\mathrm{rev}}_{d\to 1}(z,(-\tau)^{\ell-2},-(r+1)\tau,\tau^{r})\cdot T_{\ell-1,r}^{R}(-\tau)\cdot S_{1\to\ell}((-\tau)^{\ell-1})
	\end{equation}
	has a zero of multiplicity $\ge 2$ on its domain $V^{\otimes d}$. Since
	\begin{align*}
		&S^{\mathrm{rev}}_{d\to 1}(z,(-\tau)^{\ell-2},-(r+1)\tau,\tau^{r})\cdot T_{\ell-1,r}^{R}(-\tau)\\
		&\;=\;T_{d}(\tau^{r},-(r+1)\tau,(-\tau)^{\ell-2},z)\\
		&\;=\;M_{r+1,\ell}(-(r+1)\tau;\tau^{r};(-\tau)^{\ell-2},z)\cdot T^{R}_{r+1}(\tau^{r})\cdot T^{L}_{\ell}((-\tau)^{\ell-2},z)
	\end{align*}
	by \cref{eq:tt-tw} and \cref{le.TMT}, the operator  in \cref{eq.SrevTR} is the composition of
	\begin{equation}\label{eq.TLM}
		M_{r+1,\ell}(-(r+1)\tau;\tau^{r};(-\tau)^{\ell-2},z)\cdot T^{R}_{r+1}(\tau^{r})
	\end{equation}
	and
	\begin{equation}\label{eq.TLS}
		T^{L}_{\ell}((-\tau)^{\ell-2},z)\cdot S_{1\to\ell}((-\tau)^{\ell-1}).
	\end{equation}
	We will show that  each of these operators is zero at $z=(\ell-1)\tau$.
	
	First,
		\begin{align*}
	\text{\cref{eq.TLM}}
		& \;=\; M^L_{r+1,\ell-1}(-(r+1)\tau;\tau^{r};(-\tau)^{\ell-2})\cdot S^{\mathrm{rev}}_{d\to\ell}(z-(\ell+r-1)\tau,\tau^{r})\cdot T^{R}_{r+1}(\tau^{r})\\
		& \;=\; M^L_{r+1,\ell-1}(-(r+1)\tau;\tau^{r};(-\tau)^{\ell-2})\cdot T^{R}_{r+2}(\tau^{r},z-(\ell+r-1)\tau)
	\end{align*}
	where the first and second equalities follow from \cref{eq.M.decr} and \cref{le.f-id}, respectively. When $z=(\ell-1)\tau$, 
	the right-most factor is $T^{R}_{r+2}(\tau^{r},-r\tau)$ which $=0$ by \cref{eq.dim.ker.eq.plus}. Hence $\text{\cref{eq.TLM}}=0$.
	
	When $z=(\ell-1)\tau$, 
	\begin{align*}
		\text{  \cref{eq.TLS}} &  \;=\; 
		T^{L}_{\ell-1}((-\tau)^{\ell-2})\cdot S_{\ell \to 1}((-\tau)^{\ell-2},(\ell-1)\tau)\cdot  S_{1\to\ell}((-\tau)^{\ell-1}) 
		\qquad \text{by \cref{le.f-id}}
		\\
		&  \;=\; 
		T^{L}_{\ell-1}((-\tau)^{\ell-2})\cdot R(\tau)_{\ell-1,\ell} \cdots R((\ell-1)\tau)_{12}  
		\cdot R(-(\ell-1)\tau)_{12}\cdots R(-\tau)_{\ell-1,\ell}.
	\end{align*}
By \cref{lem.comp.pm}, for all $u$ the operator 
$R(u)R(-u)$ is a scalar multiple of the identity so we can rearrange the terms in this product to
	obtain a factor of the form $R(\tau)_{\ell-1,\ell} R(-\tau)_{\ell-1,\ell}$ which $=0$.
\end{proof}

\begin{lemma}
\label{le.-tau}
$\dim(\ker H_\tau(-\tau)) \; \ge \; \dim W^{\ell+1,r+1}-n\dim W^{\ell,r+1}$.
\end{lemma}
\begin{proof}
Since $\im H_\tau(m\tau)=\im T_{\ell,r}(m\tau)$ by \cref{prop.defn.H},
\begin{align*}
  \dim(\ker H_\tau(z)) + \dim(\im T_{\ell,r}(z)) &  \;=\;   \dim(\im T^R_{\ell-1,r}(-\tau))
\\
&  \;=\;  n^{\ell}\tbinom{n}{r+1} \,-\, n\dim W^{\ell,r+1}
\end{align*}
where the second equality follows from the induction hypothesis.
Thus, to prove the lemma we must show that 
\begin{equation*}
n^{\ell}\tbinom{n}{r+1}   \, - \, \dim(\im T_{\ell,r}(-\tau)) \; \ge \; \dim W^{\ell+1,r+1}.
\end{equation*}

However, 
\begin{align*}
 \dim(\ker T_{\ell,r}(-\tau)) &  \; \ge \; \dim(\ker F^L_{\ell+1}(-\tau)F^{R}_{r+1}(\tau)) \qquad \text{because, as observed in \cref{lem.new.goal},}
 \\ 
  &  \phantom{\; \ge \; \dim(\ker F^L_{\ell+1}(-\tau)F^{R}_{r+1}(\tau))}
  \qquad \text{   $T_{\ell,r}(-\tau) = A \cdot F^L_{\ell+1}(-\tau)F^{R}_{r+1}(\tau)$ for some  $A$}
 \\
                        &\;=\; \dim \left(\ker F^L_{\ell+1}(-\tau)\cap \im F^{R}_{r+1}(\tau)\right) \,+\, \dim(\ker F^{R}_{r+1}(\tau))
                         \\
                         &\;=\; \dim \left(\ker F^L_{\ell+1}(-\tau)\cap \im F^{R}_{r+1}(\tau)\right) \,+\, n^d-n^{\ell}\tbinom{n}{r+1}
                         \qquad \text{by \Cref{thm.dim.rel+.sp}}
                         \\
                         &\;=\;  \dim \left(\sum_{i=1}^{\ell} W_i \, \cap  \, \bigcap_{j=\ell+1}^{d-1}W_j \right) \,+\, n^d-n^{\ell}\tbinom{n}{r+1} 
                         \qquad \text{by \Cref{thm.dim.rel.sp,thm.dim.rel+.sp}}
                         \\
                         &\;=\;  \dim \left(\Sigma_\ell \cap I_r\right) \,+\, n^d-n^{\ell}\tbinom{n}{r+1}.
\end{align*}
It follows that $n^d -  \dim(\im T_{\ell,r}(-\tau)) \ge  \dim \left(\Sigma_\ell \cap I_r\right) \,+\, n^d-n^{\ell}\tbinom{n}{r+1}$; i.e.,
\begin{equation*}
n^{\ell}\tbinom{n}{r+1} \,-\, \dim(\im T_{\ell,r}(-\tau)) \;  \ge \; \dim \left(\Sigma_\ell \cap I_r\right).
\end{equation*}
Thus, the proof will be complete once we show that
\begin{equation*}
 \dim \left(\Sigma_\ell \cap I_r\right)   \; \ge \;  \dim W^{\ell+1,r+1}
 \end{equation*}
  (note that the right-hand side is the classical analogue of the left-hand side).
Since
\begin{equation*}   
\Sigma_\ell \cap I_r  \; \supseteq \; \sum_{i=1}^{\ell}\left(W_i\cap \bigcap _{j=\ell+1}^{d-1}W_j\right),
\end{equation*}
with equality when $\Lat(d)$ is distributive, we consider the expression on the right. The term inside the parentheses has classical dimension by \Cref{thm.dim.rel+.sp} and hence the right hand sum has generically large dimension by \Cref{pr:genker}. By \Cref{cor.sym} the sum on the right has classical dimension in the limit as $\tau\to 0$, so
\begin{equation*}
 \dim \left(\Sigma_\ell \cap I_r\right)   \; \ge \; \dim\sum_{i=1}^{\ell}\left(W_i\cap \bigcap _{j=\ell+1}^{d-1}W_j\right) \; \ge \; \dim W^{\ell+1,r+1}
\end{equation*}
on a dense set of $\tau$'s.  But $\dim \left(\Sigma_\ell \cap I_r\right)$ is generically small because $\dim(\ker F^L_{\ell+1}(-\tau)F^{R}_{r+1}(\tau))$ is, so $ \dim \left(\Sigma_\ell \cap I_r\right)  \ge  \dim  W^{\ell+1,r+1}$ for all $\tau$.
\end{proof}

\begin{lemma}\label{le.(d-1)tau}
  We have
  \begin{equation}
  \label{eq:(d-1)tau}
    \dim(\ker H_\tau((d-1)\tau)) \; \ge \; \dim W^{\ell,r+2}+n^{\ell}\tbinom{n}{r+1}-n^{\ell-1}\tbinom{n}{r+2}-n\dim W^{\ell,r+1}.
  \end{equation}
\end{lemma}
\begin{proof}
As observed in the proof of \cref{le.-tau},
\begin{equation*}
\dim(\ker H_\tau(z))+\dim(\im T_{\ell,r}(z))\;=\;n^{\ell}\tbinom{n}{r+1} \,-\, n\dim W^{\ell,r+1}.
\end{equation*}
Thus, to prove the lemma we must show that 
\begin{equation*}
n^{\ell}\tbinom{n}{r+1} \,-\, n\dim W^{\ell,r+1}  \, - \, \dim(\im T_{\ell,r}((d-1)\tau)) \; \ge \;  \text{the right-hand side of \cref{eq:(d-1)tau}}
\end{equation*}
or, equivalently, that 
\begin{equation}
\label{eq:new.inequality}
\dim (\ker T_{\ell,r}((d-1)\tau)) \; \ge \;  n^d \,-\, n^{\ell-1}\tbinom{n}{r+2} \,+\, \dim W^{\ell,r+2}.
\end{equation}

We now consider $\dim\ker T_{\ell,r}((d-1)\tau)$.
We have
\begin{align*}
& T_{\ell,r}((d-1)\tau) \\
&  \; =\; M_{r+1,\ell}(-(r+1)\tau; \tau^r; (-\tau)^{\ell-2},(d-1)\tau)  \cdot  \big( T_{\ell}((-\tau)^{\ell-2},(d-1)\tau) \otimes  T_{r+1}(\tau^r)\big) 
 \qquad \text{by \cref{MTT}}
 \\
& \; =\; B \cdot S^{\rm rev}_{d \to \ell}(\tau,\tau^{r})  \cdot  \big(T_{\ell}((-\tau)^{\ell-2},(d-1)\tau)\otimes  T_{r+1}(\tau^r)\big)   
 \qquad \text{by \cref{eq.M.decr}} 
\end{align*}
where
\begin{align*}
B & \;=\;    S^{\mathrm{rev}}_{r+2\to 1}(-(r+1)\tau,\tau^r)S^{\mathrm{rev}}_{r+3\to 2}(-(r+2)\tau,\tau^r)\cdots 
S^{\mathrm{rev}}_{d-1\to \ell-1}(-(d-2)\tau,\tau^r)
\\
& \; = \;  S^{\rm rev}_{r+1 \to d-1}(-(r+1)\tau, (-\tau)^{\ell-2}) S^{\rm rev}_{r \to d-2}(-r\tau, (-\tau)^{\ell-2})  \cdots 
S^{\rm rev}_{1 \to \ell}((-\tau)^{\ell-1}), 
\end{align*}
 the equality being essentially the same as \cref{eq.M.incr} = \cref{eq.M.decr}.
We write $B =  C \cdot S^{\rm rev}_{1 \to \ell}((-\tau)^{\ell-1})$. 

By \cref{le.f-id}, $T_{\ell}((-\tau)^{\ell-2},z) =  T^L_{\ell-1}((-\tau)^{\ell-2}) \cdot S_{\ell\to 1}((-\tau)^{\ell-2},z)$. 
Hence $\dim\ker T_{\ell,r}((d-1)\tau)$ is $\ge $ the dimension of the kernel of
  \begin{align*} 
& \phantom{\mbox{}\;=\;\mbox{}}  
B \cdot S^{\rm rev}_{d \to \ell}(\tau,\tau^{r})  \cdot  \big(T_{\ell-1}((-\tau)^{\ell-2})\otimes I \otimes  T_{r+1}(\tau^r)\big)
\\
&  \;=\; B \cdot \big(I^{\otimes (\ell-1)} \otimes S^{\rm rev}_{r+2 \to 1}(\tau^{r+1}) \big) \cdot \big(T_{\ell-1}((-\tau)^{\ell-2})\otimes I \otimes  T_{r+1}(\tau^r)\big)
\\
&  \;=\; B \cdot \big(T_{\ell-1}((-\tau)^{\ell-2})\otimes  S^{\rm rev}_{r+2 \to 1}(\tau^{r+1} )T^R_{r+1}(\tau^r)\big)
\\
&  \;=\; B \cdot \big(T_{\ell-1}((-\tau)^{\ell-2})\otimes  T_{r+2}(\tau^{r+1})\big)
\qquad \text{by \cref{le.f-id}}
\\
&  \;=\; B \cdot T^L_{\ell-1}((-\tau)^{\ell-2}) \cdot T^R_{r+2}(\tau^{r+1})
\\
 &  \;=\; C \cdot S^{\rm rev}_{1 \to \ell}((-\tau)^{\ell-1}) \cdot T^L_{\ell-1}((-\tau)^{\ell-2})  \cdot T^R_{r+2}(\tau^{r+1})
\\
&  \;=\; C \cdot T^L_\ell((-\tau)^{\ell-1}) \cdot T^R_{r+2}(\tau^{r+1})\qquad \text{by \cref{le.f-id}}.
  \end{align*}
In particular,
\begin{align*}
  \dim(\ker T_{\ell,r}((d-1)\tau)) &  \; \ge \; \dim \! \big( \ker F^L_{\ell}(-\tau)F^{R}_{r+2}(\tau) \big) 
 \\
 & \;=\;   \dim \! \big( \ker F^L_{\ell}(-\tau) \cap \im F^{R}_{r+2}(\tau) \big) \, + \, \dim(\ker F^{R}_{r+2}(\tau))
  \\
 & \;=\;   \dim ( \Sigma_{\ell-1}  \cap I_{r+1}) \, + \, \dim(\ker F^{R}_{r+2}(\tau))\qquad \text{by \cref{thm.dim.rel.sp,thm.dim.rel+.sp}}
 \\
  & \; \ge \;   \dim W^{\ell,r+2} \, + \, \dim(\ker F^{R}_{r+2}(\tau))
  \qquad \text{as in the proof of \Cref{le.-tau}}
 \\
&\; =\; \dim W^{\ell,r+2}+n^d-n^{\ell-1}\tbinom{n}{r+2}
\qquad \text{by \cref{thm.dim.rel+.sp}.}
  \end{align*}
Thus, the inequality in \cref{eq:new.inequality} holds and the proof is complete.
\end{proof}

\begin{lemma}
\label{lem.H.nullity}
For all  $\zeta \in \frac{1}{n}\Lambda$,  $H_\tau(z+\zeta) $ and $H_\tau(z)$ have the same nullity. 
\end{lemma}
\begin{proof} 
The proof resembles that of \cref{lem.G.nullity}; actually, it's a little simpler because of some cancellation. 
As in that proof, we write $\zeta=\frac{a}{n}+\frac{b}{n}\eta$ with $a,b \in \bZ$ and set $C:=T^bS^{ka}$.
 
By \cref{cor.R.transl.props},  $R_\tau(z+\zeta) = f(z,\zeta,\tau) C_2^{-1} R_\tau(z) C_1$.

By definition, $H_\tau(z)$ is the restriction of $S^{\rm rev}_{d \to 1}(z,\ldots)=R(z+*)_{d-1,d} \ldots R(z+*)_{12}$, 
where the $*$'s represent some terms that play no role
in the calculations below, to the image of $T^R_{\ell-1,r}(-\tau)$. 
Hence $H_\tau(z+\zeta)$  is the restriction of 
\begin{align*}
& R(z+\zeta+*)_{d-1,d}R(z+\zeta+*)_{d-2,d-1}\cdots R(z+\zeta+*)_{12}
\\
& \;=\; g(z,\zeta,\tau) C_d^{-1} R(z +*)_{d-1,d}C_{d-1} \cdot C_{d-1}^{-1}R(z +*)_{d-2,d-1}C_{d-2}\cdots C_2^{-1}R(z+*)_{12}C_1
\\
& \;=\; g(z,\zeta,\tau)C_d^{-1} \, R(z+*)_{d-1,d}R(z+*)_{d-2,d-1}\cdots R(z+*)_{12} C_1
\\
& \;=\; g(z,\zeta,\tau) C_d^{-1} \, S^{\rm rev}_{d \to 1}(z,\ldots )\, C_1
\end{align*}
to  the image of  $T^R_{\ell-1,r}(-\tau)$, where the function $g(z,\zeta,\tau)$ is a product of various $f(\cdot,\cdot,\cdot)$'s and therefore 
never vanishes.

Thus, the nullity of $H_\tau(z+\zeta)$ is the same as the nullity of the restriction of 
$S^{\rm rev}_{1 \to d}(z,\ldots) C_1$ to  $\im T^R_{\ell-1,r}(-\tau)=V \otimes \im T_{\ell-1,r}(-\tau)$. 
But $C_1$ is an automorphism so  the nullity of $H_\tau(z+\zeta)$ equals the nullity of the restriction of 
$S^{\rm rev}_{1 \to d}(z,\ldots)$ to $V \otimes \im T_{\ell-1,r}(-\tau)$; i.e., it equals the nullity of $H_\tau(z)$.
\end{proof}

\Cref{le.hzero,le.-tau,le.(d-1)tau}, which also hold when $\zeta \in \frac{1}{n}\Lambda$ is added to the input of $H_{\tau}(z)$,
tell us that
\begin{equation}\label{eq:sumh}
  \dim(\ker H_\tau(p))\geq
  \begin{cases}
        \dim W^{\ell,r+2}+n^{\ell}\tbinom{n}{r+1}-n^{\ell-1}\tbinom{n}{r+2}-n\dim W^{\ell,r+1} & p \in (d-1)\tau+\frac{1}{n}\Lambda,
         \\
    n^{\ell}\tbinom{n}{r+1}-n\dim W^{\ell,r+1} & p\in m\tau+\frac 1n \Lambda \; \text{and}
    \\
    &  \,\, m\in \{1,\cdots,d-3\},
    \\
    \dim W^{\ell+1,r+1}-n\dim W^{\ell,r+1} & p\in-\tau+\frac{1}{n}\Lambda,
    \\
    0 & \text{otherwise.}
  \end{cases}
\end{equation}

Further taking into account the multiplicity-two result in \Cref{le.ker.mult.two} for $p=(\ell-1)\tau$, and for
$p=(\ell-1)\tau+\zeta$ with $\zeta \in \frac{1}{n}\Lambda$, the following inequalities for the singularity partitions of $H_{\tau}$ (\Cref{def.sing-part}) hold:

\begin{equation}
  \label{eq:6}
  |\sigma_p(H_\tau)|\geq
  \begin{cases}
    \dim W^{\ell,r+2}+n^{\ell}\tbinom{n}{r+1}-n^{\ell-1}\tbinom{n}{r+2}-n\dim W^{\ell,r+1} & p \in (d-1)\tau+\frac{1}{n}\Lambda, \\
    n^{\ell}\tbinom{n}{r+1}-n\dim W^{\ell,r+1} & p \in m\tau+\frac 1n \Lambda \; \text{and} 
    \\ 
    & \,\, m\in \{1,\cdots,d-3\} - \{\ell-1\},   \\
    n^{\ell}\tbinom{n}{r+1}-n\dim W^{\ell,r+1}+\tbinom{n+\ell-1}{\ell}\tbinom{n}{r+1} & p \in (\ell-1)\tau+\frac 1n \Lambda,  \\
    \dim W^{\ell+1,r+1}-n\dim W^{\ell,r+1} & p\in-\tau+\frac{1}{n}\Lambda,   \\
    0 & \text{otherwise.}
  \end{cases}  
\end{equation}

\begin{proposition}\label{pr.h-part}
All inequalities in \Cref{eq:6} are equalities. 
Furthermore,  if $p\notin (\ell-1)\tau+\frac 1n\Lambda$, all the inequalities in \cref{eq:sumh} are equalities.
\end{proposition}
\begin{proof}
Let $P$ be a fundamental parallelogram with respect to $\Lambda$.

We will show that
\begin{equation}\label{eq:ww}
\dim W^{\ell,r+2} + \dim W^{\ell+1,r+1} \; =\;  n^{\ell}\tbinom{n}{r+1} + n^{\ell-1}\tbinom{n}{r+2}-\tbinom{n+\ell-1}{\ell}\tbinom{n}{r+1}
\end{equation}
in the last paragraph of this proof. For now, assume that  \cref{eq:ww} is true. 
With that assumption,
\begin{align*}
\numberthis \label{eq:almost.there}
\sum_{p \in P}   |\sigma_p(H_\tau)|   &  \; \ge \; (d-1) n^2 \left(n^{\ell}\tbinom{n}{r+1}\,-\, n\dim W^{\ell,r+1}\right)
\\
& \;=\; \text{the number of zeros $\det H_\tau(z)$ has in $P$} \qquad \text{by \Cref{le.hdet}} 
\\
&  \; =  \;   \sum_{p \in P}  \mult_p(\det H_\tau(z))
\\
&  \; \ge  \;   \sum_{p \in P} |\sigma_p(H_\tau)|      \qquad \text{by \Cref{pr.hol-op-bis}}.
\end{align*}
Hence the two inequalities in \cref{eq:almost.there} are equalities. It follows that the  inequalities in \Cref{eq:6} are equalities, as claimed.

The second sentence in the proposition follows since the only points $p$ where we have to consider zeros of multiplicity
$\ge 2$ are those $p\in (\ell-1)\tau+\frac 1n\Lambda$.

We will now prove \cref{eq:ww}.
Recall the notation in \cref{sect.cl.notn}, and consider
  \begin{equation*}
W^{\ell+1,r+1} \;=\;     \sum_{i=1}^{\ell}\Lambda_{i,i+1}\cap \bigcap_{i=\ell+1}^{d-1}\Lambda_{i,i+1} \;=\;  (X+Y)\cap Z
  \end{equation*}
 where 
  \begin{equation*}
    X=\sum_{i=1}^{\ell-1}\Lambda_{i,i+1},\quad Y= \Lambda_{\ell,\ell+1},\quad Z=\bigcap_{i=\ell+1}^{d-1}\Lambda_{i,i+1}. 
  \end{equation*}
The lattice generated by the $  \Lambda_{i,i+1}$'s is distributive so
$(X+Y)\cap Z=X\cap Z + Y\cap Z$. Hence 
  \begin{equation}\label{eq:in-ex}
    \dim W^{\ell+1,r+1} = \dim(X\cap Z) + \dim(Y\cap Z) - \dim(X\cap Y\cap Z).
  \end{equation}
But $X\cap Y\cap Z=W^{\ell,r+2}$, $ \dim(X\cap Z) =\left(n^{\ell}-\tbinom{n+\ell-1}{\ell}\right)\tbinom{n}{r+1}$, and 
$\dim(Y\cap Z) =n^{\ell-1}\tbinom{n}{r+2}$; substituting  these into \Cref{eq:in-ex} yields \cref{eq:ww}.
\end{proof}

\begin{proposition}\label{pr.dist-right-dim}
$\Sigma_\ell \cap I_r$  has classical dimension. 
\end{proposition}
\begin{proof}
We have
\begin{align*}
&\dim(\im T_{\ell,r}(-\tau))\\
& \;=\; n^\ell  \tbinom{n}{r+1} \,-\, n\dim W^{\ell,r+1} \, -\, \dim (\ker H_\tau(-\tau))
 \qquad \text{by the proof of \cref{le.-tau}}
\\
& \;=\; n^\ell  \tbinom{n}{r+1} \,-\, n\dim W^{\ell,r+1} \, -\,      \dim W^{\ell+1,r+1}  \,+\,     n\dim W^{\ell,r+1} 
 \qquad \text{by \cref{pr.h-part}}
\\
& \;=\; n^\ell  \tbinom{n}{r+1} \, -\,      \dim W^{\ell+1,r+1} .   
\end{align*}
The result now follows from \cref{lem.new.goal}.
\end{proof}

\begin{theorem}
\label{thm.Koszul}
For all $\tau\in(\bC-\bigcup_{m\geq 1}\frac{1}{m}\Lambda)\cup\frac{1}{n}\Lambda$, $Q_{n,k}(E,\tau)$ is a Koszul algebra.
\end{theorem}
\begin{proof}
We observed this in the proof of \cref{thm.H.series.dual} for $\tau\in\frac{1}{n}\Lambda$. The result for $\tau\in\bC-\bigcup_{m\geq 1}\frac{1}{m}\Lambda$ follows from the arguments in this section.
\end{proof}

\section{Artin-Schelter regularity of $Q_{n,k}(E,\tau)$ }\label{se.reg}

In this section we show, for all but countably many $\tau$, that $Q_{n,k}(E,\tau)$ is an Artin-Schelter regular algebra 
in the sense of \cite{AS87}. Suppose $\tau\in(\bC-\bigcup_{m\geq 1}\frac{1}{m}\Lambda)\cup\frac{1}{n}\Lambda$. Since $Q:=Q_{n,k}(E,\tau)$ has finite Gelfand-Kirillov dimension, it is Artin-Schelter regular of 
dimension $n$ if the global dimension of $Q$ is $n$ and
\begin{equation*}
  \Ext_Q^i(\bC,Q) \; \cong \; 
  \begin{cases}
  \bC & \text{if $i=n$,}
  \\
  0 & \text{if $i\ne n$.}
  \end{cases}
\end{equation*}

The next result provides a partial confirmation of Artin-Schelter regularity. 

\begin{theorem}\label{th.fgl}
For all $\tau\in(\bC-\bigcup_{m\geq 1}\frac{1}{m}\Lambda)\cup\frac{1}{n}\Lambda$, the global dimension of $Q_{n,k}(E,\tau)$ is $n$.  
\end{theorem}
\begin{proof}
Let $A$ be a connected graded algebra over a field $\Bbbk$.
It is well known that the global dimension of $A$ is the largest integer $d$
such that $\Ext^d_A(\Bbbk,\Bbbk) \ne 0$ (see, e.g., \cite[Prop.~3.18]{RR18} and \cite[Thm.~11]{Eil56}). If $A$ is a Koszul algebra, then $\Ext^d_A(\Bbbk,\Bbbk) \cong A_{d}^!$ so
its global dimension is the largest integer $d$ such that $A^!_d \ne 0$. 
Since $Q_{n,k}(E,\tau)^!$ is a Koszul algebra with Hilbert series  $(1+t)^n$, 
the global dimension of $Q_{n,k}(E,\tau)$ is $n$.
\end{proof}

We will use the following result with $A=Q_{n,k}(E,\tau)$; see, e.g., \cite[Prop.~5.10]{smith-mexico} or \cite[Thm.~1.9]{LPWZ}. 

\begin{theorem}
\label{thm.as-reg.frob}
Let $A$ be a connected graded Koszul algebra over a field $\Bbbk$. If its global dimension is finite,
then $A$ is Artin-Schelter regular if and only if $A^!$ is Frobenius.  
\end{theorem}

A finite-dimensional $\Bbbk$-algebra $S$ is {\sf Frobenius} if it is isomorphic as a left $S$-module to its dual 
$S^*:=\Hom_{\Bbbk}(S,\Bbbk)$ equipped with the left module structure resulting from right multiplication.
By \cite[Lem.~3.2]{smith-mexico}, $S:=S_{n,k}(E,\tau)$ is Frobenius  if and only if the 
multiplication maps
\begin{equation*}
S_i \times S_{n-i} \; \longrightarrow \; S_n\;\cong\;\Bbbk
\end{equation*}
are non-degenerate bilinear forms for all $i=0,\ldots,n$;
this happens if and only if $S_{n}$ is the socle of $S$ as a left (or right) $S$-module.

\begin{proposition}
\label{lem.S.frob}
For each $(n,k,E)$, $S_{n,k}(E,\tau)$ is a Frobenius algebra for all but finitely many $\tau \in E$.
\end{proposition}
\begin{proof}
In this proof, $\tau$ denotes a complex number.

Let $\cF=\bigcup_{m=1}^{n+1} E[mn]$. 
Assume $\tau+\Lambda \in E-\cF$; 
 i.e., $\tau \in  \bC-\bigcup_{m=1}^{n+1}\frac{1}{mn}\Lambda$.

Let $S:=S_{n,k}(E,\tau)$.
By \cref{thm.dim.rel+.sp}, $S_n \cong \bC$ and $S_{n+1}=0$.
 Since $S$ is generated in degree one, $S_{d}=0$ for all $d\geq n+1$. In particular, $S$ is finite-dimensional.

 By \cref{prop.ell.shuffle.prod}\cref{item.prop.ell.shuffle.prod.dual}, $S$ is Frobenius if and only if the bilinear maps
\begin{equation*}
\begin{tikzcd}
  \im F_{i}(\tau)\times\im F_{n-i}(\tau)\ar[r] & \im F_{i}(\tau)\otimes\im F_{n-i}(\tau)\ar[rr,"M_{n-i,i}(\tau)"] & & \im F_{n}(\tau) \;\cong\; \bC
\end{tikzcd}
\end{equation*}
are non-degenerate for all $i=0,\ldots,n$. 
Since $M_{n-i,i}(\tau) \! \cdot \!(F_{i}(\tau)\otimes F_{n-i}(\tau)) = F_{n}(\tau)$ (\cref{le.mtt}), this happens if and only if,
for all $i=0,\ldots,n$, the rank of the bilinear map
\begin{equation}\label{eq.bilin.Fn}
\begin{tikzcd}
  V^{\otimes i}\times V^{\otimes (n-i)}\ar[r] & V^{\otimes i}\otimes V^{\otimes (n-i)}\ar[rr,"F_{n}(\tau)"] && \im F_{n}(\tau)\;\cong\;\bC
\end{tikzcd}
\end{equation}
equals $\dim(\im F_{i}(\tau)) = \dim(\im F_{n-i}(\tau))=\binom{n}{i}$. Clearly, the rank can be no larger than this.

Fix $x\in V^{\otimes n}$ such that $F_{n}(0)(x)\neq 0$. 
There is a Zariski-open dense subset $U \subseteq E$ such that $F_{n}(\tau)(x)\neq 0$ for all $\tau+\Lambda\in U$.
Assume $\tau + \Lambda \in U$. 

Let $\{v_{j}\}$ and $\{w_{k}\}$ be bases for $V^{\otimes i}$ and $V^{\otimes (n-i)}$, respectively. The rank of the bilinear map in 
\cref{eq.bilin.Fn} is the rank of the matrix $(c_{jk}(\tau))_{j,k}$ where
\begin{equation*}
  F_{n}(\tau)(v_{j}\otimes w_{k})\;=\;c_{jk}(\tau)\cdot F_{n}(\tau)(x).
\end{equation*}
Since $F_{n}(\tau)$ is a theta operator, $c_{jk}(\tau)$ is an elliptic function, and so are all minors of $(c_{jk}(\tau))_{j,k}$. So the rank of $(c_{jk}(\tau))_{j,k}$ is generically large. Since the rank attains the maximal value $\binom{n}{i}$ at $\tau=0$, it equals $\binom{n}{i}$ for all $\tau+\Lambda$ belonging to a 
dense open subset $U_i \subseteq U$. Since $U_0 \cap \ldots \cap U_n=E-\cF'$ for some finite subset $\cF' \subseteq E$, 
we see that $S_{n,k}(E,\tau)$ is Frobenius for all $\tau+\Lambda \in E-(\cF \cup \cF')$.
\end{proof}

Suppose $\tau\in\bC-\bigcup_{m= 1}^{n+1}\frac{1}{mn}\Lambda$. Since the space of degree-$d$ relations for $S_{n,k}(E,\tau)$ is
\begin{equation*}
\ker F_{d}(\tau)  \; = \;    \sum_{s+t+2=d}V^{\otimes s}\otimes\im R_{\tau}(-\tau)\otimes V^{\otimes t}
\end{equation*}
(\cref{thm.dim.rel+.sp}), the Frobenius property for $S_{n,k}(E,\tau)$ can be reduced to a statement about the kernel of the 
operators $F_{d}(\tau)$ (defined in \cref{defn.F_d}) and $G^+_{\tau}(\tau)$ (defined in the proof of \Cref{thm.dim.rel+.sp}).
The algebra $S_{n,k}(E,\tau)$ is Frobenius if and only if the following statements
(and their left-right symmetric versions which we do not state) are true for all $d=0,\ldots,n-1$:
\begin{quote}
 the largest subspace $W\subseteq V^{\otimes d}$ such that $V\otimes W \subseteq \ker F_{d+1}(\tau)$ is $W=\ker F_{d}(\tau)$
\end{quote}
or, equivalently, 
\begin{quote}
if  $V\otimes \{w\}$ is in the kernel of $G^+_{\tau}(\tau):V\otimes \im F_{d}(\tau) \, \longrightarrow \, \im F_{d}(\tau)\otimes V$,
 then $w=0$.
\end{quote}

\begin{theorem}\label{th.reg}
Let $\tau \in \bC$ and fix $(n,k,E)$.
\begin{enumerate}
  \item\label{item.reg.cnt}
  $Q_{n,k}(E,\tau)$ is Artin-Schelter regular of dimension $n$ for all but countably many $\tau$.
  \item\label{item.reg.fin}
 If $Q_{n,k}(E,\tau)$ is a Koszul algebra for all $\tau$, 
 then it is Artin-Schelter regular of dimension $n$ for all but finitely many $\tau + \Lambda$.
\end{enumerate}
\end{theorem}

\begin{proof} 
\cref{item.reg.cnt}
The algebra $Q_{n,k}(E,\tau)$ is Artin-Schelter regular of dimension $n$ if the following three statements are true: (a) it is a Koszul algebra; 
(b) the Hilbert series of $Q_{n,k}(E,\tau)^!$ is $(1+t)^n$; (c) $Q_{n,k}(E,\tau)^!$ is a Frobenius algebra. 
By \cref{thm.H.series.dual} and \cref{lem.S.frob}, there is a finite set $\cF \subseteq E$ such that (b) and (c) are true. 
By \cref{thm.Koszul}, (a) is true for all but countably many cosets $\tau+\Lambda$ and hence for all but countably many $\tau$.
Thus (a), (b), and (c) are simultaneously true for all but countably many $\tau$.  

\cref{item.reg.fin} 
The argument follows that in \cref{item.reg.cnt}. The only difference is that we are now assuming that (a) is true for all $\tau$.
Thus, since (b) and (c) are true for all but finitely many $\tau+\Lambda$, $Q_{n,k}(E,\tau)$
 is Artin-Schelter regular of dimension $n$ for all but finitely many $\tau + \Lambda$. 
\end{proof}


\printindex
\clearpage

\bibliography{biblio4}

\def\cprime{$'$}
\providecommand{\bysame}{\leavevmode\hbox to3em{\hrulefill}\thinspace}
\providecommand{\MR}{\relax\ifhmode\unskip\space\fi MR }
\providecommand{\MRhref}[2]{%
  \href{http://www.ams.org/mathscinet-getitem?mr=#1}{#2}
}
\providecommand{\href}[2]{#2}
\begin{thebibliography}{ATVdB91}

\bibitem[AS87]{AS87}
M.~Artin and W.~F. Schelter, \emph{Graded algebras of global dimension {$3$}},
  Adv. in Math. \textbf{66} (1987), no.~2, 171--216. \MR{917738 (88k:16003)}

\bibitem[ATVdB90]{ATV1}
M.~Artin, J.~Tate, and M.~Van~den Bergh, \emph{Some algebras associated to
  automorphisms of elliptic curves}, The {G}rothendieck {F}estschrift, {V}ol.\
  {I}, Progr. Math., vol.~86, Birkh\"auser Boston, Boston, MA, 1990,
  pp.~33--85. \MR{1086882 (92e:14002)}

\bibitem[ATVdB91]{ATV2}
\bysame, \emph{Modules over regular algebras of dimension {$3$}}, Invent. Math.
  \textbf{106} (1991), no.~2, 335--388. \MR{1128218 (93e:16055)}

\bibitem[Bax72]{bax1}
R.~J. Baxter, \emph{Partition function of the eight-vertex lattice model}, Ann.
  Physics \textbf{70} (1972), 193--228. \MR{0290733}

\bibitem[Bax82]{bax3}
\bysame, \emph{The inversion relation method for some two-dimensional exactly
  solved models in lattice statistics}, J. Statist. Phys. \textbf{28} (1982),
  no.~1, 1--41. \MR{664123}

\bibitem[Bax89]{bax2}
R.~J. Baxter, \emph{Exactly solved models in statistical mechanics}, Academic
  Press, Inc. [Harcourt Brace Jovanovich, Publishers], London, 1989, Reprint of
  the 1982 original. \MR{998375}

\bibitem[Bel80]{bel}
A.~A. Belavin, \emph{Discrete groups and integrability of quantum systems},
  Funktsional. Anal. i Prilozhen. \textbf{14} (1980), no.~4, 18--26, 95.
  \MR{595725}

\bibitem[CC81]{ch2}
D.~V. Chudnovsky and G.~V. Chudnovsky, \emph{Completely {$X$}-symmetric
  {$S$}-matrices corresponding to theta functions}, Phys. Lett. A \textbf{81}
  (1981), no.~2-3, 105--110. \MR{597095}

\bibitem[Che82]{cher}
I.~V. Cherednik, \emph{On the properties of factorized {$S$}\ matrices in
  elliptic functions}, Yadernaya Fiz. \textbf{36} (1982), no.~2, 549--557
  (1983). \MR{700961}

\bibitem[CKS18]{CKS1}
A.~{Chirvasitu}, R.~{Kanda}, and S.~P. {Smith}, \emph{{Feigin and Odesskii's
  elliptic algebras}}, arXiv:1812.09550v3.

\bibitem[CKS19a]{CKS3}
\bysame, \emph{{Maps from Feigin and Odesskii's elliptic algebras to twisted
  homogeneous coordinate rings}}, arXiv:1908.06525v2.

\bibitem[CKS19b]{CKS2}
\bysame, \emph{{The characteristic variety for Feigin and Odesskii's elliptic
  algebras}}, arXiv:1903.11798v4.

\bibitem[{De }16]{dL}
K.~{De Laet}, \emph{{On the center of 3-dimensional and 4-dimensional Sklyanin
  algebras}}, arXiv:1612.06158v1.

\bibitem[EH16]{3264}
D.~Eisenbud and J.~Harris, \emph{3264 and all that---a second course in
  algebraic geometry}, Cambridge University Press, Cambridge, 2016.
  \MR{3617981}

\bibitem[Eil56]{Eil56}
S.~Eilenberg, \emph{Homological dimension and syzygies}, Ann. of Math. (2)
  \textbf{64} (1956), 328--336. \MR{82489}

\bibitem[FO89]{FO-Kiev}
B.~L. Feigin and A.~V. Odesskii, \emph{Sklyanin algebras associated with an
  elliptic curve}, Preprint deposited with Institute of Theoretical Physics of
  the Academy of Sciences of the Ukrainian SSR (1989), 33 pages.

\bibitem[LPWZ07]{LPWZ}
D.-M. Lu, J.~H. Palmieri, Q.-S. Wu, and J.~J. Zhang, \emph{Regular algebras of
  dimension 4 and their {$A_\infty$}-{E}xt-algebras}, Duke Math. J.
  \textbf{137} (2007), no.~3, 537--584. \MR{2309153}

\bibitem[Mum07]{Mum07}
D.~Mumford, \emph{Tata lectures on theta. {I}}, Modern Birkh\"auser Classics,
  Birkh\"auser Boston, Inc., Boston, MA, 2007, With the collaboration of C.
  Musili, M. Nori, E. Previato and M. Stillman, Reprint of the 1983 edition.
  \MR{2352717}

\bibitem[Ode02]{Od-survey}
A.~V. Odesskii, \emph{Elliptic algebras}, Uspekhi Mat. Nauk \textbf{57} (2002),
  no.~6(348), 87--122. \MR{1991863}

\bibitem[OF89]{FO89}
A.~V. Odesskii and B.~L. Feigin, \emph{Sklyanin elliptic algebras},
  Funktsional. Anal. i Prilozhen. \textbf{23} (1989), no.~3, 45--54, 96.
  \MR{1026987 (91e:16037)}

\bibitem[PP05]{PP05}
A.~Polishchuk and L.~Positselski, \emph{Quadratic algebras}, University Lecture
  Series, vol.~37, American Mathematical Society, Providence, RI, 2005.
  \MR{2177131 (2006f:16043)}

\bibitem[RR18]{RR18}
M.~L. {Reyes} and D.~{Rogalski}, \emph{{A twisted Calabi-Yau toolkit}},
  arXiv:1807.10249v1.

\bibitem[RT86]{ric-tra}
M.~P. Richey and C.~A. Tracy, \emph{{$Z_n$} {B}axter model: symmetries and the
  {B}elavin parametrization}, J. Statist. Phys. \textbf{42} (1986), no.~3-4,
  311--348. \MR{833020}

\bibitem[Sal99]{saltman}
D.~J. Saltman, \emph{Lectures on division algebras}, CBMS Regional Conference
  Series in Mathematics, vol.~94, Published by American Mathematical Society,
  Providence, RI; on behalf of Conference Board of the Mathematical Sciences,
  Washington, DC, 1999. \MR{1692654}

\bibitem[Skl82]{Skl82}
E.~K. Sklyanin, \emph{Some algebraic structures connected with the
  {Y}ang-{B}axter equation}, Funktsional. Anal. i Prilozhen. \textbf{16}
  (1982), no.~4, 27--34, 96. \MR{684124 (84c:82004)}

\bibitem[Smi96]{smith-mexico}
S.~P. Smith, \emph{Some finite-dimensional algebras related to elliptic
  curves}, Representation theory of algebras and related topics ({M}exico
  {C}ity, 1994), CMS Conf. Proc., vol.~19, Amer. Math. Soc., Providence, RI,
  1996, pp.~315--348. \MR{1388568}

\bibitem[SS92]{SS92}
S.~P. Smith and J.~T. Stafford, \emph{Regularity of the four-dimensional
  {S}klyanin algebra}, Compositio Math. \textbf{83} (1992), no.~3, 259--289.
  \MR{1175941 (93h:16037)}

\bibitem[Tra85]{tra}
C.~A. Tracy, \emph{Embedded elliptic curves and the {Y}ang-{B}axter equations},
  Phys. D \textbf{16} (1985), no.~2, 203--220. \MR{796270}

\bibitem[TVdB96]{TvdB96}
J.~T. Tate and M.~Van~den Bergh, \emph{Homological properties of {S}klyanin
  algebras}, Invent. Math. \textbf{124} (1996), no.~1-3, 619--647. \MR{1369430
  (98c:16057)}

\bibitem[Zha96]{Z-twist}
J.~J. Zhang, \emph{Twisted graded algebras and equivalences of graded
  categories}, Proc. London Math. Soc. (3) \textbf{72} (1996), no.~2, 281--311.
  \MR{1367080 (96k:16078)}

\end{thebibliography}
\bibliographystyle{customamsalpha}

\end{document}